\documentclass[10pt]{article} 
\usepackage{amssymb, amsmath}
\usepackage{amsmath, amsfonts, amssymb}
\usepackage{stmaryrd}
\usepackage{color}

\input{liemacs10.sty} 
\def\lbr{\llbracket}
\def\rbr{\rrbracket}

\newcommand{\comp}{\mathop{{\rm comp}}\nolimits}
\newcommand{\bd}{\mathbf{d}}

\renewcommand{\phi}{\varphi} 

\newcommand{\sH}{{\sf H}}

\usepackage[pagebackref, colorlinks=true,linkcolor=blue,citecolor=red]{hyperref}
\usepackage{hyperref}
\usepackage{pdfpages}
\usepackage{pgfplots}

\begin{document}

\title{Ground state representations of topological groups} 
\author{Karl-Hermann Neeb, Francesco G. Russo} 

\maketitle

\begin{abstract} 
Let $\alpha \: \R \to \Aut(G)$ define a continuous $\R$-action on 
the topological group~$G$.
A unitary representation $(\pi^\flat,\cH)$ of the extended group 
$G^\flat := G \rtimes_\alpha \R$ is called a ground 
state representation if the unitary one-parameter group 
$\pi^\flat(e,t) = e^{itH}$ has a non-negative generator $H \geq 0$ and 
the subspace $\cH^0 := \ker H$ of ground states 
generates $\cH$ under $G$. 
In this paper we introduce the class of strict ground state 
representations, where $(\pi^\flat,\cH)$ and 
the representation of the subgroup $G^0 := \Fix(\alpha)$ 
on $\cH^0$ have the same commutant. The advantage of this 
concept is that it permits us to classify strict ground 
state representations in terms of the corresponding 
representations of $G^0$. This is particularly effective 
if the occurring representations of $G^0$ can be characterized 
intrinsically in terms of concrete positivity conditions. 

To find such conditions, it is natural to restrict to 
infinite dimensional Lie groups such as 
(1) Heisenberg groups (which exhibit examples of 
non-strict ground state representations); (2) Finite dimensional 
groups, where highest weight representations provide natural examples; 
(3) Compact groups, for which our approach provides a new perspective 
on the classification of unitary representations; 
(4) Direct limits of compact groups, as a class of examples for which 
strict ground state representations can be used to classify 
large classes of unitary representations.
\end{abstract}

\tableofcontents 
\section{Introduction}
\label{sec:1} 

Let $G$ be a (Hausdorff) topological group and 
$\alpha \: \R \to \Aut(G)$ be a homomorphism defining a continuous 
action of $\R$ on $G$ and let 
\[ G^0 :=\Fix(\alpha) := \{ g \in G \: (\forall t \in \R)\ 
\alpha_t(g) = g \}\] 
be the closed subgroup of fixed points. The 
semidirect product $G^\flat := G \rtimes_\alpha \R$ is a topological group 
with respect to the product topology, 
and a unitary representation $(\pi^\flat, \cH)$ of $G^\flat$ on a 
complex Hilbert space $\cH$ always has the form 
\begin{equation}
  \label{eq:pisharp}
\pi^\flat(g,t) = \pi(g) U_t,
\end{equation}
where $(\pi, \cH)$ is a unitary representation of $G$ and $(U_t)_{t \in \R}$ 
is a unitary one-parameter group on $\cH$ satisfying the covariance condition 
\begin{equation}
  \label{eq:commrel}
 U_t \pi(g) U_{-t} = \pi(\alpha_t(g)) \quad \mbox{ for } \quad 
t \in \R, g \in G.
\end{equation}
Writing $U_t= e^{itH}$ with a selfadjoint operator $H$ 
(Stone's Theorem, \cite[Thm.~13.38]{Ru73}), we call $\pi^\flat$, resp., the pair $(\pi,U)$ 
a {\it positive energy representation of $(G,\alpha)$}  if $H \geq 0$. 
If, in addition, for the {\it minimal energy space} $\cH^0 := \ker H$, 
the subset $\pi(G)\cH^0$ spans a dense subspace of $\cH$, 
we call $(\pi, \cH)$ a {\it ground state representation}.
\begin{footnote}{In \cite{JN21} the term ``ground state representation'' 
is used in a slightly more general context 
where the minimal eigenvalue of $H$ is not necessarily~$0$, but this is 
only a matter of terminology.}  
\end{footnote}
One expects that that ground state representations 
are determined by the representation $(\pi^0, \cH^0)$ of $G^0$ 
on the ``minimal energy space''~$\cH^0$, but this is 
in general not the case (Example~\ref{ex:4.7}). To make up for this defect,
we introduce the concept of a {\it strict ground state representation} 
(Definition~\ref{def:strict}), where 
we require, in addition, that the (injective) restriction map 
from the commutant $\pi(G)'$ of $\pi(G)$ to the commutant 
$\pi^0(G^0)'$ (a morphism of von Neumann algebras) is surjective. 
This class of representations of $G$ is completely determined by 
the corresponding representations of $G^0$ and 
we want to use it to classify strict ground 
state representations in terms of the corresponding 
representations of $G^0$. This is particularly effective 
if the occurring representations of $G^0$ can be characterized 
intrinsically in terms of concrete positivity conditions. 

To find such conditions, it is natural to restrict to the 
class of, possibly infinite dimensional, Lie groups. For these
groups the method of holomorphic induction which has been 
developed in \cite{Ne13} for Banach--Lie groups and 
extended in \cite[App.~C]{Ne14b} to certain Fr\'echet--Lie groups, 
can be used to construct strict ground state representations. 

We discuss four classes of groups: 
\begin{itemize}
\item[\rm(1)] Heisenberg groups $\Heis(V,\sigma)$, i.e., the canonical 
central extension of the additive groups of a 
symplectic vector space $(V,\sigma)$ by the circle group~$\T$. These groups 
provide in particular examples of non-strict ground state representations. 
As these group arise naturally in Quantum Field Theory 
in the context of the canonical commutation relations, 
we were able to use old results of M.~Weinless \cite{We69} 
to derive a structure theorem for ground state representations 
under some extra conditions 
on the $\R$-action. In this context our results complement 
the theory of semibounded representations 
of these group, developed in \cite{NZ13} and \cite{Ze14}. 
\item[\rm(2)] Finite dimensional Lie groups,  as a class of 
groups whose well-developed structure theory permits to 
understand the intricacies on the conditions to be imposed on 
the {$\R$-action} on~$G$. Here 
highest weight representations, and the more general class of 
semibounded representations provide natural examples of 
strict ground state representations (\cite{Ne00}). 
\item[\rm(3)] Connected compact groups, as a class of topological groups which is 
rather close to Lie groups. They are projective limits of connected Lie groups 
and we show that this approximation can be aligned with the $\R$-action 
on $G$. This permits us to show that {\it all} unitary representations 
are strict ground state representations, and this provides a 
novel perspective on the classification of 
unitary representations of connected compact Lie groups 
(Subsection~\ref{subsec:6.2}). If $\alpha$ is defined by 
$\alpha_t(g) = \exp(t\bd) g \exp(-t\bd)$ for a regular element 
$\bd \in \g$ (the Lie algebra of $G$), 
then $G^0 = T$ is a maximal torus and the approach to 
the classification in terms of ground state representations 
leads to the Cartan--Weyl Theorem on the classification 
in terms of their highest (lowest) weights (cf.~\cite[p.~209]{Wa72}). 

\item[\rm(4)] Direct limits of compact groups (\cite{Gl05}); as a class of Lie groups 
for which strict ground state representations can be studied systematically 
with the methods developed in this paper. We only briefly discuss 
some concrete examples to give an impression of how this can be done in 
principle.
\end{itemize}

In the representation theory of Lie groups,  ground states 
have classically been studied in the context of highest or lowest 
weight representations, which require a much finer structure 
theoretric context (cf.~\cite{KR87}). 
For this specific class of representations 
(and for the more general class of semibounded representations), 
methods similar to ours have also been 
used in the following contexts: 
\begin{itemize}
\item double extensions of Hilbert--Lie groups (\cite{MN16})  
\item twisted loop groups with values in Hilbert--Lie groups (\cite{MN17})
\item Hermitian Banach--Lie groups of compact type, i.e., 
automorphism groups of Banach hermitian symmetric spaces 
(\cite[\S 8]{Ne12})
\item the particular Schatten class 
groups $\U_1(\cH)$ and $\U_2(\cH)$ (\cite[App.~D]{Ne12}) 
\item for groups of local gauge transformations,  
positive energy and ground state representations have been treated from a 
similar perspective in \cite[\S\S 3, 9]{JN21}. 
\end{itemize}
The specific features appearing in these papers suggest that a 
general theory of ground state representations for general topological groups 
could be a useful tool to classify important classes of unitary representations 
of groups which are not locally compact. \\

We now describe the structure of this paper in some more detail. 
In Section~\ref{sec:2} we develop the basic concepts. 
First we explain how the concept of a 
minimal implementing group from the theory of operator algebraic 
dynamical systems 
translates into our context (Subsection~\ref{sec:minimal}). 
It provides the language to define 
minimal positive energy representations and ground state representations. 
In this context, the Borchers--Arveson Theorem implies that, 
for a ground state representation $(\pi, U,\cH)$ of 
$(G,\alpha)$, the one-parameter group $(U_t)_{t \in \R}$ 
is redundant in the sense that it is completely determined by the 
representation $\pi$ of $G$ and the assumption that the generating 
subspace $\cH^0$ is fixed pointwise by~$U$ 
(see also \cite{Ne14} for a formulation 
of the Borchers--Arveson Theorem in the context of topological groups). 
After discussing some elementary properties of ground state representations, 
we introduce the new concept of a 
strict ground state representation in Subsection~\ref{subsec:2.3}. 
The purpose of this concept is to classify 
ground state representations of $G$ in terms of representations 
of $G^0$ arising on some~$\cH^0$. 

In Section~\ref{sec:3} we 
develop for Lie groups methods to 
identify these representations of $G^0$ in intrinsic terms. 
 We formulate four conditions: 
(L1) $G$ is a Lie group (modelled over a locally convex space), 
(L2) $\alpha$ defines a smooth $\R$-action on $G$, 
(L3) the subgroup~$G^0$ of fixed points is a Lie group, and 
(L4) the Lie algebra $\g$ of $G$ is the direct sum of 
the Lie algebra $\g^0$ of $G^0$ and $\oline{D(\g)}$, where 
$D$ is the infinitesimal generator of the induced $\R$-action on~$\g$. 
Note that conditions (L1-3) are automatic if $G$ is a finite dimensional Lie group, 
but (L4) corresponds to $\ker(D^2) = \ker(D)$. 
If these conditions are satisfied and $p_0 \: \g \to \g^0$ is the projection 
with kernel $\oline{D(\g)}$, then the closed convex cone 
\[ C_\alpha \subeq \g^0, \] 
generated by elements of the form $p_0([Dx,x])$, $x \in \g$, 
turns out to play an important role. The main result in Section~\ref{sec:3} 
is Theorem~\ref{thm:2.18} which asserts that, for every ground 
state representation $(\pi,U,\cH)$, we have 
\[ C_\alpha \subeq C_{\pi^0} := \{ x \in \g^0 \: -i \partial \pi^0(x) \geq 0\},\] 
where $\partial \pi^0(x)$ is the infinitesimal generator of the 
unitary one-parameter group $(\pi^0(\exp tx))_{t \in \R}$. 
This is a necessary condition for a representation 
$(\pi^0,\cH^0)$ of $G^0$ to appear in a ground state representation 
of $G$. Unfortunately it is not sufficient in general 
(Remark~\ref{rem:countex}). 

As a consequence of the Borchers--Arveson Theorem, 
for abelian groups $G$ with a faithful ground state representation, 
the $\R$-action on $G$ is trivial. Therefore the simplest 
non-trivial examples arise from $2$-step nilpotent groups. 
We therefore discuss Heisenberg groups $G = \Heis(V,\sigma)$ 
in Section~\ref{sec:4}, where $\R$ acts on $G$ through a symplectic 
one-parameter group $\beta \: \R \to \Sp(V,\sigma)$. 
In this context condition (L4) turns into the weak splitting condition, 
which implies a tensor factorization of ground state representations 
that  can be derived from old results of M.~Weinless \cite{We69}. 
Actually an example where (L4) is violated led us to an example 
of a non-strict ground state representation (Example~\ref{ex:4.7}). 

Section~\ref{sec:5} introduces the powerful technique of 
holomorphic induction as a means to construct  ground state representations. 
It  deals with unitary representations that can be realized 
in Hilbert spaces of holomorphic sections of complex vector bundles 
whose fibers are Hilbert spaces. Unfortunately, 
holomorphic induction requires rather fine geometric assumptions on the Lie groups. 
But it provides effective tools to determine 
if a given representation can be realized in this setup, 
and then one can typically conclude that it is a 
strict ground state representation. If $G$ is a 
Banach--Lie group and $D \: \g \to \g$ is a bounded derivation for 
which $0$ is isolated in $\Spec(D)$ and the norm on $\g$ is 
$\alpha$-invariant ($D$ is elliptic), then 
Theorem~\ref{thm:6.2}, our main result in Section~\ref{sec:5} 
is one of our key tools to identify strict ground state representations. 
In the following sections it is applied to finite dimensional 
Lie groups, where the necessary requirements are verified more easily.

This is why we first turn to finite dimensional Lie groups in Section~\ref{sec:6}. 
Again, an old result, C.~Moore's Eigenvector Theorem (\cite{Mo80}), 
turns out to be quite helpful. We use it to see that, if an 
irreducible $G$-representation $(\pi,\cH)$ with discrete kernel 
has ground states for an 
$\R$-action corresponding to an inner derivation $D = \ad \bd$, 
then $\bd$ must be an elliptic element, i.e., $\oline{e^{\R \ad \bd}}$ is a 
torus group, or, equivalently, $\ad \bd$ is semisimple with purely 
imaginary spectrum. Our setup applies particularly well 
to compact connected Lie groups. In this context, 
for any $\alpha \: \R \to \Aut(G)$, we show in Theorem~\ref{thm:compcase} 
that every unitary representation $(U,\cH)$ of $G$ is a strict ground 
state representation, and that these correspond precisely to the 
representations $(\pi^0, \cH^0)$ of the connected subgroup $G^0$ 
satisfying $C_\alpha \subeq C_{\pi^0}$. If 
the derivation $D = \ad \bd$ is such that $\bd$ is a regular element, then 
$G^0$ is a maximal torus, and this result reproduces 
the Cartan--Weyl classification of irreducible representations, 
but it also works for any~$\bd$. 
In Section~\ref{sec:8} this result is extended to general 
connected compact topological groups (Theorem~\ref{thm:8.13}), which 
provides a novel global perspective on the classification for these  
groups. We conclude this paper with 
Section~\ref{sec:7} which is devoted to countable 
direct limits of finite dimensional Lie groups. These are always 
locally convex Lie groups by Gl\"ockner's Theorem (\cite{Gl05}). 
For direct limits of compact Lie groups,  
Theorem~\ref{thm:5.3} generalizes the characterization 
of representations $(\pi^0, \cH^0)$ of $G^0$ which extend to 
ground state representations in terms of the positivity 
condition $C_\alpha \subeq C_{\pi^0}$. It thus reduces the corresponding 
classification problem from $G$ to $G^0$. 
We discuss some examples where $G^0$ is abelian 
(a direct limit of tori), so that concrete classification 
results can be obtained with the Bochner Theorem for nuclear groups 
(\cite{Ba91}). 

We include four short appendices: 
In Appendix~\ref{app:arv} we recall Arveson's concept 
of spectral subspaces which is an important tool to formulate 
the splitting conditions in Section~\ref{sec:3}. 
Similarly Appendix~\ref{app:B} recalls the vector-valued 
version of the Gefland--Neimark--Segal (GNS) correspondence 
between operator-valued positive definite functions and unitary representations 
generated by Hilbert subspaces. 
Some facts on bosonic Fock spaces are collected in 
Appendix~\ref{sec:7.1} because they are needed in our discussion of Heisenberg 
groups in Section~\ref{sec:4}, Finally, Appendix~\ref{app:D} 
contains a key observation that we use in Section~\ref{sec:8} for 
the reduction from general compact groups to finite dimensional 
ones and also at some point in Section~\ref{sec:7}.

\subsection*{Notation} 
\begin{itemize}
\item $\cH$ denotes a complex Hilbert space, the scalar product 
is linear in the second argument, 
$B(\cH)$ denotes the algebra of bounded operators on $\cH$,  
and $\U(\cH)$ the unitary group. 
We call a subset $E \subeq \cH$ total if $\lbr E \rbr := \oline{\spann E} = \cH$. \item 
For a set $\cS \subeq B(\cH)$ of bounded operators, its {\it commutant}
is denoted 
\[ \cS' := \{ A \in B(\cH) \:(\forall B \in \cS)\, AB = BA\}.\] 
A $*$-subalgebra $\cM \subeq B(\cH)$ is called a {\it von Neumann algebra} 
if it equals its own bicommutant: $\cM = \cM'' := (\cM')'$ 
(cf.~\cite[\S 2.4]{BR02}). Then its center 
is $\cZ(\cM) = \cM \cap \cM'$.  
\item For a selfadjoint operator $H = H^*$ on $\cH$ we write $H \geq 0$ if 
its spectrum $\Spec(H)$ in contained in $[0,\infty)$. We then say that 
the corresponding unitary one-parameter group 
$(e^{itH})_{t \in \R}$ has {\it positive spectrum}. 
\item For a (continuous) unitary representation $(\pi, \cH)$ of a Lie group 
$G$ with Lie algebra $\g$ and exponential function 
$\exp \: \g \to G$, we write 
$\partial \pi(x)$ for the skew-adjoint infinitesimal generator of 
the unitary one-parameter group 
$\pi(\exp t x)$, so that we have 
$\pi(\exp t x) = e^{t \partial \pi(x)}$ in the sense of measurable functional calculus 
for unbounded normal operators 
(cf.\ Definition~\ref{def:smoothrep}). To $\pi$ we associate two convex cones 
in the Lie algebra: 
\begin{equation}
  \label{eq:wpi-intro} 
W_\pi := \{ x \in \g \: \inf(\Spec(-i \partial \pi(x))) > -\infty\}
\supeq C_\pi := \{ x \in \g \: -i \partial \pi(x) \geq 0\}.
\end{equation}
\item $G^\flat = G \rtimes_\alpha \R$ for a homomorphism $\alpha \: \R \to \Aut(G)$. 
\end{itemize}

\section{Generalities on ground state representations} 
 \label{sec:2} 

In this section we discuss three aspects of 
ground state representations $(\pi, U,\cH)$ for a pair $(G,\alpha)$. 
In Subsection~\ref{sec:minimal} we explain how the concept of a 
minimal implementing group from the theory of operator algebraic 
dynamical systems 
translates into our context. It provides the language to define 
minimal positive energy representations and ground state representations. 
In Subsection~\ref{subsec:2.2} we take a closer look at 
elementary properties of ground state representations. 
The main new concept we introduce is that of a 
strict ground state representation (Subsection~\ref{subsec:2.3}). 
Roughly speaking, strictness means that the representation 
$(\pi,\cH)$ of $G$ decomposes in the same way as the 
representation $(\pi^0,\cH^0)$ of the fixed point group $G^0$ on the 
minimal energy space~$\cH^0$. In particular, strictness implies that 
$(\pi^0, \cH^0)$ determines $(\pi,\cH)$. 
The purpose of this concept is to classify 
ground state representations of $G$ in terms of the representations 
$(\pi^0, \cH^0)$ of $G^0$ extending to ground state representations 
of~$G$. A key problem is to 
identify these representations of $G^0$ in intrinsic terms. 
We shall see in Section~\ref{sec:3} how this problem can be addressed for 
Lie groups.

\subsection{Minimal representations}
\label{sec:minimal}

The following refinement of the Borchers--Arveson Theorem 
can be found in \cite[Thm.~3.7]{JN21}. 
Part~(i) is in  \cite[Thm.~3.2.46]{BR02}. 
We also refer to \cite[Thm.~4.14, Lemma~4.17]{BGN20} for a detailed discussion 
of this circle of ideas. 

\begin{Theorem} \label{thm:BAthm} {\rm(Borchers--Arveson Theorem)} 
Let $\cH$ be a Hilbert space and $\cM \subeq B(\cH)$ be a von Neumann 
algebra. Further, let $(U_t)_{t \in \R}$ be a strongly continuous unitary one-parameter group on~$\cH$ for which $\cM$ 
is invariant under conjugation with the operators $U_t$, so that 
we obtain a one-parameter group $\alpha \: \R \to \Aut(\cM)$ by 
$\alpha_t(M) := \Ad(U_t)M := U_t M U_t^*$ for $M \in \cM$. 
If $U_t = e^{itH}$ with $H \geq 0$, then the following assertions hold: 
\begin{itemize}
\item[\rm(i)] There exists a strongly continuous unitary one-parameter group
$U^0_t = e^{it H_0}$ in $\cM$ with $\Ad(U^0_t) = \alpha_t$ and positive spectrum. 
It is uniquely determined by the requirement that it is minimal in the sense that, for
any other unitary one-parameter group $(V_t)_{t\in \R}$ in $\U(\cH)$, 
such that $\Ad(V_t) = \alpha_t$ for $t \in \R$, 
the unitary one-parameter group $(V_t U^0_{-t})_{t \in \R}$ 
in the commutant $\cM'$ 
has positive spectrum. 
\item[\rm(ii)] If $U^0_T = \1$ for some $T > 0$ and $\cF \subeq \cH$ is an $\cM$-invariant subspace, 
then the subspace $\cF^0 := \cF \cap \ker H_0$ satisfies 
$\lbr \cM \cF_0\rbr = \cF$. 
\item[\rm(iii)] If $T \in \R$ satisfies $\alpha_T = \id_\cM$, then $U^0_T =\1$. 
\end{itemize}
\end{Theorem}

The Borchers--Arveson Theorem applies naturally to 
positive energy representations (\cite[Cor.~3.9]{JN21}), 
where it takes the following form: 

\begin{Corollary} \label{cor:borch}   
Let $(\pi, U,\cH)$ be a positive energy representation  of 
$(G,\alpha)$ and $(U^0_t)_{t \in \R}$ be the minimal 
positive implementing unitary one-parameter group in~$\pi(G)''$ from 
{\rm Theorem~\ref{thm:BAthm}}.
Then $U$ factorizes as 
\[ U_t = U_t^0 W_t \quad \mbox{ for } \quad t \in \R, \] 
where $(W_t)_{t \in \R}$ 
is a unitary one-parameter group with positive spectrum 
in the center $\cZ(\pi^\flat(G)'')$ 
of the von Neumann algebra $\pi^\flat(G^\flat)''$ 
generated by $\pi^\flat(G^\flat)$, where 
$\pi^\flat(g,t) = \pi(g) U_t$. 
\end{Corollary}

\begin{prf} We apply Theorem~\ref{thm:BAthm} to the von Neumann algebra 
$\cM := \pi(G)''$. As 
$W_t := U_{-t}^0 U_t$ commutes with $\pi(G)$, it is contained 
in $\pi(G)'$. Further,  $U_{-t}^0 \in \pi(G)''$ implies 
$W_t \in \pi^\flat(G^\flat)''$. Since 
$W_t$ commutes with $\pi(G)$ and $U_\R$, it is central in $\pi^\flat(G^\flat)''$. 
\end{prf}

\begin{cor} \label{cor:irred} {\rm(\cite[Thm.~2.5]{Ne14})} 
If $(\pi, U,\cH)$ is an irreducible
 positive energy representation  of 
$G^\flat = G \rtimes_\alpha \R$, then its restriction to $G$ is also 
irreducible.
\end{cor}

\begin{prf} By Schur's Lemma, the irreducibility of $\pi^\flat$ 
implies that $\pi^\flat(G^\flat)' = \C \1$. Hence $W_t \in \T \1$ 
implies $U_t \in \pi(G)''$. Therefore 
$\C \1 = \pi^\flat(G)' = \pi(G)'$, so that the irreducibility of 
$\pi$ follows from Schur's Lemma. 
\end{prf}

The concept of a minimal representation that we now define is inspired by the 
situation in Corollary~\ref{cor:borch}. We shall see 
in Proposition~\ref{prop:swallow} that the concept of a 
ground state representation is more restrictive, but 
it is this extra assumption of the existence of ``sufficiently many'' 
ground states that permits us to obtain more 
information to classify representations. 

\begin{defn} (Minimal and ground state 
representations) \label{def:mini} A positive energy 
representation $(\pi,U,\cH)$ of $G^\flat$ with $U_t = e^{itH}$ is called 
\begin{itemize}
\item {\it minimal} if the one-parameter group $U$ 
is minimal with respect to  the von Neumann algebra $\pi(G)''$, 
i.e., $U = U^0$ (see \cite[\S 5]{BGN20} for a discussion of this 
concept in the context of von Neumann algebras). 
\item a {\it ground 
state representation} if $\lbr \pi(G) \cH^0 \rbr = \cH$ holds 
for the {\it minimal energy subspace} 
\[ \cH^0 := \ker H.\]
\end{itemize}

In view of the factorization in Corollary~\ref{cor:borch}, 
we adopt the point of view that we understand positive 
energy representations if we know the minimal ones. For those, 
the extension of the representation $\pi$ of $G$ to $G^\flat$ 
is uniquely determined by the Borchers--Arveson Theorem. As 
the extendability of a unitary 
representation $(\pi,\cH)$ of $G$ to a positive energy representation 
of $G^\flat$ is an intrinsic property, we 
also call $(\pi, \cH)$ a {\it positive energy representation} 
of $(G,\alpha)$ if it extends to a positive energy representation of $G^\flat$. 
Here we keep in mind that the 
minimal extension with $U = U^0$ provides a natural 
extension $\pi^\flat$ to $G^\flat$ with the same commutant, 
hence with the same closed invariant subspaces. 
\end{defn}

With this terminology, Theorem~\ref{thm:BAthm} has the following 
consequence: 
\begin{cor} \label{cor:borch2}   
Let $(\pi, \cH)$ be a positive energy representation 
of $(G,\alpha)$. If $\alpha_T = \id_G$ for some $T > 0$, then $U^0_T = \1$ 
and $(\pi,\cH)$ is a ground state representation. 
\end{cor} 

We write $\hat G$ for the set of equivalence classes of 
irreducible unitary representations of the topological group 
$G$. Since the problem of parametrizing $\hat G$ 
for an infinite dimensional Lie group $G$ 
is rather intractable (\cite{Ne14c}), the positive energy 
condition with respect to some $\alpha$ provides 
a regularity condition 
that selects a class of representations for which one expects concrete 
classification results (\cite{JN19}, \cite{Ne12, Ne14, Ne14b}). 

\begin{rem}
Every $\R$-action $\alpha$ on $G$ specifies 
a subset $\hat G(\alpha)\subeq \hat G$ of irreducible positive energy 
representations with respect to $\alpha$  (cf.\ Corollary~\ref{cor:irred}), 
and these subsets of $\hat G$ 
are expected to be more tractable for specific $\alpha$'s 
than general irreducible representations whose classification is elusive. 
\end{rem}

\subsection{Ground state representations} 
\label{subsec:2.2}

It is a natural problem to describe 
the positive energy representations $(\pi, \cH)$ of $(G,\alpha)$ in 
terms of simpler data. For ground state representations, 
the necessary information should be  provided by  the representation of $G^0$ 
on~$\cH^0$. 
Here the following observation is an important tool. 

\begin{prop} \label{prop:swallow} 
{\rm(Ground state representations are minimal)} For a 
ground state representation 
$(\pi, U,\cH)$, the following assertions hold: 
\begin{itemize}
\item[\rm(i)] $(\pi, U,\cH)$ is minimal, hence in particular $U_\R \subeq \pi(G)''$.
\item[\rm(ii)] Let $P_0$ denote the orthogonal projection onto $\cH^0$. 
Then the restriction map 
\[ R \: \pi(G)' \to (P_0 \pi(G)'' P_0)' \subeq B(\cH^0), \quad 
A \mapsto P_0 A P_0 \] 
is an isomorphism of von Neumann algebras whose range 
is contained in $\pi^0(G^0)'$.  
\end{itemize}
\end{prop}

\begin{prf} (i) follows by applying \cite[Prop.~5.4]{BGN20} to the 
$C^*$-algebra $\pi(G)''$. 

\par\nin (ii) By (i), $U_ t= U^0_t \in \cM := \pi(G)''$, so that 
the orthogonal projection $P_0$ of $\cH$ onto $\cH^0$ is contained in $\cM$. 
That $\cH^0$ is generating under $\pi(G)$ shows that 
the central support of $P_0$ in $\cM$, i.e., 
the minimal central projection $Z$ with 
$Z P_0 = P_0$, is $\1$. 
Therefore \cite[Lemma~3.14(iv)]{BGN20} implies that 
the restriction map $R$ 
is an isomorphism onto the commutant of the von Neumann algebra  
$\cM^0 := P_0 \cM P_0 \supeq \pi^0(G^0)$, hence contained in~$\pi^0(G^0)'$. 
\end{prf}

We recall that a unitary representation $(U,\cH)$ of a group $G$ is called 
{\it factorial} if the von Neumann algebra $U(G)''$ is a factor, i.e., 
its center is trivial: $\cZ(U(G)'') = \C \1$. 

\begin{lem} Let $(\pi, U, \cH)$ be a factorial minimal positive energy 
representation and $U_t = e^{itH}$. 
If~$\cH^0 := \ker H$ is non-zero, then 
$\lbr \pi(G)\cH^0\rbr = \cH$, so that $(\pi,\cH)$ is a ground state representation. 
\end{lem}

\begin{prf} The subspace $\cF := \lbr \pi(G)\cH^0\rbr$ is invariant 
under the von Neumann algebra $\cM := \pi(G)''$ generated by $\pi(G)$ 
because it is $\pi(G)$-invariant. 
Since $\cH^0$ is also invariant under $\cM'$, the same holds for $\cF$. 
Therefore $\cF$ is invariant under $(\cM \cup \cM')'' = B(\cH)$. 
Here we use that factoriality implies that 
$(\cM \cup \cM')' = \cM' \cap \cM'' = \C \1$. 
Invariance under $B(\cH)$ clearly entails $\cF = \cH$. 
\end{prf}

The following lemma is a useful elementary tool to verify the positive 
energy and the ground state condition for direct sums. 
\begin{lem} \label{lem:dirsum}
Let $(\pi, \cH)$ be a unitary representation of $G$ 
which is a direct sum of unitary subrepresentations $(\pi_j,\cH_j)_{j \in J}$. 
Then $(\pi,\cH)$ is a positive energy (ground state) representation for 
$(G,\alpha)$ if and only if all the representations $(\pi_j, \cH_j)_{j \in J}$ 
have this property.

In particular, every subrepresentation of a positive energy (ground state)
 representation inherits this property.  
\end{lem}

\begin{prf} ``$\Rarrow$'': If $(\pi,\cH)$ is of positive energy, 
then $U^0_t \in \pi(G)''$ for every $t \in \R$ (Theorem~\ref{thm:BAthm}), 
so that the $U^0_t$  preserve all 
the subspaces $\cH_j$. Therefore its restriction to $\cH_j$ 
shows that all representations $(\pi_j,\cH_j)$ are of positive energy 
since the one-parameter groups $(U^0_t\res_{\cH_j})_{t \in \R}$ have positive spectrum. 

If, in addition, $(\pi,U, \cH)$ is a ground state representation, 
then Proposition~\ref{prop:swallow}(i) implies that $U =U^0$ is minimal, 
so that $U_\R \subeq \pi(G)''$. Therefore $\cH^0 = \ker H$ 
is invariant under the commutant $\pi(G)'$, and this implies that 
\[ \cH^0 = \oline{\sum_{j \in J} \cH_j \cap \cH^0} \cong \hat\bigoplus_{j \in J} \cH_j^0.\] 
As $\pi(G)\cH^0$ is total in $\cH$, projection to $\cH_j$ shows that 
$\pi_j(G)\cH_j^0$ is total in $\cH_j$, i.e., $(\pi_j,U_j,\cH_j)$ is a 
ground state representation. 

\nin``$\Larrow$'': If all representations $(\pi_j,\cH_j)$ are of positive energy, 
then $U_t\res_{\cH_j} := U_{j,t}^0$ defines a positive unitary one-parameter group 
with the correct commutation relations, so that $(\pi, \cH)$ is of positive energy. 

If all representations $(\pi_j,\cH_j)$ are ground state representations, 
then $\cH^0 = \ker H_0 = \hat\oplus_{j \in J} \cH_j^0$ satisfies 
$\lbr \pi(G) \cH^0 \rbr = \cH$, 
and therefore $(\pi,\cH)$ is a ground state representation. 
\end{prf}

In the following proposition we take a closer look at the special 
case where $\alpha$ is given by conjugation with a 
one-parameter subgroup $\gamma \: \R \to G$, i.e., by inner 
automorphism. 
Then the spectrum of $\pi \circ \gamma$ determines the positive energy  
properties of $(\pi,\cH)$. 

\begin{prop} {\rm(The inner case)} \label{prop:inner}
Let $\gamma \: \R \to G$ be a continuous one-parameter group and 
suppose that $\alpha$ is defined by 
\[ \alpha_t(g) = \gamma(t) g \gamma(t)^{-1} \quad \mbox{ for } \quad 
t\in \R, g \in G.\] 
We consider a unitary representation  $(\pi, \cH)$ of $G$ and 
write $\pi(\gamma(t)) = e^{itH}$ for $t \in \R, H^* = H$. 
Then the following assertions hold: 
\begin{itemize}
\item[\rm(i)] If $H$ is bounded from below, then there exists  
$c \in \R$ such that  $U_t := e^{it(H + c \1)}$ defines a positive energy 
representation $(\pi, U,\cH)$. 
\item[\rm(ii)] If $(\pi, \cH)$ is factorial, then 
it is of positive energy for $(G,\alpha)$ if and only if 
$H$ is bounded from below. 
\item[\rm(iii)] $(\pi,\cH)$ is a positive energy representation 
for $(G,\alpha)$ if and only if it is 
a direct sum of representations $(\pi_j, \cH_j)_{j \in J}$ 
for which the corresponding generators 
$H_j$ of $\pi_j \circ \gamma$ are bounded from below. 
\end{itemize}
\end{prop}

\begin{prf}  (i) is clear. 

\nin (ii) If $(\pi,U,\cH)$ is a positive energy representation, we have 
\begin{equation}
  \label{eq:2.10}
 U_t^0  = \pi(\gamma(t)) V_t \quad \mbox{ with } \quad 
V_t \in \cZ(\pi(G)'') = \pi(G)'' \cap \pi(G)'.
\end{equation}
If $\pi$ is factorial, i.e., $\cZ(\pi(G)'') = \C \1$, 
then $V_t = e^{ict}\1$ for some $c \in \R$. For 
$U^0_t = e^{itH^0}$, this implies that 
$\pi(\gamma(t)) = e^{it(H_0 - c \1)}$ has the generator 
$H_0 - c \1$ which is bounded from below by $- c \1$. 
The converse follows from (i). 

\nin (iii) Let $(\pi, \cH)$ be a positive energy representation 
and $U_t^0 = e^{itH} V_t$ be as in \eqref{eq:2.10}  with 
$V_t \in \cZ(\pi(G)'')$. If $V_t = e^{it Z}$, then all spectral 
subspaces $\cH_n := P^Z([n,n+1))$, $n \in \Z$, of $Z$ are 
$\pi(G)$-invariant, hence lead to subrepresentations 
$(\pi_n, \cH_n)_{n \in \N}$. On $\cH_n$ the restriction $Z_n$ of 
$Z$ has spectrum in $[n,n+1]$, hence is bounded. 
Therefore $H_n = H_n^0 - Z_n$ is bounded from below. 

If, conversely, the representations $(\pi_n, \cH_n)_{n \in \N}$ 
are such that the generators $H_n$ are bounded from below, say 
$H_n \geq c_n \1$, then $U^n_t := e^{it (H_n - c_n\1)}$ defines a covariant 
positive energy representation $(\pi_n, U^n, \cH_n)$. 
Now Lemma~\ref{lem:dirsum} implies that the direct sum also is a positive 
energy representation. 
\end{prf}

\begin{ex} \label{ex:trivial} (a) 
If the $\R$-action on $G$ is trivial, then every unitary 
representation $(\pi, \cH)$ of $G$ is a ground state 
representation with $\cH = \cH^0$ and $U^0_t = \1$ for $t \in \R$. 

\nin (b) If $(\pi, \cH)$ is a positive energy representation, 
then $\pi(Z(G))$ is pointwise fixed by $\Ad(U_t)$. In particular, 
for an abelian group $G$ and a faithful positive energy 
representation $\pi$, we have $\alpha_t = \id_G$ for $t \in \R$ 
(cf.\ \cite{Ne14}). 

Therefore the simplest class of groups with non-trivial 
positive energy representations are two-step nilpotent groups $G$, 
for which Heisenberg groups are the simplest examples. Then 
$\pi((G,G)) \subeq \T \1$ for every irreducible representation, 
which leads to oscillator groups $G^\flat$. 
We discuss this special case in Section~\ref{sec:4}. 
For these groups the existence of semibounded representations 
(cf.~Definition~\ref{def:momset}) and 
the existence of non-trivial ground state representations has been studied 
in  \cite{NZ13} and \cite{Ze14}. 
\end{ex}

\subsection{Strict ground state representations} 
\label{subsec:2.3}

We now identify a class of  ground state representations for 
which the subgroup $G^0$ is ``large'' in a suitable sense. 
This is the class of strict ground state representations, 
which are determined by the representation $(\pi^0, \cH^0)$ of~$G^0$. 
In Example~\ref{ex:4.7} we shall see a ground state representation 
which is not strict. 

\begin{defn} \label{def:strict}
We call a ground state representation 
$(\pi,\cH)$ for $(G,\alpha)$ {\it strict} if 
every operator on $\cH^0$ commuting with $\pi^0(G^0)$ 
extends to an operator on $\cH$ commuting with $\pi(G)$. 
In view of Proposition~\ref{prop:swallow}, this is equivalent to the 
following identity of the von Neumann algebras in $B(\cH^0)$: 
\[ \pi^0(G^0)'' = P_0 \pi(G)'' P_0.\] 
As $\pi(G)$ spans a weakly dense subspace of $\pi(G)''$, 
the von Neumann algebra $P_0 \pi(G)'' P_0 \subeq B(\cH^0)$ is generated by 
$P_0 \pi(G) P_0$ which always contains $\pi^0(G^0)$, hence also 
$\pi^0(G^0)''$. 
\end{defn}

\begin{problem} Suppose that $\alpha$ is periodic. Is every 
ground state representation of $(G,\alpha)$ strict?     
  \end{problem}

  \begin{rem} If the ground state representation 
$(\pi, \cH)$ of $(G,\alpha)$ is irreducible, then 
Schur's Lemma implies $\pi(G)'' = B(\cH)$ which leads to 
$P_0 \pi(G)'' P_0 = P_0 B(\cH) P_0 = B(\cH^0)$. Therefore 
$(\pi, \cH)$ is strict if and only if $(\pi^0, \cH^0)$ is also irreducible, 
i.e., $\pi^0(G^0)'' = B(\cH^0)$. 
      \end{rem}

\begin{rem}  \label{rem:2.14} 
Since von Neumann algebras are generated by their projections, 
a ground state representation is strict if and only if 
the map $\cF \mapsto \cF \cap \cH^0$ defines a bijection 
between the closed $\pi(G)$-invariant subspaces of $\cH$ 
and the the closed $\pi^0(G^0)$-invariant subspaces of~$\cH^0$. 
As this map is injective (cf.\ Proposition~\ref{prop:swallow}(i)), 
the main point is its surjectivity.
\end{rem}

\begin{ex} (An operator algebraic example of a strict ground state 
representation) 
Let $\cM \subeq B(\cH)$ be a von Neumann algebra 
and $G := \U(\cM)$ be its unitary group. 
For a unitary one-parameter group $U_t = e^{itH}$ in $\cM$ we obtain a 
continuous action on $G$ by $\alpha_t(g) = U_t g U_{-t}$ for $t \in \R, g \in \U(\cM)$. 
We assume that $H \geq 0$ and 
$U = U^0$ in the sense of the Borchers--Arveson Theorem 
(Theorem~\ref{thm:BAthm}). Then the identical representation of $G$ on $\cH$ is a 
positive energy representation. 

It is a ground state representation if and only if $\cH^0 = \ker H$ 
satisfies $\lbr \U(\cM) \cH^0 \rbr = \cH$, which is equivalent to 
$\lbr \cM\cH^0 \rbr = \cH$. For the projection $P_0$ onto $\cH^0$, 
which is contained in $\cM$, this is equivalent to its central 
support being equal to $\1$ (\cite[Lemma~3.14]{BGN20}). 

The group $G^0$ is the centralizer of $U_{\R}$ in $G = \U(\cM)$, hence 
contained in the subalgebra 
\[ P_0 \cM P_0 \oplus (\1 - P_0) \cM (\1 - P_0) \] 
and it contains the unitary group $\U(P_0 \cM P_0)$. 
We conclude that $\pi^0(G^0) = \U(P_0 \cM P_0)$, 
and the von Neumann algebra 
generated by this group is $P_0 \cM P_0 = P_0 G'' P_0$. 
Therefore the ground state representation of $(G,\alpha)$ on 
$\cH$ is strict. 
\end{ex}

\begin{prop} \label{prop:strictsubrep}
If $(\pi, \cH)$ is a strict ground state representation, 
then every subrepresentation is also a strict ground state representation. 
\end{prop}

\begin{prf} Let $(\rho,\cF)$ be a subrepresentation of $(\pi,\cH)$ 
and $P_\cF$ denote the orthogonal projection onto~$\cF$.
Then $Q_0 := P_\cF P_0$ is the projection onto $\cF^0$, and 
Lemma~\ref{lem:dirsum} implies that $(\rho,\cF)$ is a ground state representation. 
Let $A \in B(\cF^0)$ commute with $\rho^0(G^0)$. 
Extending $A$ by $0$ on the orthogonal complement of $\cF^0$ in $\cH^0$, 
we obtain an operator $A' \in \pi^0(G^0)'$. 
This operator commutes with $P_0 \pi(G) P_0$ by strictness of $(\pi,\cH)$, 
and therefore $A$ commutes with 
$Q_0 \rho(G) Q_0 = P_\cF P_0 \pi(G) P_0 P_\cF$. 
Hence $(\rho, \cF)$ is also strict. 
\end{prf}

\begin{defn} We say that the pair $(G,\alpha)$ has the {\it unique 
extension property} if two ground state representations 
$(\pi_j, \cH_j)_{j = 1,2}$ for which the $G^0$-representations 
$(\pi_1^0, \cH_1^0)$ and 
$(\pi_2^0, \cH_2^0)$ are equivalent, 
the representations $\pi_1$ and $\pi_2$ are unitarily equivalent, 
that is, the following diagram commutes: 
\[ \mat{ 
(\pi_1^0, \cH_1^0) & \into & (\pi_1, \cH_1) \\
\mapdown{\forall\phi^0}  && \mapdown{\exists\phi} \\ 
(\pi_2^0, \cH_2^0) & \into & (\pi_2, \cH_2)}\] 
\end{defn}

The following lemma is a key to some of our main results below, in 
particular to Theorem~\ref{thm:5.3}. 

\begin{prop} 
\label{prop:strictvsuniqueext} 
{\rm(Strictness and unique extension)} 
A pair $(G,\alpha)$ has the unique extension property if and only if 
every ground state representation of $(G,\alpha)$ is strict. 
If this is the case, then, for ground state representations 
$(\pi_j, \cH_j)_{j =1,2}$ and any unitary $G^0$-equivalence 
$\Phi^0 \: \cH_1^0 \to \cH_2^0$, there exists a unique $G$-equivalence 
$\Phi \: \cH_1 \to \cH_2$ extending~$\Phi^0$.
\end{prop}

\begin{prf} We first observe that the second statement on the uniqueness of 
$\Phi$ follows from the fact that we must have 
$\Phi(\pi_1(g)\xi) = \pi_2(g) \Phi^0(\xi)$ for 
$g \in G$, $\xi \in \cH_1^0$, and $\pi_1(G)\cH_1^0$ is total in~$\cH_1$.

We now show the first statement. Suppose first that 
$(G,\alpha)$ has the unique extension property 
and that $(\pi, \cH)$ is a ground state representation.  
In view of Proposition~\ref{prop:swallow}(ii), it suffice to show that 
$R$ is surjective. Since the von Neumann algebra 
$\pi^0(G^0)'$ is generated by its unitary 
elements $V$, it suffices to observe that the unique extension property 
implies that any such $V$ extends uniquely to an element $\tilde V \in \pi(G)'$. 
Therefore every ground state representation is strict. 

Now we assume that every ground state representation is strict. 
Let $(\pi_j, \cH_j)_{j = 1,2}$ be two ground state representations 
and $V \: \cH_1^0 \to \cH_2^0$ be a unitary $G^0$-equivalence. Then 
\[ W \: \cH_1^0 \oplus \cH_2^0 \to \cH_1^0 \oplus \cH_2^0, \quad 
W(\xi,\eta) := (V^*\eta, V\xi) \] 
is a unitary element in the commutant of 
$(\pi_1^0 \oplus \pi_2^0)(G^0)$. As  
$\rho := \pi_1 \oplus \pi_2$ is a ground state representation 
by Lemma~\ref{lem:dirsum}, it is strict. 
Hence there exists an element $\tilde W \in \rho(G)'$ extending~$W$. 
As the restriction map $R$ is an injective homomorphism of 
$*$-algebras, $\tilde W$ is unitary. 
Further, $W \cH_1^0 = \cH_2^0$ implies that 
$\tilde W\cH_1= \cH_2$, so that $\tilde V := \tilde W\res_{\cH_1}\: \cH_1 \to \cH_2$ 
is a unitary $G$-equivalence extending $V$. 
\end{prf}

  \begin{rem} If $(\pi, \cH)$ is a ground state representation 
and $P_0$ the orthogonal projection onto~$\cH^0$, then 
\[ \phi(g) := P_0 \pi(g) P_0 \in B(\cH^0) \] 
defines a positive definite $B(\cH^0)$-valued function on $G$ with 
\begin{equation}
  \label{eq:expectval}
\phi(h_1 g h_2) = \pi^0(h_1) \phi(g) \pi^0(h_2) \quad \mbox{ for } \quad 
g \in G, h_1, h_2 \in G^0
\end{equation}
and in particular $\phi\res_{G^0} = \pi^0$. 
The requirement $\lbr \pi(G)\cH^0 \rbr = \cH$ implies that the 
representation $(\pi, \cH)$ is equivalent to the 
GNS representation $(\pi_\phi, \cH_\phi)$ in the reproducing kernel 
Hilbert space $\cH_\phi \subeq C(G,\cH^0)$.  
We refer to Proposition~\ref{prop:gns} in 
Appendix~\ref{app:B} 
for a precise formulation of the vector-valued 
Gelfand--Naimark--Segal (GNS) construction.
From this perspective, the unique extension property asserts that 
the representation $(\pi^0, \cH^0)$ of $G^0$ determines the function 
$\phi$ if $\pi_\phi$ is a ground state representation with 
$\cH_\phi^0 = \cH^0$. 
\end{rem}

\section{Lie theoretic aspects} 
\label{sec:3}

To formulate necessary conditions for a representation 
$(\pi^0, \cH^0)$ to extend to a ground state representation for 
$(G,\alpha)$, it is instructive to take a closer look at the context 
of, possibly infinite dimensional, Lie groups. 

\subsection{Regularity conditions for actions on Lie groups} 

We assume that 
\begin{itemize}
\item[(L1)] $G$ is a Lie group with Lie algebra $\g$, modeled on a locally convex space (cf.\ \cite{Ne06}). 
\item[(L2)] $\alpha$ is smooth, so that $G^\flat$ is a Lie group. 
The  $\R$-action on its Lie algebra $\g = \Lie(G)$ 
is denoted by $\alpha^\g_t := \Lie(\alpha_t) \in \Aut(\g)$ 
and we write $D := \derat0 \alpha^\g_t  \in \der(\g)$ for the infinitesimal generator of 
this one-parameter group. 
\item[(L3)] $G^0$ is a Lie group with Lie algebra $\g^0 = \Fix(\alpha^\g)$. 
\item[(L4)] The subspace $\g_+ := \oline{D(\g)}$ complements $\g^0$ 
in the sense that $\g = \g^0 \oplus \g_+$ as topological vector spaces. 
Then $p_0 \: \g \to \g^0$ is an $\alpha^\g$-invariant continuous 
projection onto the subspace 
of fixed points. 
\end{itemize}
We define the {\it $\alpha$-cone in $\g^0$} by 
\begin{equation}
  \label{eq:cd}
C_\alpha :=  \oline\conv \{ p_0([Dx,x])  \: x \in \g\} \subeq  \g^0 
\end{equation}
for the closed convex cone generated by all elements $p_0([Dx,x])$. 

\begin{rem} \label{rem:3.1} (a) 
 If $\alpha^\g$ is continuous and periodic, i.e., $\alpha^\g_T = \id_\g$ for 
some $T > 0$, and $\g$ is complete, then 
\[ p_0(x) := \frac{1}{T} \int_0^T \alpha^\g_t(x)\, dt \] 
is a projection onto the subspace $\g^0$ of fixed points and (L4) is satisfied. 

\nin (b) If $\g$ is finite dimensional, (L4) 
means that the generalized $0$-eigenspace of $D$ 
coincides with the eigenspace. Clearly, this is the case if 
$D$ is semisimple. 
\end{rem}

\begin{rem} \label{rem:4.7}
(The diagonalizable case) 
Suppose that the one-parameter group $\alpha^\g$ of Lie algebra automorphisms 
can be written as $\alpha^\g_t = e^{t\ad \bd}$ for some $\bd \in \g$ 
for which $\ad \bd$ is diagonalizable on~$\g_\C$ with purely imaginary 
eigenvalues: 
\[ \g_\C = \bigoplus_{\lambda \in \R} \g_\C^\lambda(-i \bd) 
\quad \mbox{ with } \quad 
\g_\C^\lambda(-i\bd) 
= \{ z \in \g_\C \: [\bd,z] = i\lambda z \}.\] 
Defining $(x + iy)^* := -x + iy$ for $x,y \in \g$, we have 
$\g = \{ z \in \g_\C \: z^* = -z\}$ and 
\[ [\bd, \g] = [\bd,\g_\C] \cap \g 
= \Big\{ z = \sum_{\lambda \not=0} z_\lambda \: z^* = - z\Big\}.\] 
For $x = \sum_\lambda x_\lambda \in \g$ we then have 
$-x = x^* = \sum_\lambda x_\lambda^*$ with $x_\lambda^* = - x_{-\lambda}$. 
Write $p_0 \: \g_\C \to \g_\C^0$ for the fixed point projection introduced by (L4). 
It is given by the $0$-component $x_0 := p_0(x)$.  
\begin{equation}
  \label{eq:sumpos}
 p_0([[\bd,x],x])
= \sum_{\lambda, \mu} i p_0(\lambda [x_\lambda, x_\mu])
= \sum_{\lambda} i \lambda [x_\lambda, x_{-\lambda}]
= \sum_{\lambda} -i \lambda [x_\lambda, x_{\lambda}^*]
= 2i \sum_{\lambda > 0}  \lambda [x_\lambda^*, x_{\lambda}],
\end{equation}
where we have used that 
\[ (-\lambda) [x_{-\lambda}^*, x_{-\lambda}] 
=  (-\lambda) [- x_{\lambda}, - x_{\lambda}^*] 
=  \lambda [x_{\lambda}^*, x_{\lambda}].\]
We conclude that 
\begin{equation}
  \label{eq:c-alpha-root}
 C_\alpha = 
\oline\conv \{ i[x_{\lambda}^*, x_{\lambda}] \: \lambda > 0, 
x_\lambda \in \g_\C^\lambda(-i\bd)\}
\end{equation}
(see \eqref{eq:cd}). 

If $\alpha^\g$ is continuous and periodic and $\g$ is complete, 
then the sum of the eigenspaces is a dense subalgebra 
(Remark~\ref{rem:3.1}). 
\end{rem}

\begin{ex} (Twisted loop groups) 
Important examples where $\alpha$ is periodic arise as follows. 
Let $K$ be a Lie group with a complete Lie algebra $\fk$,  
and $\Phi \in \Aut(K)$ be of finite order $\Phi^N = \id_K$. 
We consider the {\it twisted loop group} 
\[ \cL_\Phi(K) := \{ \xi \in C^\infty(\R,K) \:  
(\forall x \in \R)\ \xi(t + 1) = \Phi^{-1}(\xi(t)) \}. \] 
This is a Lie group with Lie algebra 
\[ \cL_\phi(\fk) := \{ \xi \in C^\infty(\R,\fk) \:  
(\forall x \in \R)\ \xi(t + 1) = \phi^{-1}(\xi(t)) \}, \quad \mbox{ where } \quad 
\phi = \Lie(\Phi) \in \Aut(\fk)\] 
is the induced automorphism of the Lie algebra $\fk$ of~$K$.
Then 
\begin{equation}
  \label{eq:alphaong}
 (\alpha_t \xi)(x) := \xi(x + t) 
\end{equation}
defines a smooth action of $\R$ on $G := \cL_\Phi(K)$ with 
$\alpha_N = \id_G$. Therefore (L4) follows from 
Remark~\ref{rem:3.1}(a).
The infinitesimal generator of the automorphism group 
$\alpha_t^\g = \Lie(\alpha_t)$ (acting also by \eqref{eq:alphaong}) is given by 
$D\xi = \xi'$. The subgroup of $\alpha$-fixed points is the subgroup 
\[ G^0 \cong  K^\Phi \] 
of constant elements with values in the subgroup 
$K^\Phi$ of $\Phi$-fixed points in~$K$. If 
\[ \fk_\C^n  = \{ x \in \fk_\C \: \phi^{-1}(x) = e^{\frac{2\pi in}{N}} x\} 
 \quad \mbox{ for } \quad n \in \Z, \]
denotes the $\phi$-eigenspaces in $\fk_\C$, then 
$\fk_\C^n = \fk_\C^{n+N}$. Now 
\[  \cL_\phi(\fk_\C)^n 
= \fk_\C^{n} \otimes e_n,
\quad \mbox{ where } \quad
e_n(t) = e^{\frac{2\pi i nt}{N}}, \]
are the $D$-eigenspaces in $\cL(\fk_\C) \cong \cL(\fk)_\C$ 
corresponding to the eigenvalue 
$\frac{2\pi i n}{N}$. The expansion as Fourier 
series $x = \sum_{n \in \Z} x_n$ with 
$x_n \in \cL_\phi(\fk_\C)^n$ 
converges in $\cL_\phi(\fk_\C)$ by Harish--Chandra's Theorem 
(\cite[Thm.~4.4.21]{Wa72}) and 
\[ \im(D) = \Big\{ x = \sum_{n \in \Z} x_n \in \cL_\phi(\fk) \: x_0 = 0\Big\}.\]

From Remark~\ref{rem:4.7} we know that the cone $C_\alpha \subeq \fk^\phi$ is 
generated by the brackets 
\[ [(y_n \otimes e_n)^*, y_n \otimes e_n]  
= [y_n^* \otimes e_{-n}, y_n \otimes e_n]  
= [y_n^*, y_n] \otimes 1, \qquad y_n \in \fk_\C^n, n > 0.\] 
Therefore 
\[ C_\alpha = \oline\cone\{ i[y_n^*, y_n] \: n > 0, y_n \in \fk_\C^n\}.\] 
From $\fk_\C^{n+N} = \fk_\C^n$ it follows that, for $0 < n \leq N$, 
$z_{2N-n} := y_n^* \in \fk_\C^{2N-n}$ with $2N - n > 0$, and 
\[ [y_n^*,y_n] = [z_{2N-n}, z_{2N-n}^*] = -[z_{2N-n}^*, z_{2N-n}].\] 
Hence the cone $C_\alpha$ is a linear space which coincides 
with $\g^0\cap [\g,\g]$, which is an ideal in~$\g^0$. \\

To create a situation with a non-trivial cone $C_\alpha$, which 
by Theorem~\ref{thm:2.18} below is necessary for the existence of 
ground state representations with trivial 
kernel, one has to pass to a central extension of the loop algebras:
\[  \cL_\phi^\sharp(\fk) :=
\R \oplus_{\sigma} \cL_\phi(\fk) \quad \mbox{ where } \quad 
\sigma(\xi,\eta)
:=  \frac{1}{2\pi}\int_0^1 \kappa(\xi(t),\eta'(t))\, dt, \]
and the bracket is given by 
\[ [(t,\xi),(s,\eta)] = (\sigma(\xi,\eta),[\xi,\eta]).\]
Here $\kappa$ is a positive definite invariant bilinear form 
on $\fk$ which is also $\phi$-invariant. 
Then the elements $i[y_n^*,y_n]$ generating $C_\alpha$ have a non-trivial central component: 
\begin{align*}
 i\sigma(y_n^* \otimes e_{-n}, y_n \otimes e_n)
& = i\kappa(y_n^*, y_n) \frac{1}{2\pi} \int_0^1 e_{-n}(t) e_n'(t)\, dt 
 = i\kappa(y_n^*, y_n) \frac{1}{2\pi} \frac{2\pi i n}{N} 
 = -\kappa(y_n^*, y_n) \frac{n}{N}.
\end{align*}
For $a,b \in \fk$, the complex bilinear extension of $\kappa$ satisfies 
\[ -\kappa((a+ i b)^*, a + i b) 
= -\kappa(-a+ i b, a + i b) 
=  \kappa(a,a) + \kappa(b,b).\] 
Therefore all elements in $C_\alpha$ have a non-negative central 
component and $C_\alpha$ is non-trivial. 

If $\fk$ is abelian, this construction simply leads to a Heisenberg 
algebra, a class of examples discussed in Section~\ref{sec:4} below; 
see in particular \cite{SeG81} for the case where $\alpha$ is periodic. 
For more details on (twisted) loop groups with values in compact Lie groups, 
we refer to \cite{PS86}, \cite{Ne14b}, \cite{MN16, MN17}, \cite{JN21}. 
\end{ex}

\subsection{Momentum sets and positive cones}

\begin{defn} \label{def:smoothrep}
Let $G$ be a Lie group 
and $(\pi, \cH)$ be a unitary representation of $G$. 
An element $\xi \in \cH$ is called a {\it smooth vector} 
if its orbit map 
\[ \pi^\xi \: G \to \cH, \quad g \mapsto \pi(g)\xi \] 
is smooth. The smooth vectors form a $\pi(G)$-invariant 
subspace $\cH^\infty \subeq \cH$, 
and the representation $(\pi,\cH)$ is said to be {\it smooth} 
if $\cH^\infty$ is dense in $\cH$. 
This is always the case if $G$ is finite dimensional, 
but not in general (\cite{BN08}).

On  $\cH^\infty$ the 
derived representation $\dd\pi$ of the Lie algebra $\g = \Lie(G)$ 
is defined by 
\[ \dd\pi(x)v := \derat0 \pi(\exp tx)v.  \] 
For a smooth representation 
the invariance of $\cH^\infty$ under $\pi(G)$ implies that, 
for $x \in \g$, the operator $i \cdot\dd\pi(x)$ on $\cH^\infty$ 
is  essentially selfadjoint (cf.~\cite[Thm.~VIII.10]{RS75}) 
and that its closure,  coincides with the selfadjoint generator 
$i \partial \pi(x)$ of the unitary one-parameter group 
$\pi_x(t) := \pi(\exp tx)$, i.e., 
$\pi_x(t) = e^{t \partial \pi(x)}$ for $t \in \R$. 
\end{defn}

\begin{definition} \label{def:momset}
(a) Let $\bP({\cal H}^\infty) = \{ [v] := 
\C v \: 0 \not= v \in 
{\cal H}^\infty\}$ 
denote the projective space of the subspace ${\cal H}^\infty$ 
of smooth vectors. The map 
\[  \Phi_\pi \: \bP({\cal H}^\infty)\to \g' \quad \hbox{ with } \quad 
\Phi_\pi([v])(x) 
=  \frac{\la v, -i \cdot\dd\pi(x)v \ra}{\la v, v \ra} \] 
is called the {\it momentum map of the unitary representation $\pi$}.  
The operator $i\cdot\dd\pi(x)$ is symmetric so
that the right hand side is real, and since $v$ is a smooth vector, 
it defines a continuous linear functional on $\g$. 
We also observe that we have a natural action of $G$ on 
$\bP(\cH^\infty)$ by $g.[v] := [\pi(g)v]$, and the relation 
$$ \pi(g) \dd\pi(x) \pi(g)^{-1} = \dd\pi(\Ad(g)x) $$
immediately implies that $\Phi_\pi$ is equivariant with respect 
to the coadjoint action of $G$ on the topological dual space $\g'$.

\nin (b) The weak-$*$-closed convex hull 
$I_\pi \subeq \g'$ of the image of $\Phi_\pi$ is called the 
{\it (convex) momentum set of $\pi$}. In view of the equivariance 
of $\Phi_\pi$, it is an $\Ad^*(G)$-invariant weak-$*$-closed 
convex subset of~$\g'$. 

\nin (c) 
The momentum set $I_\pi$ 
provides complete information on the extreme spectral 
values of the selfadjoint operators $i\cdot\partial \pi(x)$: 
\begin{equation}
  \label{eq:momspec}
\sup(\Spec(i\partial \pi(x))) = s_{\pi}(x):= \sup  \la I_\pi,- x \ra 
 \quad \mbox{ for } \quad x \in \g 
\end{equation}
(cf.\ \cite[Lemma 5.6]{Ne08}). 
This relation shows that $s_\pi$ is the {\it support functional} 
of the convex subset $I_\pi \subeq \g'$, which implies that 
it is lower semicontinuous and convex. It is obviously positively homogeneous. 
The {\it semibounded} unitary representations are those for which 
the set $I_\pi$ 
is {\it  semi-equicontinuous} in the sense that its support function 
$s_\pi$ is bounded in a neighborhood of some $x_0 \in \g$.

The closed convex cone 
\begin{equation}
  \label{eq:poscon}
 C_\pi := \{ x \in \g \: -i \partial \pi(x) \geq 0\} 
\ {\buildrel\eqref{eq:momspec}\over =}\ I_\pi^\star := \{ x \in \g \: (\forall \alpha \in I_\pi)\ \alpha(x) \geq 0\} 
\end{equation}
is called the {\it positive cone of $\pi$}.
\end{definition}

\begin{defn} We call a ground state representation $(\pi,\cH)$ 
of the Lie group $G$ {\it smooth} if the subspace 
\[ \cH^{0,\infty} := \cH^0 \cap \cH^\infty \] 
is dense in $\cH^0$. This implies in particular that 
the representation $(\pi^0,\cH^0)$ of $G^0$ is also smooth.   
\end{defn}

\begin{lem} \label{lem:smoothground} 
If {\rm(L1/2)} hold and $(\pi, \cH)$ is a smooth ground state 
representation of $G$, then 
the extended representation $(\pi^\flat,\cH)$ of $G^\flat$ is smooth. 
\end{lem}

\begin{prf} The assumptions imply that 
$\cH^{0,\infty}$ is contained in the space $\cH^\infty(G^\flat)$ of 
smooth vectors for~$G^\flat$. As $\pi(G)\cH^{0,\infty}$ is total in $\cH$ 
by the ground state condition 
and $\cH^\infty(G^\flat)$ is $\pi(G)$-invariant, 
this subspace is dense. 
\end{prf}

\begin{defn} \label{def:equicont} 
Let $E$ be a locally convex space and 
let $\alpha:\R\rightarrow \GL(E)$, $t\mapsto \alpha_t$ 
be a group homomorphism. Then $\alpha$ is called 
\begin{itemize}
\item[(a)] \textit{equicontinuous}, if the subset $\{\alpha_t: t\in \R\}\subset \End(E)$ is equicontinuous (cf.~Definition~\ref{equicontdef}).
\item[(b)] \textit{polynomially bounded}, if for every continuous seminorm 
$p$ on $E$, there exists a $0$-neighborhood $U\subeq E$ 
and $N \in \N$ such that 
\[ \sup_{x \in U} \sup_{t \in \R} \frac{p(\alpha_t(x))}{1 + |t|^N} < \infty.\] 
\end{itemize}
\end{defn}

\begin{rem} \label{rem:c-polybound}
If $E$ is finite dimensional, 
then $\alpha$ is polynomially bounded if and only if the spectrum of 
its infinitesimal generator 
$A$ is purely imaginary. 
However, for an infinite dimensional Hilbert space $\cH$,  
there exists is a one-parameter group $\alpha:\R\rightarrow \GL(\cH)$ with $\|\alpha_t\|=e^{|t|}$ whose generator has purely imaginary spectrum, 
cf. \cite[Example 1.2.4]{vN96}.
\end{rem}

From \cite[Prop.~3.2]{NSZ15} we quote the following sufficient 
condition for the density of $\cH^{0,\infty}$ in~$\cH^0$.  

\begin{prop} Suppose that 
  \begin{itemize}
  \item the one-parameter group 
$(\alpha^\g_t)_{t \in \R}$ of Lie algebra automorphisms 
is polynomially bounded, 
  \item  $(\pi, \cH)$ is a smooth 
positive energy representation of $(G,\alpha)$, and
  \item there exists an 
$\eps > 0$ such that $\Spec(-i\partial \pi(\bd)) \cap [0,\eps] = \{0\}$ 
(spectral gap condition), 
  \end{itemize}
then $\cH^{0,\infty}$ is dense in $\cH^0$.  
\end{prop}

Let $\bd := (0,1)\in \g^\flat = \g \rtimes_D \R$ be the element implementing 
$D$, so that $Dx = [\bd ,x]$ for $x \in \g$.
In the setting specified above, we formulate 
in the following theorem a necessary 
positivity condition that a representation $(\pi^0, \cH^0)$ 
arising in a ground state representation of $(G,\alpha)$ has to satisfy.

\begin{thm} \label{thm:2.18} 
{\rm($C_\alpha$-positivity Theorem)} Suppose that {\rm(L1-4)} are satisfied. 
If $(\pi, \cH)$ is a smooth ground state representation of the Lie group~$G$, 
then the cone $C_\alpha$ introduced in \eqref{eq:cd} satisfies 
\begin{equation}
  \label{eq:poscond}
C_\alpha \subeq C_{\pi^0} = \{ x \in \g^0 \: -i \partial \pi^0(x) \geq 0\}.
\end{equation}
\end{thm}

\begin{prf} Let $\xi \in \cH^0 \cap \cH^\infty$, $t \in \R$ and $x \in \g$. 
Then $\pi(\exp tx)\xi$ is a smooth vector for $G^\flat$, 
hence contained in the domain of the infinitesimal generator $H$ of~$U$, and we have 
\[ f(t) := \la \pi(\exp tx) \xi, H \pi(\exp tx)\xi \ra \geq 0 
\quad \mbox{ for } \quad t \in \R \] 
because $H\geq 0$. 
As $f(0) = \la \xi, H \xi \ra = 0$, we also have 
\[ f'(0) = \la \partial \pi(x) \xi, H \xi \ra 
+ \la H \xi, \partial \pi(x) \xi \ra = 0\] 
and $f''(0) \geq 0$, and this is what we shall exploit. 
To this end, we rewrite $f$ as 
\begin{align*}
 f(t) 
&= \la \xi, \pi(\exp -tx)  H \pi(\exp tx)\xi \ra 
= \la \xi, \pi(\exp -tx) (-i\partial \pi^\flat(\bd)) \pi(\exp tx)\xi \ra \\
&= -i \la \xi, \partial \pi^\flat(e^{-t \ad x} \bd)\xi \ra.
\end{align*}
This immediately leads to 
\begin{equation}
  \label{eq:pos1}
 0 \leq f''(0) 
= -i \la \xi, \partial \pi([x,[x,\bd]]) \xi \ra
= -i \la \xi, \partial \pi([Dx,x]) \xi \ra.
\end{equation}

Next we observe that, for $y \in \g$, we have 
\begin{align*}
 \la \xi, \partial \pi(Dy) \xi \ra 
&= \frac{d}{dt}\Big|_{t = 0} 
 \la \xi,\partial \pi(e^{t D} y) \xi \ra 
= \frac{d}{dt}\Big|_{t = 0} 
 \la \xi, U_t \partial \pi(y) U_{-t} \xi \ra \\
&= \frac{d}{dt}\Big|_{t = 0} 
 \la U_{-t} \xi,  \partial \pi(y) U_{-t} \xi \ra 
= \frac{d}{dt}\Big|_{t = 0} 
 \la  \xi,  \partial \pi(y)  \xi \ra = 0
\end{align*}
because $U_t \xi = \xi$ for all $t \in \R$. 
Therefore $\la \xi, \partial \pi(z) \xi \ra = 0$ for $z \in \g_+$, and 
thus (L4) entails 
\[  \la \xi, \partial \pi([Dx,x]) \xi \ra
=  \la \xi, \partial \pi(p_0([Dx,x])) \xi \ra
=  \la \xi, \partial \pi^0(p_0([Dx,x])) \xi \ra. \] 
Since $\cH^{0,\infty}$ is dense in $\cH^0$ and invariant under $\pi^0(G^0)$, 
it follows that the operators 
$i\cdot\partial \pi^0(x_0)\res_{\cH^{0,\infty}}$, $x_0 \in \g^0$, are 
essentially selfadjoint with closure equal to $i\cdot\partial \pi^0(x_0)$. 
We therefore obtain 
\[ -i \partial \pi^0(p_0([Dx,x])) \geq 0 \quad \mbox{ for every } \quad 
x \in \g\] 
and thus $C_\alpha \subeq C_{\pi^0}$. 
\end{prf}

\begin{ex} In the context of Remark~\ref{rem:4.7}, 
for the  representation of $G^0$ on $\cH^0$, the condition  
$C_\alpha  \subeq C_{\pi_0}$ in Theorem~\ref{thm:2.18} 
is equivalent to 
\begin{equation}
  \label{eq:brackdrel}
 \partial \pi^0([x_\lambda^*, x_\lambda]) \geq 0 \quad \mbox{ for } \quad 
\lambda > 0, \quad \mbox{ and} \quad 
 [\bd, x_\lambda] = i \lambda x_\lambda, x_\lambda \in \g_\C.
\end{equation}
\end{ex}

In the following we shall encounter various circumstances, where 
\eqref{eq:poscond} is also sufficient for a representation 
$(\pi^0, \cH^0)$ of $G^0$ to extend to a ground state representation 
of~$G$.
In some cases we can derive a positivity property 
similar to \eqref{eq:poscond} from the positive energy condition. 
The idea for the following proposition 
is taken from \cite{JN21}. 

\begin{prop} \label{prop:peineq}
Let $(\pi, U,\cH)$ be a positive energy representation of 
$(G,\alpha)$ for which the extension to $G^\flat$ is smooth. 
If $x \in \g$ satisfies $(\ad x)^2 Dx = 0$, then 
\begin{equation}
  \label{eq:osciposcond}
 -i \partial \pi([Dx,x]) \geq 0.
\end{equation}
\end{prop}

\begin{prf} Let $\pi^\flat(g,t) := \pi(g) e^{itH}$ be the 
extension of $\pi$ to $G^\flat$. 
We proceed as in the proof of Theorem~\ref{thm:2.18} 
for any $\xi \in \cH^\infty$. Our assumption implies that the 
smooth function $f \:  \R \to \R$ is non-negative. Now 
$[[Dx,x],x] = 0$ leads to 
\[ f(t) 
= -i \la \xi, \partial \pi^\flat(e^{-t \ad x} \bd)\xi \ra
= -i \la \xi, H \xi \ra
- i t \la \xi, \partial \pi^\flat(Dx) \xi \ra 
-   \frac{it^2}{2} \la \xi, \partial \pi^\flat([Dx,x]) \xi \ra \geq 0\] 
for $t \in \R$. This implies that 
$-\frac{i}{2} \la \xi, \partial \pi([Dx,x]) \xi \ra \geq 0$, so that 
\eqref{eq:osciposcond} follows. 
\end{prf}

\section{Heisenberg and oscillator groups} 
\label{sec:4} 

In this section we discuss ground state representations of 
Heisenberg groups. Here  an old result by 
M.~Weinless \cite{We69} can be used to obtain crucial 
information on the structure of ground 
state representations.

We consider a locally convex space $V$, endowed with a continuous 
alternating form $\sigma \: V \times V \to \R$. We further 
assume that $\sigma$ is {\it weakly symplectic}, i.e., non-degenerate.
The corresponding {\it Heisenberg 
group} is the central extension 
\begin{equation}
  \label{eq:heisdef}
 G = \Heis(V,\sigma) := \T \times V \quad \mbox{ with } \quad 
(z,v)(z',v') = (zz' e^{\frac{i}{2} \sigma(v,v')}, v + v'). 
\end{equation}
It is a Lie group with Lie algebra $\g = \R \oplus V$, on which the bracket is
\[ [(z,v),(z',v')] = (\sigma(v,v'),0).\]
Any smooth  action $\alpha \: \R \to \Aut(G)$ 
fixing all elements in the central circle $\T \times \{0\}$ 
corresponds to a one-parameter group 
$\beta \: \R \to \Sp(V,\sigma)$ whose infinitesimal 
generator is denoted $D_V \in \sp(V,\sigma)$. 
Then the infinitesimal generator $D$ of $\alpha^\g$ has the form 
\begin{equation}
  \label{eq:ddv}
D(z,v) = (0,D_V v) \quad \mbox{ for } \quad z \in \R, v \in V. 
\end{equation}
This implies in particular that 
\[ G^0 = \T \times \ker(D_V) = \T \times V^\beta, \] 
which also is a Heisenberg group. 
The group 
\[ G^\flat := G \rtimes_\alpha \R \] 
is called the associated {\it oscillator group} 
(cf.\ \cite{NZ13}). 

We define the {\it effective subspace of $V$} as 
\[ V_{\rm eff} 
:= \oline{V_\beta}, \quad 
V_\beta := \Spann \{ \beta_t(v) -v \: t \in \R,v \in V\}.\] 
From the invariance of $\sigma$ under $\beta$, we immediately obtain 
\begin{equation}
  \label{eq:annihilbeta}
V^\beta = \{ v \in V \: (\forall t \in \R)\ \beta_t(x) = x\} 
= V_{\beta}^{\bot_\sigma}  = V_{\rm eff}^{\bot_\sigma}.
\end{equation}
Then $G_{\rm eff} := \T \times V_{\rm eff}$ is a closed $\alpha$-invariant subgroup 
of $G$ and 
\[ G_{\rm eff}^\flat 
= G_{\rm eff} \rtimes_\alpha \R 
\trile  G^\flat = G \rtimes_\alpha \R \] 
is a normal subgroup for which the quotient $G^\flat/G^\flat_{\rm eff}$ is abelian, 
and $G^0$ commutes with $G^\flat_{\rm eff}$.

We consider unitary representations $(\pi,\cH)$ of $G$ 
with $\pi(z,0) = z \1$ for $z \in \T$. 
These representations can also be viewed as 
projective representations of the abelian group $(V,+)$ 
(see \cite{JN19} for generalities on projective Lie group representations). 

For any element $x = (z,v) \in \g$, we have 
\[ [Dx,x] = [(0,D_V v), (z,v)] = (\sigma(D_V v, v),0) \in \fz(\g).\] 
Hence Proposition~\ref{prop:peineq} shows that the 
positive energy condition for $\pi$ implies the 
following positivity condition 
on $D_V \in \sp(V,\sigma)$: 
\begin{equation}
  \label{eq:peosci}
\sigma(D_V v, v) = -i \partial\pi\big( (\sigma(D_V v, v),0)\big)\geq 0 
\quad \mbox{ for } \quad v \in V
\end{equation}
(cf.\ \cite[Ex.~4.26, Prop.~4.27]{BGN20}).


\begin{defn}
  {\rm(Weinless conditions)} 
The triple $(V,\sigma, \beta)$ defines a 
{\it boson single particle space} in the sense of \cite{We69}. 
Weinless defines a {\it positive energy Bose--Einstein field} 
over $(V,\sigma,\beta)$ as a quadruple 
$(\cH,W,\Omega,U)$, consisting of a complex 
Hilbert space $\cH$, a unit vector $\Omega \in \cH$, a 
continuous unitary one-parameter group $(U_t)_{t \in \R}$ on $\cH$, 
and a map $W \: V \to \U(\cH)$, with the following properties: 
\begin{itemize}
\item[\rm(W1)] $W(x) W(y) = e^{i\frac{\sigma(x,y)}{2}} W(x+y)$ for 
$x,y \in V$ (Weyl relations).
\item[\rm(W2)] $W(\beta_t(x)) = U_t W(x) U_{-t}$ for 
$x \in V$, $t \in \R$ ($\beta$-equivariance).
\item[\rm(W3)] $U_t \Omega = \Omega$ for $t\in \R$. 
\item[\rm(W4)] $U_t = e^{itH}$ with $H \geq 0$ ($U$ has positive spectrum).
\item[\rm(W5)] The unitary one-parameter groups $W^x_t := W(tx)$, $x \in V$, 
are strongly continuous (regularity). 
\item[\rm(W6)] $W(V)\Omega$ is total in $\cH$. 
\end{itemize}
\end{defn}

These requirements translate naturally into our context. 
Relation (W1) means that 
\[ \pi(z,v) := z W(v) \]
defines a unitary representation of $\Heis(V,\sigma)$ and 
(W2,3,4,6) imply that $(\pi, U,\cH)$ defines a ground state representation 
with cyclic vector $\Omega$. Condition (W5) is a rather weak 
continuity requirement which is in particular satisfied if $W$ 
is strongly continuous. Therefore every continuous  
positive energy representation $(\pi, U,\cH)$ with a cyclic ground 
state vector $\Omega$ defines by $W(x) := \pi(1,x)$ 
a positive energy Bose--Einstein field over $(V,\sigma,\beta)$. 

The following uniqueness result is a variant of the 
Stone--von Neumann Uniqueness Theorem, which fails in general for 
infinite  dimensional spaces $V$, but the existence of $\beta$ implements 
additional structure that can be used to obtain a similar uniqueness 
for ground state representations if $V_\beta$ is large 
(\cite{BR02},\cite{Ka79, Ka85}).

\begin{thm} \label{thm:weinless} 
{\rm(Weinless' Uniqueness Theorem; \cite[Thm.~4.1]{We69})} 
Let $(\cH,W, \Omega, U)$ be a positive energy Bose--Einstein 
field over $(V,\sigma,\beta)$, i.e., {\rm(W1-6)} are satisfied, and 
\[ V_\beta = \Spann \{ \beta_t(x) - x \: x \in V, t \in \R \}.\] 
Then there exists an, up  to unitary equivalence unique, complex 
Hilbert space $\sH$ and a symplectic linear map 
$j \: (V_\beta,\sigma) \to (\sH, 2 \Im \la \cdot,\cdot \ra)$ such that 
\begin{itemize}
\item[\rm(j1)] $j(V_\beta)$ is dense in $\sH$, and 
\item[\rm(j2)] there exists a unitary one-parameter group 
$U_t^\sH = e^{itB}$ on $\sH$ with 
$j \circ \beta_t = U_t^\sH \circ j$ for $t \in \R$  and $B > 0$. 
\end{itemize}
We further have 
\begin{itemize}
\item[\rm(j3)] $\la \Omega, W(x)\Omega \ra = e^{-\frac{1}{4}\|j(x)\|^2}$ for 
$x \in V_\beta$. 
\item[\rm(j4)] If $\xi \in  \cH^U$ is $U$-fixed and orthogonal to $\Omega$, 
then $W(V_\beta)\Omega \bot \xi$. 
\end{itemize}
\end{thm}

The Weinless Theorem has interesting consequences for 
the structure of strongly continuous 
ground state representations. 

\begin{prop} \label{prop:4.3} 
Let  $(\pi, U,\cH)$  be a continuous ground state representation 
of $(\Heis(V,\sigma), \alpha)$ with 
$\pi(z,0) = z \1$ for $z \in \T$ and cylic ground state vector. 
Then the following assertions hold: 
\begin{itemize}
\item[\rm(a)] The linear map $j \: V_\beta \to \sH$ 
extends to a continuous linear map $j \: V_{\rm eff} = \oline{V_\beta} \to \sH.$
\item[\rm(b)] On the subspace $\cK_0 := \oline{W(V_{\rm eff})\Omega}$ 
we have an irreducible representation $\rho_0$ of $G^\flat_{\rm eff}$ 
with a smooth ground state vector $\Omega$ and $(\cK_0)^0 = \C \Omega$. 
\item[\rm(c)] The representation $\rho_1$ of 
$G^\flat_{\rm eff}$ on the $G^\flat_{\rm eff}$-invariant 
subspace $\cK_1 \subeq \cH$ generated by $\cH^0$ 
is equivalent to the representation 
\begin{equation}
  \label{eq:factorize}
 \cK_1 \cong \cH^0 \hat\otimes \cK_0, \quad 
\rho_1(g) = \1 \otimes \rho_0(g) \quad \mbox{ for } \quad g \in G^\flat_{\rm eff}.
\end{equation}
In particular, the commutant of $\rho_1(G^\flat_{\rm eff})$ is 
$\rho_1(G^\flat_{\rm eff})' \cong B(\cH^0) \otimes \1.$
\item[\rm(d)] The subspace $\cK_1$ is also invariant under $G^0$, which acts on it 
by $\rho_1(g) = \pi^0(g) \otimes \1$ for $g \in G^0$. 
\end{itemize}
\end{prop}

\begin{prf}
 (a)  From (j3) we derive that $j$ is bounded on a $0$-neighborhood in 
$V_\beta$, hence continuous and therefore extends to a continuous 
linear map on $V_{\rm eff}$. 

\nin (b) Clearly, $\cK_0$ 
is invariant under $G_{\rm eff}^\flat$ with 
cyclic vector $\Omega$. Since the corresponding positive  definite function 
\[ \phi(z,x,t) = \la \Omega, z W(x) U_t \Omega \ra = 
z e^{-\frac{1}{4} \|j(x)\|^2}, \qquad z \in \T, x \in V_{\rm eff}, t \in \R, \] 
is smooth, $\Omega$ is a smooth vector (\cite[Thm.~7.2]{Ne10}). 

Further, (j4) implies that $(\cK_0)^0  = \cK \cap \cH^0 = \C \Omega$ is 
one-dimensional. As $\Omega$ is cyclic, the representation  
of the commutant $\rho_0(G^\flat_{\rm eff})'$ on $(\cK_0)^0 = \C \Omega$ 
is faithful, 
hence the commutant is one-dimensional and thus 
$(\rho_0,\cK_0)$ is irreducible. 

\nin (c) Let $\xi \in \cH^U$ be a unit vector, and consider the representation 
of $G^\flat$ on the subspace $\cH_\xi$ generated by $W(V)\xi$. 
Then Weinless' Theorem applies to the representation 
of $G^\flat$ on $\cH_\xi$, and it follows that, for $x \in V_{\rm eff}$, we also have 
$\la \xi, W(x) \xi \ra = e^{-\frac{1}{4} \|j(x)\|^2}.$
If $(\xi_j)_{j \in J}$ in an orthonormal bases of $\cH^0$, it follows that the 
subspaces $\cK_{\xi_j}$ generated by $W(V_{\rm eff})\xi_j$ are mutually orthogonal, 
and the GNS Theorem implies that the representation on this cyclic subspace 
is equivalent to $(\rho_0,\cK_0)$. This proves (c).

\nin (d) The subgroup $G^0$ commutes with $G^\flat_{\rm eff}$ by 
\eqref{eq:annihilbeta} 
and acts on the subspace $\cH^0$ of ground state vectors by 
the representation $\pi^0$. Hence it also preserves the subspace $\cK_1$. 
By (c), it acts on the left tensor factor, 
which is a multiplicity space for the action of $G^\flat_{\rm eff}$. 
\end{prf}

\begin{thm} \label{thm:4.4} 
 {\rm(Factorization Theorem for ground state representations)} 
Suppose that $(V,\sigma,\beta)$ satisfies the 
\emph{weak splitting condition}
\begin{footnote}
{Condition (WCS) is a weaker version of the splitting condition 
\eqref{eq:splitcondb} that plays a crucial role in the construction 
of ground state representations by holomorphic induction in 
Section~\ref{sec:5}.}
\end{footnote}

  \begin{equation}
    \label{eq:splitcondsymp}
    V = \oline{V^\beta + V_{\rm eff}}.  \tag{WSC}
  \end{equation}
Then every ground state representation $(\pi, U, \cH)$ of 
$G^\flat$ factorizes as a tensor product $\cH = \cH^0 \hat\otimes \cK_0$ 
with 
\[ \pi(g) = \pi^0(g) \otimes \1 \quad \mbox{ for } \quad g \in G^0 
\quad \mbox{ and } \quad 
\pi(g) = \1 \otimes \rho_0(g) \quad \mbox{ for } \quad g \in G_{\rm eff}^\flat, \] 
where the representation $(\rho_0, \cK_0)$ of $G^\flat_{\rm eff}$ is 
irreducible and equivalent to the pullback of the canonical 
Fock representation of $\Heis(V_{\rm eff},\sigma)$ on $\cF_+(\sH)$, defined by 
the positive definite function 
\[ \phi(z,x) = ze^{-\frac{1}{4}\|j(x)\|^2} \] 
{\rm(see Appendix~\ref{sec:7.1} and Remark~\ref{rem:4.5} below)}. 
We further have $\pi^\flat(G^\flat)' = \pi(G)' \cong \pi^0(G^0)'$. 
In particular, $(\pi, U,\cH)$ is strict.
\end{thm}


\begin{prf} Our assumption implies that 
$W(V^\beta + V_\beta) \Omega \subeq \C W(V_\beta) W(V^\beta)\Omega 
\subeq W(V_\beta)\cH^0$ 
is total in $\cH$. Therefore $\cH = \cK_1$, and the assertion follows 
from Proposition~\ref{prop:4.3}. 
\end{prf}

\begin{rem} (a) If $\beta$ is periodic of period $T > 0$ and 
$V$ is complete, then the fixed point 
projection 
\[ p_0 \: V \to  V^\beta, \quad p_0(v) = \frac{1}{T} \int_0^T \beta_t(v)\, dt \] 
satisfies $p_0(V)= V^\beta$ and 
$\ker(p_0) = V_{\rm eff}$ by the Peter--Weyl Theorem (\cite{HM06}). 
We therefore have $V = V^\beta \oplus V_{\rm eff}$ in this case. 

\nin (b) Suppose that $V^\beta + V_{\rm eff}$ is dense in $V$, i.e., that the 
weak splitting condition (WSC) is satisfied. 
As $V^\beta$ and $V_{\rm eff}$ are $\sigma$-orthogonal, it follows that 
$(V^\beta,\sigma)$ is also symplectic. 

For every $t\not=0$, the range of the map 
$V \to V_{\rm eff}, v \mapsto \beta_t(v) -v$ generates the same 
closed subspace as its restriction to $V_{\rm eff}$. We therefore 
have $(V_{\rm eff})_{\rm eff} = V_{\rm eff}$. 

The direct sum $\tilde V := V^\beta \oplus V_{\rm eff}$ carries a 
natural symplectic form and the addition map 
$i \: \tilde V \to V$ is symplectic and $\R$-equivariant with dense range. 
Therefore the adjoint 
\[ j \: V \to \tilde V'\cong (V^\beta)' \oplus V_{\rm eff}', \quad 
j(v) = (\sigma(v,\cdot), \sigma(v,\cdot)) \] 
is injective. 
\end{rem}

\begin{rem} \label{rem:4.5} (a) If $V = \sH^\R$ (the real subspace underlying a complex 
Hilbert space $\sH$) and $\sigma(v,w) = 2 \Im \la v,w \ra$ and 
$\beta_t = e^{itD}$ with $D \geq 0$, then second quantization leads to a 
positive energy representation 
$(\pi,U,\cH)$ on the symmetric Fock space 
\[ \cH := \cF_+(\sH)  = \hat{\bigoplus}_{n \in \N_0} S^n(\sH) \] 
with cyclic unit vector $\Omega \in S^0(\sH)$ 
for which we have 
\[ \la \Omega, \pi(x) \Omega \ra = e^{-\frac{1}{4}\|x\|^2} \quad \mbox{ for } \quad 
x \in V=\sH\] 
(cf.\ Appendix~\ref{sec:7.1}). 
We conclude that the irreducible representation $(\rho_0, \cK_0)$ in 
Theorem~\ref{thm:4.4} is equivalent to the canonical representation 
of $G_{\rm eff}$ on the Fock space $\cF_+(\sH)$. 
Note that 
\[ \cH^0 = \cF_+(\ker D)\]
is one-dimensional if and only if $D > 0$.

\nin (b) Suppose that $V = V_0 \oplus \sH^\R$ is a direct sum of two 
symplectic spaces, where $\sH$ is a complex Hilbert space and 
$\beta_t(v_0 + v_1) = v_0 + e^{itD} v_1$ for $t \in \R$, $v_0 \in V_0, v_1 \in \sH$ 
and $D = D^* \geq 0$ on $\sH$. For any unitary representation 
$(\rho, \cK)$ of $\Heis(V_0,\sigma_0)$ with $\rho(z,0)  = z\1$ for $z \in \T$, 
we obtain on 
\[ \cH := \cK \hat\otimes \cF_+(\sH) \quad \mbox{ by} \quad 
\pi(z,v_0 + v_1) := \pi(z,v_0) \otimes \pi_{\cF_+}(v_0) \] 
a ground state representation of $\Heis(V,\sigma)$, 
where $\cH^0 = \cK \hat\otimes \cF_+(\ker D)$. Therefore one 
cannot expect to draw any finer conclusion on the representation 
of $G^0$ on $\cH^0$. 

\nin (c) As the representation of $G_{\rm eff}^\flat$ on $\cK_0$ is smooth, 
Proposition~\ref{prop:peineq} implies that, for $x \in V_{\rm eff}$, we have 
\begin{equation}
  \label{eq:sigpos}
 \sigma(D_V(x),x) = \la \Omega, -i \partial \pi([Dx,x]) \Omega \ra\geq 0.
\end{equation}
As $j \: V_{\rm eff} \to \sH$ is symplectic, we obtain with 
$U_t^\sH = e^{-it D_\sH}$ for $t \in \R$ the relation 
\begin{equation}
  \label{eq:sigmapos2}
  \sigma(D_V(x),x) 
= 2 \Im \la j(D_V(x)), j(x) \ra 
= 2 \Im \la -i D_{\sH} j(x), j(x) \ra 
= 2 \la D_{\sH} j(x), j(x) \ra. 
\end{equation}
This implies that $D_\sH \geq 0$
\end{rem}

\begin{lem} \label{lem:veff-nondeg} 
If there exists a ground state representation 
$(\pi, U,\cH)$ of $(\Heis(V,\sigma), \alpha)$ with 
$\pi(z,0) = z \1$ for $z \in \T$, then the following assertions hold: 
\begin{itemize}
\item[\rm(a)] The subspace 
$V_{\rm eff} \subeq V$ is symplectic, i.e., the restriction of $\sigma$ 
to $V_{\rm eff}$ is non-degenerate and 
\begin{equation}
  \label{eq:trivintersect}
 V_{\rm eff} \cap V^\beta = \{0\}.
\end{equation}
\item[\rm(b)] The map $j \: V_{\rm eff} \to (\sH, 2 \Im \la \cdot,\cdot \ra)$ 
is injective. 
\item[\rm(c)] If $\Omega$ is an eigenvector for some $x \in V_{\rm eff}$, then 
$x = 0$. 
\item[\rm(d)] If $x \in V$ with $W(x)\Omega \in \cH^0$, then $x \in V^\beta$. 
\end{itemize}
\end{lem}

\begin{prf} (a) In view of \eqref{eq:annihilbeta}, we have to verify 
\eqref{eq:trivintersect}. 
Let $x \in V_{\rm eff} \cap V^\beta$. 
Then 
\[ W(x) = \rho_1(x) = \1 \otimes \rho(x) = \pi^0(x) \otimes \1 \] 
in the terminology of Proposition~\ref{prop:4.3}, 
and thus $W(x) \in \T \1 = \pi(\T \times \{0\})$.  
This contradicts the injectivity of $\pi$ 
(Remark~\ref{rem:4.5x}). 

\nin (b) follows immediately from (a) and the fact that $j$ is symplectic. 

\nin (c) If $W(x) \Omega = \lambda \Omega$ for some $\lambda \in \T$, 
then (j3) implies $\lambda = e^{-\frac{1}{4}\|j(x)\|^2}$, so 
that $j(x) = 0$. This entails $x = 0$ because $j$ is injective by (b). 

\nin (d) For every $t \in \R$, we have 
\[ W(\beta_t(x))\Omega = U_t W(x)\Omega = W(x)\Omega \] 
because $W(x)\Omega \in \cH^0$. Therefore 
$\Omega$ is an eigenvector of $W(x)^{-1} W(\beta_t(x)) \in  W(\beta_t(x)-x)$. 
Now (c) shows that $\beta_t(x) = x$, hence that $x \in V^\beta$. 
\end{prf}

\begin{ex} \label{ex:4.7} 
We present an example where $\sigma$ is 
non-degenerate on $V_{\rm eff}$ but $V_{\rm eff} + V^\beta$ is not dense in $V$. 
We consider the Banach space 
\[ V := C([0,1],\C) \quad \mbox{ with } \quad 
\sigma(f,g) 
:= 2\Im \int_0^1 \oline{f(x)}g(x)\, dx
=  -i \int_0^1 \oline{f(x)}g(x) - \oline{g(x)}f(x) \, dx,\]
endowed with the symplectic $\R$-action, defined by 
\[ (\beta_t f)(x) = e^{-itx} f(x).\] 
Then $V^\beta = \{0\}$, and since $\sigma$ is the imaginary part 
of a hermitian scalar product, it is non-degenerate. 
By \eqref{eq:annihilbeta}, 
the vanishing of $V^\beta$ implies that $\sigma$ is non-degenerate on 
\[ V_{\rm eff} \subeq \{ f \in V \: f(0) = 0\}.\] 
As all functions in $V_{\rm eff}$ vanish in $0$, this subspace is not dense 
in~$V$.

The uniqueness of $j$ implies that $\sH = L^2([0,1])$, where
$j \: V_{\rm eff} \into \sH$ is the canonical inclusion. 
Although this inclusion is injective, there exist 
positive energy representations of $\Heis(V,\sigma)$ 
that are not multiples of Fock space representations 
(cf.\ Remark~\ref{rem:4.5}(a)). 

For any unitary representation $\kappa$ of the additive group 
$(\C,+)$, we obtain by 
$f \mapsto \kappa(f(0))$  a unitary representation of $V$. 
Let $(\rho_0, \cK_0)$ be the cyclic Fock representation of 
$\Heis(V,\sigma)$, specified by the positive definite function 
$\phi(z,f) = ze^{-\frac{1}{4}\|f\|_2^2}$ 
and let $(\kappa, \cE)$ be a cyclic representation of $\C$ with cyclic 
vector $\Omega^0$. Recall that this implies that 
$\cE \cong L^2(\R^2,\mu)$ with a finite positive measure $\mu$, 
$\Omega^0 = 1$ and 
\[ (\kappa(x+iy)F)(a,b) = e^{i (a x + by)} F(a,b).\] 
We define on the tensor product 
$\cH := \cE \otimes \cK_0$ a unitary positive energy representation of 
$\Heis(V,\sigma)^\flat$ by 
\[ \pi(z,f,t) := \kappa(f(0)) \otimes \rho_0(z,f,t).\] 
Its minimal energy subspace is $\cH^0 = \cE \otimes \Omega\cong \cE$, 
which is clearly cyclic. The commutant $\pi(\Heis(V,\sigma))'$ is 
isomorphic to the commutant $\kappa(V)' = \kappa(V)''$; where we use 
that $\kappa$ is cyclic and thus $\kappa(V)''$ maximal abelian. 
As the cyclic vector $\Omega^0 \in \cE$ separates 
$\kappa(V)' \cong \pi(\Heis(V,\sigma))'$, it is also cyclic for 
$\Heis(V,\sigma)$. Therefore $(\pi, \cH)$ defines a positive energy 
Bose--Einstein field on $\cH$. 

From the isomorphism of the commutants
$\kappa(V)' \cong \pi(\Heis(V,\sigma))'$ 
and Schur's Lemma, 
it follows that the representation $\pi$ is irreducible if and only if 
$\mu$ is a point measure, i.e., $\kappa$ is simply a character of the 
group $(\C,+)$. Then $\cH = \cK_0$ and 
$\pi(z,f,t) = \kappa(f(0)) \rho_0(z,f,t)$.

We also note that $\pi^0(G^0) = \T \1$, so that the inclusion 
$\kappa(\C)' \subeq \pi^0(G^0)' = B(\cH^0)$ is proper if and only if 
$\kappa$ is not irreducible. We conclude that 
$\pi$ is strict if and only if $\dim \cE = 1$, i.e., 
$\kappa$ is irreducible. 
\end{ex}

As a consequence of the preceding discussion, we record: 
\begin{prop} There exist pairs $(G,\alpha)$, where $G$ 
is a Banach--Lie group and $\alpha$ is a smooth action, such 
that not all smooth ground state representations are strict.   
\end{prop}

\begin{rem} One cannot 
replace the density assumption \eqref{eq:splitcondsymp} 
in Theorem~\ref{thm:4.4} by the assumption that 
$V^\beta + V_\beta$ is $\sigma$-weakly dense, i.e., that $\sigma$ 
is non-degenerate on $V^\beta$. 
If $V^\beta = \{0\}$, then $V_\beta$ is $\sigma$-weakly dense, 
and we have this situation in Example~\ref{ex:4.7}. The construction 
of the non-strict ground state representations in this example 
show that the conclusion of Theorem~\ref{thm:4.4} is invalid  in this 
example.
\end{rem}

\begin{rem} (a) 
If $\dim V < \infty$, then $V^\beta = V_{\rm eff}^{\bot_\sigma}$ intersects 
$V_{\rm eff}$ trivial, hence is a linear complement. 

\nin (b) Without the assumption of $V_{\rm eff}$ being symplectic, 
the sum $V_{\rm eff} + V^\beta$ need not be dense in $V$, 
as the following example shows. 
On $V = \R^2$ with $\sp(V,\sigma) \cong \fsl_2(\R)$, we consider 
the nilpotent element 
\[ D = \pmat{0 & 0 \\ 1 & 0} \in\fsl_2(\R) = \sp(V,\sigma)
\quad \mbox{ and  } \quad 
\beta_t = e^{tD} = \pmat{1 & 0 \\ t & 1}.\] 
Then $V^\beta = \R e_2 = V_\beta$ are both one-dimensional, 
hence not symplectic. It follows in particular, 
that any map $j \: V_\beta \to \sH$ into a complex Hilbert space $\sH$ 
with dense range is trivial. Therefore Weinless' Theorem implies 
that $\R e_2$ annihilates all ground state vectors 
(see also Theorem~\ref{thm:4.3}(ii) below). 

\nin (c) From \eqref{eq:annihilbeta} it follows that 
\[ (V^\beta + V_{\rm eff})^{\bot_\sigma}  
= (V^\beta + V_\beta)^{\bot_\sigma}  
= V^\beta \cap (V^\beta)^{\bot_\sigma} \] 
is the radical of the restriction of $\sigma$ to $V^\beta$. 
Therefore the non-degeneracy of $\sigma$ on $V^\beta$ 
is equivalent to $V^\beta + V_\beta$ being $\sigma$-weakly dense in $V$.
\end{rem}

\begin{rem} (Kay's Uniqueness Theorem) 

\nin (a) A {\it single particle structure for 
$(V,\sigma,\beta)$} is a triple $(j,\sH,U^\sH)$, consisting 
of a complex Hilbert space~$\sH$, a real linear map 
$j \: V \to \sH$ and a unitary one-parameter group 
$U_t^\sH= e^{itB}$ with $B > 0$ such that:
\[ \oline{j(V)} = \sH, \quad 
\sigma(v,w) = 2 \Im \la j(v), j(w) \ra \quad \mbox{for } \quad v,w \in V, 
\quad \mbox{ and } \quad 
j \circ \beta_t = U_t^\sH \circ j \quad \mbox{  for } \quad t \in \R.\]
Conditions (j1/2) in Theorem~\ref{thm:weinless} imply that 
$(j, \sH, U^\sH)$ defines a single particle structure 
for $(V_\beta, \sigma,\beta)$. According to 
Kay's Uniqueness Theorem (\cite{Ka79, Ka85}), any 
two single particle structures for 
$(V_\beta,\sigma,\beta)$ are unitarily equivalent. 
Note that single particle structures on $V$ can only exist 
if $V^\beta = \{0\}$ because $j$ is supposed to be injective and $U^\sH$ has no 
non-zero fixed points on $\sH$. 

\nin (b) If we start with a complex Hilbert space 
$\cH$ and $V = \cH^\R$ (the underlying real space) with 
$\sigma = 2 \Im \la \cdot,\cdot \ra$, then 
any unitary one-parameter group $\beta_t = e^{itA}$ on $\cH$ 
defines a boson single particle space 
$(V,\sigma,\beta)$ and $V = \ker A \oplus V_\beta$.

Suppose that $\ker A = 0$ and write 
$A = R |A|$ for the polar decomposition of $A$. Then 
$R$ is a unitary involution and $I := iR$ defines a new complex 
structure on $\cH$; we write $\sH$ for the resulting complex 
Hilbert space. Then the identity map 
$j \: \cH \to \sH$ is symplectic 
and $\beta_t = e^{I t |A|}$, so that $(j,\sH, \beta)$ 
is the unique single particle structure for $(V,\sigma,\beta)$. 

\nin (c) An important special case arises from the 
translation action $(\beta_t f)(z) = f(e^{it}z)$ of $\R$ on 
$\cH = L^2(\T)$. Then 
\[ \sigma(\xi,\eta) 
=  2\int_0^1 \Im(\oline{\xi(t)} \eta(t))\, dt 
=  \frac{1}{i} \int_0^1 \oline{\xi(t)} \eta(t) - \oline{\eta(t)} \xi(t))\, dt 
\]
In this case $(Af)(z) = z f'(z)$ with $\ker A = \C 1$ (the constant 
functions). Therefore $R$ corresponds to the Hilbert transform on 
$1^\bot \subeq L^2(\bS^1)$. 
\end{rem}

\section{Holomorphic induction} 
\label{sec:5}

One can use the technique of holomorphic induction 
to show that certain ground state representations 
are strict and to obtain further regularity properties. 
In \cite[App.~C]{Ne14b}, it is shown that this technique, 
developed in \cite{Ne13} for Banach--Lie groups, also applies 
to Fr\'echet--BCH--Lie groups satisfying the conditions (SC) and (A1-3) below. 

In this section we first describe the geometric environment 
on the Lie group and Lie algebra level needed 
for holomorphic induction (Subsection~\ref{subsec:5.1}) and 
define this concept in Subsection~\ref{subsec:5.2}. 
All this is combined with a homomorphism $\alpha \: \R \to \Aut(G)$ 
in Subsection~\ref{subsec:5.3}. 
If $G$ is a 
Banach--Lie group and $D \in \der(\g)$ is a bounded operator for 
which $0$ is isolated in $\Spec(D)$ and the norm on $\g$ is 
$\alpha$-invariant ($D$ is elliptic), then 
this leads to Theorem~\ref{thm:6.2}, the main result of this section.

\subsection{The geometric setup for holomorphic induction} 
\label{subsec:5.1}

Let $G$ be a Lie group with 
the smooth exponential function 
$\exp \:  \g \to G$. 
If, in addition, $G$ is analytic and
the exponential function is an analytic local diffeomorphism in $0$,
then $G$ is called a {\it BCH--Lie group} (for Baker--Campbell--Hausdorff).
Then the Lie algebra $\g$ is a
{\it BCH--Lie algebra}, i.e., there exists an open
$0$-neighborhood $U \subeq \g$ such that for $x,y \in U$ the Hausdorff series
$$x * y = x + y + \frac{1}{2}[x,y] + \cdots  $$
converges and defines an analytic function
$U \times U \to \g, (x,y) \mapsto x * y$.
The class of BCH--Lie groups contains in particular all Banach--Lie groups
(\cite[Prop.~IV.1.2]{Ne06}, \cite{GN}).

Let $G$ be a connected Fr\'echet--BCH--Lie group $G$ with Lie algebra $\g$. 
We further assume that there 
exists a complex BCH--Lie group $G_\C$ with Lie algebra $\g_\C$ and a natural 
map $\eta \: G \to G_\C$ for which $\Lie(\eta)$ is the inclusion $\g \into \g_\C$. 
Assume that $G^0 \subeq G$ is a Lie subgroup for which $M := G/G^0$ 
carries the structure of a smooth manifold 
with a smooth $G$-action. 
We also assume the existence of closed $\Ad(G^0)$-invariant complex subalgebras 
$\fg_\C^\pm \subeq \g_\C$ for which the complex conjugation \break 
$\sigma_\g(x+ iy) = x -iy$ on $\g_\C$ satisfies 
 $\sigma_\g(\fg_\C^\pm) = \fg_\C^\mp$ and  we have a topological direct 
sum decomposition 
\begin{equation}
  \label{eq:splitcondb}
 \g_\C = \fg_\C^+ \oplus \fg^0_\C \oplus \fg_\C^-\tag{SC}
\end{equation}
(cf.~(L3/4) in Section~\ref{sec:3}).
We put 
\[ \fq := \fg_\C^+ \rtimes \fg^0_\C \quad \mbox{ and } \quad 
\fp := \g \cap (\fg_\C^+ \oplus \fg_\C^-),\]
so that $\g = \g^0 \oplus \fp$ is a topological direct sum. 
We assume that there exist open symmetric convex $0$-neighborhoods 
\[ U_{\g_\C} \subeq \g_\C, \quad 
U_\fp \subeq \fp \cap U_{\g_\C}, \quad U_{\g^0} \subeq \g^0 \cap U_{\g_\C}, \quad 
U_{\fg_\C^\pm} \subeq\fg_\C^\pm  \cap U_{\g_\C}, \quad 
\mbox{ and  }\quad U_\fq \subeq \fq \cap U_{\g_\C} \] 
 such that the 
BCH-product is defined and holomorphic on 
$U_{\g_\C} \times U_{\g_\C}$, and the following maps are analytic diffeomorphisms onto an open subset: 
\begin{description}
\item[\rm(A1)] $U_{\fp} \times U_{\g^0} \to \g, (x,y) \mapsto x * y$. 
\item[\rm(A2)] $U_{\fp} \times U_\fq \to \g_\C, (x,y) \mapsto x * y$. 
\item[\rm(A3)] $U_{\g_\C^-} \times U_\fq \to \g_\C, (x,y) \mapsto x * y$. 
\end{description}

\begin{exs} \label{ex:c.4a} 
\nin (a) If $\dim G < \infty$ and $D$ is elliptic, 
then all these assumptions are satisfied 
for the positive/negative spectral subspaces $\g_\C^\pm$ of the derivation 
$i\cdot D$ of~$\g_\C$.  

\nin (b) Let $G$ be  a simply connected Banach--Lie group for which $\g_\C$ also is the 
Lie algebra of a Banach--Lie group and $M = G/G^0$ is a Banach homogeneous space. 
If the subalgebras $\g_\C^\pm \subeq \g_\C$ 
satisfy the splitting condition \eqref{eq:splitcondb}, 
then (A1-3) follow directly from the Inverse Function Theorem. This is the context of \cite{Ne13}.
\end{exs}

\begin{rem} \label{rem:c.4b} (From Banach to Fr\'echet) \\ 
Let $G_B$ be a Banach--Lie group with Lie algebra $\g_B$, $G_B^0 \subeq G_B$ and 
$M_B = G_B/G^0_B$ be as above (cf.\ Example~\ref{ex:c.4a}(b)). 
We assume that the splitting condition 
\eqref{eq:splitcondb} is satisfied. 
In addition, let $\beta \: \R \to \Aut(G_B)$ be a one-parameter group of automorphisms defining a continuous 
$\R$-action on $G_B$ and assume that the subalgebras 
$\g_{B,\C}^\pm$, $\fq_B$ and $\g^0$ are $\beta$-invariant. Then the subgroup 
\[ G := \{ g \in G_B \: \beta^g \: \R \to G_B,t \mapsto \beta_t(g)\ \mbox{ is smooth} \} \] 
of $G_B$ carries the structure of a Fr\'echet--BCH--Lie group with 
Lie algebra 
\[ \g := \{ x \in \g_B \: \R \to \g_B, t \mapsto \Lie(\beta_t)x \ \mbox{ is smooth} \}, \] 
the Fr\'echet space of smooth vectors for the continuous $\R$-action on the Banach--Lie algebra $\g_B$. 
Likewise $G^0 := G \cap G_B^0$ is a  Lie subgroup of $G$ for which 
$M := G/G^0$ is a smooth manifold consisting of the elements of $M_B = G_B/G^0_B$ with smooth 
orbit maps with respect to the one-parameter group of diffeomorphisms induced by 
$\beta$ via $\beta^M_t(gG_B^0) = \beta_t(g)G_B^0$. 

Since the automorphisms $\Lie(\beta_t)$ of $\g$ resp., $\g_\C$ are compatible with the 
BCH multiplication, it is easy to see with 
\cite[Lemma~C.5]{Ne14b} 
that conditions (A1-3) are inherited by the closed Fr\'echet subalgebras 
\[ \fg^0 = \fg^0_B \cap \g, \quad \g_\C^\pm = (\g_\C^\pm)_B \cap \g_\C \quad \mbox{ and } \quad 
\fq = \fq_B \cap \g_\C.\] 
\end{rem}

\subsection{Holomorphically induced representations} 
\label{subsec:5.2}

Condition (A1) implies the existence of a smooth manifold structure on $M = G/G^0$ 
for which $G$ acts 
analytically. 
Condition (A2) implies the existence of a complex manifold structure on $M$ which is $G$-invariant 
and for which the complex structure on the tangent space $T_{e G^0}(M) 
\cong \g/\g^0$ of 
$M$ in the base point $eG^0$ is determined by the identification with 
$\g_\C/\fq$. Finally, 
(A3) makes the proof of \cite[Thm.~2.6]{Ne13} work, so that we can associate to every 
bounded unitary representation $(\rho,\cV)$ of $G^0$ 
a holomorphic Hilbert bundle $\bV := G \times_{G^0} \cV$ over the complex  $G$-manifold $M$ 
by defining $\beta \: \fq \to \gl(\cV)$ by $\beta(\g_\C^+) = \{0\}$ and $\beta\res_{\g^0} = \dd\rho$.  

\begin{defn}  \label{def:d.1} 
We write $\Gamma(\bV)$ for the space of holomorphic sections 
of the holomorphic Hilbert bundle $\bV \to M = G/G^0$ on which the group $G$ acts by 
holomorphic bundle automorphisms. 
A unitary representation $(\pi, \cH)$ of $G$ is said to be 
{\it holomorphically induced from $(\rho,\cV)$} 
if there exists a $G$-equivariant linear injection 
$\Psi \: \cH \to \Gamma(\bV)$ such that the 
adjoint of the evaluation map $\ev_{e G^0} \: \cH \to \cV = \V_{e G^0}$ 
defines an isometric embedding $\ev_{e G^0}^* \: \cV \into \cH$. 
If a unitary representation $(\pi, \cH)$ holomorphically induced 
from $(\rho,\cV)$ exists, then it is uniquely determined 
(\cite[Def.~3.10]{Ne13}) and we call $(\rho,\cV)$ {\it (holomorphically)  
inducible}. 

The concept of holomorphic 
inducibility involves a choice of sign. 
Replacing $\g_\C^+$ by $\g_\C^-$ changes the complex structure on $G/G^0$ 
and thus leads to a different class of holomorphically inducible 
representations of $G^0$. 
\end{defn}

The following two theorems contain key information on 
holomorphically induced representations. The first one 
describes properties of holomorphically induced representations 
and the second one provides a sufficient criterion for a representation 
to be holomorphically induced.

\begin{thm} \label{thm:c.1} {\rm(\cite[Thm.~C.2]{Ne14b})} 
Assume {\rm(A1-3)}. If the unitary representation $(\pi, \cH)$ of $G$ is holomorphically 
induced from the bounded $G^0$-representation $(\rho,\cV)$, then the 
following assertions hold: 
\begin{description}
\item[\rm(i)] $\cV \subeq \cH^\omega$ consists of analytic vectors, i.e., 
their orbit maps $G \to \cH$ are real-analytic. 
\item[\rm(ii)] $R \: \pi(G)' \to \rho(G^0)', A \mapsto A\res_\cV$ is an isomorphism of 
von Neumann algebras. 
\item[\rm(iii)] $\dd\pi(\g_\C^-) \cV = \{0\}$. 
\end{description}
\end{thm}

\begin{prf} (i) follows from \cite[Lemma~3.5]{Ne13} 
and (ii) from \cite[Thm.~3.12]{Ne13}. 
Further (iii) follows from Equation (1) in the discussion 
preceding Theorem~3.12 in \cite{Ne13}. 
\end{prf}

\begin{thm} {\rm(\cite[Thm.~3.17]{Ne13})} \label{thm:c.3}
Suppose that $(U,\cH)$ is a unitary representation of $G$ and 
$\cV \subeq \cH$ is a $G^0$-invariant closed subspace such that 
\begin{description}
\item[\rm(HI1)] The representation $(\rho,\cV)$ of $G^0$ on $\cV$ is bounded. 
\item[\rm(HI2)] $\cV \cap (\cH^\infty)^{\g_\C^-}$ is dense in $\cV$. 
\item[\rm(HI3)] $\lbr \pi(G)\cV\rbr = \cH$. 
\end{description}
Then $(\pi, \cH)$ is holomorphically induced from $(\rho,\cV)$. 
\end{thm}

\subsection{The setting with $\alpha$} 
\label{subsec:5.3}

Now let $\alpha \: \R \to \Aut(G)$ define a smooth $\R$-action on~$G$ 
for which the action on the Lie algebra $\g$ is polynomially bounded 
and write $D \in \der(\g)$ for the infinitesimal generator of~$\alpha$ 
(see (L1/2) in Section~\ref{sec:3}). We assume, 
in addition to (A1-3) in Subsection~\ref{subsec:5.1},  that 
$\alpha$ and the subspace $\g_+$ are linked by the requirement that 
\begin{equation}
  \label{eq:gcpluscond}
 \g_\C^+ = \oline{\bigcup_{\delta > 0} \g_\C([\delta,\infty))},
\end{equation}
where $\g_\C([\delta,\infty))$ is the Arveson spectral subspace 
for the one-parameter group $(\alpha^\g_t)_{t \in \R}$ on $\g_\C$ 
(cf.\ Appendix~\ref{app:arv}). 
Applying Proposition~\ref{prop:spec-add} to the Lie bracket 
$\g_\C \times \g_\C \to \g_\C$, we see that 
$\g_\C^+$ is a closed complex subalgebra. 
For $f \in \cS(\R)$, $\alpha^\g(f) := \int_\R f(t) \alpha^\g_t\, dt$ 
and $z \in \g_\C$, 
the relations $\oline{\alpha^\g(f)z} = \alpha^\g(\oline f)\oline z$ and 
the relation $\hat{\oline f}(\xi) = \oline{\hat f(-\xi)}$ for the 
Fourier transform 
$\hat f(\xi) = \int_\R e^{i\xi x} f(x)\, dx$ imply that 
\begin{equation}
  \label{eq:gc-}
 \g_\C^- := \sigma_\g(\g_\C^+) =
 \oline{\bigcup_{\delta > 0} \g_\C((-\infty, -\delta])},
\end{equation}
where $\sigma_\g(x+iy) = x-iy$ is complex conjugation on $\g_\C$.

\begin{ex}
\label{ex:c.4ab} (a) Suppose that $G$ is a Banach--Lie group 
and consider an element $\bd \in \g$ for which 
the one-parameter group $e^{\R \ad \bd} \subeq \Aut(\g)$ is bounded, 
i.e., preserves an equivalent norm. We call such elements, resp., 
the corresponding derivation $D = \ad \bd$  
{\it elliptic}. Then 
\[ G^0 = Z_G(\exp \R \bd) = Z_G(\bd) 
= \{  g \in G \: \Ad(g)\bd = \bd \} \] 
is a closed subgroup of $G$, not necessarily connected, 
 with Lie algebra $\g^0 = \fz_\g(\bd) = \ker(\ad \bd)$. 
Since $\g$ contains arbitrarily small $e^{\R \ad \bd}$-invariant 
$0$-neighborhoods $U$, there exists such an open $0$-neighborhood 
with $\exp_G(U) \cap G^0 = \exp_G(U \cap \g^0).$ 
Therefore $G^0$ is a Lie subgroup of $G$, i.e., a Banach--Lie group 
for which the inclusion $G^0 \into G$ is a topological embedding. 

Our assumption implies that $\alpha^\g_t := e^{t\ad \bd}$ defines an 
equicontinuous one-parameter group of automorphisms of the 
complex Banach--Lie algebra $\g_\C$. For 
$\delta > 0$, we consider the Arveson spectral subspace 
\[ \g_\C^\delta := \g_\C([\delta,\infty[). \] 
By Lemma~\ref{lem:a.17}, the splitting condition 
\begin{equation}
  \label{eq:deltasplit}
  \g_\C = \fg_\C^\delta \oplus \fg^0_\C \oplus \sigma_\g(\fg_\C^\delta) 
\end{equation}
is satisfied for some 
$\delta > 0$ if and only if $0$ is isolated in $\Spec(\ad \bd)$. 

Since $\Ad(G^0)$ commutes with $e^{\R \ad \bd}$, 
the closed subalgebras $\g_\C^\pm \subeq \g_\C$ are invariant 
under $\Ad(G^0)$ and $e^{\R \ad \bd}$. 
Now $\fp := \g \cap (\g_\C^+ \oplus \g_\C^-)$ is a closed complement for 
$\fg^0$ in $\g$, so that $M := G/G^0$ carries the structure of a 
Banach homogeneous space and 
$\fq := \fg^0_\C + \g_\C^+ \cong \g_\C^+ \rtimes \fg^0_\C$ 
defines a $G$-invariant complex manifold 
structure on $M$. 

\nin (b) Let $\g$ be a real Hilbert--Lie algebra, i.e., 
$\g$ is a Lie algebra and a real Hilbert space, the Lie bracket 
is continuous and the operators $\ad x$, $x \in \g$, are skw-symmetric. 
Then one can use
 spectral measures to see that $\g_\C^+$ is the 
spectral subspace corresponding to the open interval 
$(0,\infty)$ (cf.\ Lemma~\ref{lem:c.1}), so that the splitting condition 
\begin{equation}
  \label{eq:stricta}
 \g_\C = \fg_\C^+ \oplus \fg^0_\C \oplus \fg_\C^-
\end{equation}
is satisfied. In particular, 
$0$ need not be isolated in the spectrum 
of $\ad \bd$ (\cite[Prop.~5.4]{BRT07}). 
\end{ex}

The following results are of key importance for the following. 
It contains the main consequences of Arveson's spectral theory  
for the $\R$-actions on $\g_\C$ and $\cH^\infty$.

\begin{prop} \label{prop:c.3s}{\rm(A strictness criterion)} 
Suppose that \eqref{eq:stricta} holds, and that 
$(\pi, \cH)$ is a smooth ground state representation 
of $(G,\alpha)$, i.e., $\cH^{0,\infty}$ is dense in $\cH^0$, and that 
the representation $(\pi^0,\cH^0)$ of $G^0$ is bounded. 
Then $(\pi, \cH)$ is holomorphically induced from $(\pi^0, \cH^0)$ and 
in particular strict. 
\end{prop}

\begin{prf} First, 
Theorem~\ref{thm:specrel} implies that 
$\dd\pi(\g_\C^-)\cH^{0,\infty} = \{0\}$. 
Applying Theorem~\ref{thm:c.3} to $\cV := \cH^0$, we see that 
$(\pi, \cH)$ is holomorphically induced from $(\pi^0, \cH^0)$, 
and Theorem~\ref{thm:c.1}(ii) implies strictness. 
\end{prf}

In the Banach case we can formulate more concrete sufficient 
conditions for strictness:

\begin{thm} \label{thm:6.2}
Suppose that $G$ is Banach and $\bd \in \g$ is elliptic with 
$0$ isolated in $\Spec(\ad \bd)$ and 
$\g_\C^- = \g_\C(]-\infty, -\delta])$ for some $\delta > 0$. 
Then the following assertions hold for any smooth 
representation $(\pi, \cH)$ for which $-i\partial \pi(\bd)$ is bounded 
from below. 
\begin{itemize}
\item[\rm(a)] The $G^0$-invariant subspace 
$\cV := \oline{(\cH^\infty)^{\g_\C^-}}$ satisfies 
$\cH = \lbr \pi(G)\cV)\rbr$. 
\item[\rm(b)] If the $G^0$-representation $\rho(h) := \pi(h)\res_\cV$ 
on $\cV$ is bounded, then $(\pi, \cH)$ is holomorphically induced from 
the representation $\rho$ of $G^0$ on $\cV$, 
$\pi$ is semibounded, and $\bd \in W_\pi^0$ 
(see \eqref{eq:wpi-intro}). 
In particular, $\cV$ consists of analytic vectors. 
\item[\rm(c)] In addition to {\rm(b)}, suppose that 
$-i\partial \rho(\bd) \geq m\1$ for $m \in \R$.
Then $-i\partial \pi(\bd) \geq m\1$ and the 
associated minimal positive energy representation of 
$(G,\alpha)$ for $\alpha_t(g) = \exp(t\bd) g \exp(-t\bd)$ 
is strict with $\cH^0 = \cV$. 
\end{itemize}
\end{thm} 

\begin{prf} (a) and (b) follow from \cite[Thms.~4.7, 4.14]{Ne13}. 
It remains to prove (c). Let $P$ denote the spectral measure of 
$-i\partial \pi(\bd)$. Our assumption implies that 
$\cV \subeq P([m,\infty))$. As $\cV$ consists of analytic vectors, 
$\dd\pi(U(\g))\cV$ is dense in $\cH$. Since $\cV$ is annihilated 
by $\g_\C^-$, we have 
$\dd\pi(U(\g))\cV = \dd\pi(U(\g_\C^+))\cV$ by the 
Poincar\'e--Birkhoff--Witt Theorem. 
Finally, we observe that 
the Spectral Translation Formula (Theorem~\ref{thm:spectrans}) implies that 
\begin{equation}
  \label{eq:nulldelta}
\dd\pi(U(\g_\C^+))\cV \subeq \cV + \cH^\infty([\delta, \infty)) 
\subeq \cH^\infty([m,\infty)). 
\end{equation}
We conclude that $\cH^\infty([m,\infty))$ is dense in $\cH$, i.e., 
that $-i\partial\pi(\bd) \geq m\1$. 

Since $\bd$ is central in $\g^0$ and $\pi$ is 
holomorphically induced from $\pi^0$, there exists a 
unitary one-parameter group $(U_t)_{t \in \R}$ in the commutant 
$\pi(G)'$ with $U_t\res_{\cV} = \pi^0(\exp t \bd)$ 
(Theorem~\ref{thm:c.1}). 
As the restriction to $\cV$ is an isomorphism 
of von Neumann algebras, $U_t = e^{itB}$, where $B \geq 0$ is bounded. 
Now $W_t := \pi(\exp t \bd) U_t^{-1}$ defines a unitary one-parameter group 
implementing the same automorphisms as $\exp(t\bd)$ for which 
$\cV$ consists of fixed points. Now the same argument as in (a), applied 
to the extended Lie algebra $\g \rtimes \R \bd$ and the representation 
of $G \rtimes_\alpha \R$ by $\tilde\pi(g,t) = \pi(g) W_t$ 
implies that $(W_t)_{t \in\R}$ has positive spectrum 
and fixed point space $\cH^W = \cV$ 
(see \eqref{eq:nulldelta}). 
As $\pi(G)\cV$ is total in $\cH$, it follows that 
$(\pi, W,\cH)$ is a ground state representation, 
hence in particular minimal by Proposition~\ref{prop:swallow}. 
Theorem~\ref{thm:c.1} further implies that it is strict. 
\end{prf}

The preceding theorem does not assert that $(\pi,\cH)$ is a 
ground state representation, but we have the following corollary. 
It provides a sufficient condition for a bounded 
representation to be ground state. It applies in particular to finite 
dimensional unitary representations of compact Lie groups. 
By the strong boundedness assumptions, it follows immediately from 
Theorem~\ref{thm:6.2}.

\begin{cor} \label{cor:6.2} 
If $G$ is Banach and 
$\bd \in \g$ is elliptic with $0$ isolated in $\Spec(\ad \bd)$, 
then every bounded representation of $G$  
is a strict ground state representation 
of $(G,\alpha)$ for $\alpha_t(g) = \exp(t\bd) g \exp(-t\bd)$, 
where $\cH^0 = \cH^{\g_\C^-}$. 
\end{cor}

The following theorem shows that, assuming that $(\pi, \cH)$ is semibounded 
with $\bd \in W_\pi^0$ permits us to get rid of the quite implicit assumption 
that the $G^0$-representation on $\cV$ is bounded. It is an important 
generalization of Corollary~\ref{cor:6.2} to semibounded representations. 

\begin{thm} \label{thm:6.2b} {\rm(\cite[Thm.~4.12]{Ne13})} 
Let $(\pi, \cH)$ be a semibounded 
unitary representation of the Banach--Lie group $G$ and let 
$\bd \in W_\pi^0$ be an elliptic element for which $0$ 
is isolated in $\Spec(\ad \bd)$ and 
$\g_\C^- = \g_\C(]-\infty, -\delta])$ for some $\delta > 0$.  
We write $P \: \fB(\R) \to B(\cH)$ for the spectral measure of the 
unitary one-parameter group $\pi_\bd(t) := \pi(\exp(t\bd))$. 
Then the following assertions hold: 
\begin{itemize}
\item[\rm(i)] The representation $\pi\res_{G^0}$ of $G^0$ is semibounded and, 
for each bounded measurable subset $B \subeq \R$, the 
$G^0$-representation on $P(B)\cH$ is bounded. 
\item[\rm(ii)] The representation $(\pi, \cH)$ is a direct sum of 
representations $(\pi_j, \cH_j)_{j \in J}$ for which there exist $G^0$-invariant 
subspaces $\cD_j \subeq (\cH_j^\infty)^{\g_\C^-}$ for 
which the $G^0$-representation $\rho_j$ on $\cV_j := \oline{\cD_j}$ is bounded and 
$\lbr \pi_j(G)\cV_j\rbr = \cH_j$. 
Then the representations $(\pi_j, \cH_j)$ are holomorphically 
induced from $(\rho_j, \cV)$. 
\item[\rm(iii)] If $(\pi, \cH)$ is irreducible and 
$m := \inf\Spec(-i\partial\pi(\bd))$, then $P(\{m\})\cH 
= \oline{(\cH^\infty)^{\g_\C^-}}$ and 
$(\pi, \cH)$ is holomorphically induced 
from the bounded $G^0$-representation $\rho$ on this space. 
\end{itemize}
\end{thm}

\begin{cor} In the context of {\rm~Theorem~\ref{thm:6.2b}}, 
$(\pi, \cH)$ is a ground state representation and the 
direct summands $(\pi_j,\cH_j)_{j \in J}$ are strict ground state 
representations. 
\end{cor}

\begin{prf} First we use Theorem~\ref{thm:6.2}(c)
to see that all 
representations $(\pi_j,\cH_j)_{j \in J}$ are strict ground state 
representations. Now Lemma~\ref{lem:dirsum} shows that their 
direct sum $(\pi,\cH)$ is also a ground state representation.  
\end{prf}

\begin{prob}
  Are direct sums of strict ground state representations always strict?
\end{prob}

\section{Finite dimensional groups} 
\label{sec:6}

In this section we assume that $G$ is finite dimensional, so that 
$G^\flat$ is finite dimensional as well. Replacing $G$ by $G^\flat$, 
it suffices to consider the inner case, i.e., for some fixed 
$\bd \in \g$, we are interested in unitary representations 
$(\pi, \cH)$ for which $-i \partial \pi(\bd)$ is bounded from below 
(cf.~Proposition~\ref{prop:inner}). 
Let $I_\pi \subeq \g^*$ denote the momentum set of $\pi$ 
(Definition~\ref{def:momset}). 

\subsection{Generalities} 

From \eqref{eq:wpi-intro}  we recall the that 
\begin{equation}
  \label{eq:wpi} 
W_\pi = \{ x \in \g \: -i \partial \pi(x) \ \mbox{ bounded below } \} 
= \{ x \in \g \: \inf I_\pi(x) > -\infty\} 
\end{equation}
is a convex cone in $\g$, invariant under $\Ad(G)$. 
It contains the positive cone of $\pi$: 
\begin{equation}
  \label{eq:cpi}
C_\pi = \{ x \in \g \: -i \partial \pi(x) \geq 0  \}.
\end{equation}
  
We assume that $\pi$ has discrete kernel, i.e., that the positive cone 
$C_\pi$ of $\pi$ is pointed. The linear subspace 
\[ \g_\pi := W_\pi - W_\pi \trile \g \]  is an ideal 
because $W_\pi$ is an $\Ad(G)$-invariant convex cone,  
and the restriction of $\pi$ to the corresponding normal subgroup 
$G_\pi \trile G$ is semibounded. Semibounded representations 
of finite dimensional groups 
have been studied in detail and classified in \cite{Ne00}. As $\bd \in W_\pi 
\subeq \g_\pi$, 
we may further assume that $\g = \g_\pi$ and restrict our discussion to 
semibounded representations. With \cite[Thm.~XI.6.14]{Ne00} on the 
existence of a direct integral decompositions, many assertions 
can be reduced to the case of irreducible representations. 
This leaves us with the situation where: 
\begin{itemize}
\item[\rm(F1)] $\ker(\pi)$ is discrete, so that 
the cone $C_\pi$ is pointed. 
\item[\rm(F2)] $\bd \in W_\pi$ and $\pi$ is semibounded. 
\item[\rm(F3)] $\pi$ is irreducible.
\end{itemize}

\begin{rem} From $\bd \in W_\pi$ it follows that $\Spec(\ad \bd) \subeq i \R$ 
by \cite[Prop.~VII.3.4]{Ne00} because 
discreteness of the kernel is equivalent to $I_\pi$ spanning~$\g^*$.
As a consequence of the Jordan decomposition, 
this implies that the one-parameter group 
$e^{\R \ad \bd} \subeq \Aut(\g)$ is polynomially bounded 
because its semisimple component is bounded and its unipotent component  is 
polynomially bounded (Remark~\ref{rem:c-polybound}).
%
\end{rem}

\begin{defn} \label{def:4.3} 
A maximal abelian subspace $\ft \subeq \g$ is called a 
{\it compactly embedded Cartan subalgebra} if the closure of 
$e^{\ad \ft} \subeq \Aut(\g)$ is compact.
Let $\ft \subeq \g$ be a compactly embedded Cartan subalgebra 
and $\g_\C = \ft_\C \oplus \bigoplus_{\alpha \in \Delta} \g_\C^\alpha$ the corresponding 
root decomposition, where 
\[ \g_\C^\alpha 
= \{ y \in \g_\C \: (\forall x \in \ft)\ [x,y] = \alpha(x) y \} 
\quad \mbox{ and } \quad 
\Delta = \{ \alpha \in i \ft^* \: \g_\C^\alpha \not=\{0\}\}.\] 

The elements of $\Delta$ are called {\it roots}. 
We call a root $\alpha \in \Delta$ 
  \begin{itemize}
  \item {\it compact} if there exists an $x_\alpha \in \g_\C^\alpha$ with 
$\alpha([x_\alpha, x_\alpha^*]) > 0$ and write 
$\Delta_k\subeq \Delta$ for the set of compact roots. 
  \item {\it non-compact} if 
there exists a non-zero $x_\alpha \in \g_\C^\alpha$ with 
$\alpha([x_\alpha, x_\alpha^*]) \leq 0$ and write 
$\Delta_p\subeq \Delta$ for the set of non-compact roots. 
  \end{itemize}
Then $\dim \g_\C^\alpha = 1$ for $\alpha \in \Delta_k$ 
and there exists a unique element 
$\alpha^\vee \in [\g_\C^\alpha, \g_\C^{-\alpha}]$ 
with $\alpha(\alpha^\vee) = 2$. 
The reflections $r_\alpha \: \ft \to \ft, r_\alpha(x) = x - \alpha(x)  \alpha^\vee$ 
for $\alpha \in \Delta_k$ generate the {\it Weyl group} $\cW$. 
\end{defn}

\begin{lem} \label{lem:jac-mor}  
Let $x$ be an element of the semisimple real Lie algebra 
$\g$ and $x = x_s + x_n$ be its Jordan decomposition, where $x_s$ is semisimple 
and~$x_n$ is nilpotent. Then the adjoint orbit of $x$ contains 
all elements of the form $x_s + t x_n$, $t > 0$. 
\end{lem}

\begin{prf} Since the Jordan decomposition and the adjoint orbit 
of $x$ adapts to the decomposition of $\g$ into simple ideals, 
we may w.l.o.g.\ assume that $\g$ is simple. 

Let $\fq = \fl \ltimes \fu\subeq \g$ denote the Jacobson--Morozov 
parabolic associated to the nilpotent element~$x_n$ (\cite{HNO94}). Then 
$x_s \in \ker(\ad x_n) \subeq \fq$ implies that $x_s \in \fq$. 
As $x_s$ is semisimple, it is conjugate under the group 
of inner automorphisms of 
$\fq$ to an element of $\fl$.\begin{footnote}{Every algebraic subgroup 
$G \subeq \GL(V)$, $V$ a finite dimensional real vector space,
 is a semidirect product 
$G \cong U \rtimes L$, where $U$ is unipotent and $L$ is reductive. 
Moreover, for every reductive  subgroup $L_1 \subeq G$ there exists an 
element $g \in G$ with $gL_1 g^{-1} \subeq L$ 
(\cite[Thm.~VIII.4.3]{Ho81}).}
\end{footnote}
By the Jacobson--Morozov Theorem, $\fl$ contains a semisimple element $h$ with 
$[h,x_n] = 2 x_n$ and $h \in [x_n,\g]$. In terms of this element, we have 
$\fq = \sum_{n \geq 0} \g_n(\ad h)$ and $\fl = \ker(\ad h)$. 
We conclude that $[h,x_s] =0$, so that 
$e^{t\ad h} x = x_s + e^{2t} x_n$ for $t \in \R$.  
\end{prf}

We now come to the main result of this section. 
For its proof we shall use 
the  following theorem (\cite[Thm.~1.1]{Mo80}), which is a formidable 
tool to exclude that certain Lie algebra elements have ground states. 

\begin{thm} {\rm(Moore's Eigenvector Theorem)} 
Let $G$ be a connected finite dimensional Lie group with Lie algebra 
$\g$ and $x \in \g$. Further, let $\fn_x \trile \g$ be the 
smallest ideal of $\g$ such that the induced operator 
$\ad_{\g/\fn_x} x$ on the quotient Lie algebra $\g/\fn_x$ 
is elliptic, i.e., semisimple with purely imaginary spectrum. 
Suppose that $(\pi, \cH)$ is a continuous unitary representation of $G$ and 
$v\in \cH$ an eigenvector for the one-parameter group  $\pi(\exp \R x)$. 
Then 
\begin{itemize}
\item[\rm(a)] $v$ is fixed by the normal subgroup 
$N_x := \la \exp \fn_x \ra \trile G$, and 
\item[\rm(b)] the restriction of 
$i\partial \pi(x)$ to the orthogonal complement of the space 
$\cH^{N_x}$ of $N_x$-fixed vectors has absolutely continuous spectrum. 
\end{itemize}
\end{thm}

\begin{cor} \label{cor:moore-eigenvector} 
Let $G$ be a connected finite dimensional Lie group. 
Suppose that $(\pi, \cH)$ is an irreducible 
unitary representation of $G$ with discrete kernel and 
that $\bd \in \g$ is such that $\partial \pi(\bd)$ has an eigenvector in $\cH$. 
Then $\ad(\bd)$ is elliptic. 
\end{cor}

\begin{prf} 
Suppose that $v$ is an eigenvector of $\partial \pi(\bd)$. 
Then Moore's Eigenvector Theorem implies that the normal subgroup 
$N_\bd \trile G$ fixes $v$. As $N_\bd$ is normal, the subspace 
$\cH^{N_\bd}$ of $N_\bd$-fixed vectors is a $G$-subrepresentation, 
hence coincides with $\cH$ by irreducibility. As $\ker(\pi)$ is discrete, 
$\fn_\bd = \{0\}$, and this means that $\ad \bd$ is elliptic. 
\end{prf}

\begin{thm} \label{thm:4.3} 
Let $(\pi, \cH)$ be an irreducible semibounded representation 
with discrete kernel of the finite dimensional connected Lie group $G$ 
and let $\bd \in W_\pi$. Then the following assertions hold: 
\begin{itemize}
\item[\rm(i)] There exists a compactly embedded Cartan subalgebra  $\ft \subeq \g$ 
and a $\cW$-invariant positive system $\Delta_p^+$ of non-compact roots 
such that 
\[ \oline{W_\pi} \cap \ft = C_{\rm max} := (i\Delta_p^+)^\star 
:= \{ x \in \ft \: (\forall \alpha \in \Delta_p^+)\, 
i\alpha(x) \geq 0\} \] 
and $W_{\rm max} := \oline{\Ad(G)C_{\rm max}}$ coincides with $\oline{W_\pi}$.
\item[\rm(ii)] $\bd$ has smooth ground states if and only if 
$\ad(\bd)$ is elliptic. 
\item[\rm(iii)] If $\g$ is simple hermitian, then $W_\pi = W_{\rm max}$, 
and in particular $W_\pi$ is closed. 
\end{itemize}
\end{thm}

\begin{prf} (i) From \cite[Thm~X.4.1]{Ne00} and its proof we know that the 
set of extreme points $\Ext(I_\pi)$ is a coadjoint orbit 
$\cO_{\lambda}$ for some $\lambda \in \ft^*$, that $I_\pi = \conv(\cO_\lambda)$ 
and that there exists a $\cW$-invariant positive system $\Delta_p^+$ such that 
\[ W_\pi \cap \ft = C_{\rm max}. \] 
As $W_\pi^0$ is an elliptic cone, i.e., its interior 
consist of elliptic elements (\cite[Prop.~VII.3.4(c)]{Ne00}), it follows that 
\begin{equation}
  \label{eq:wpiclos}
 \oline{W_\pi} = \oline{\Ad(G)C_{\rm max}} = W_{\rm max}.
\end{equation}

\nin (ii) As in (i), we derive from 
\cite[Prop.~VII.3.4(c)]{Ne00} that 
$W_\pi^0 = \Ad(G) (W_\pi^0 \cap \ft)$ because $W_\pi^0$ is elliptic. 
A ground state vector for $\bd$ exists if and only if 
the minimal spectral value 
\[ m := \inf \Spec(-i\partial \pi(\bd)) \] 
is an eigenvalue. If this is the case, 
then Corollary~\ref{cor:moore-eigenvector} implies that 
$\ad \bd$ is elliptic. 

Now we assume, conversely, that $\bd \in W_\pi$ is elliptic. 
Then $\Ad(G)\bd$ intersects the Cartan subalgebra $\ft$ (\cite[Thm.~VII.1.4(vi)]{Ne00}). As 
$\bd \in W_\pi$ and $W_\pi \cap \ft = C_{\rm max}$, 
it follows that $\bd$ is $\Ad(G)$-conjugate to an element $\bd' \in C_{\rm max}$. 
Therefore the $\ft$-weight decomposition of $\cH_\lambda$ 
implies that the $m$-eigenspace of $-i\partial \pi(\bd')$ 
contains a lowest weight vector (which is smooth), 
so that the minimal eigenspace 
$\cH_\lambda^0$ contains a non-zero smooth vector. 

\nin (iii) Suppose that $\g$ is simple hermitian. 
To show that $W_{\rm max} \subeq W_\pi$, let 
$x \in W_{\rm max}$ and write $x = x_s + x_n$ for its Jordan decomposition. 
Then the adjoint orbit of $x$ contains all elements 
$x_s + t x_n$, $t > 0$ 
(Lemma~\ref{lem:jac-mor}), so that $x_s, x_n \in W_{\rm max}$. 
Since $x_n$ is nilpotent, we even have $x_n \in W_{\rm min}$ 
by \cite[Thm.~III.9]{HNO94}, and since $W_{\rm min} \subeq C_\pi$ 
(\cite[Thm.~X.4.1]{Ne00}), it follows that 
$x_n \in C_\pi$. As $x_s$ is elliptic, its adjoint orbit intersects~
$\ft \cap W_{\rm max} = C_{\rm max} \subeq W_\pi$, as we have seen in (i). 
This implies that every elliptic element in $W_{\rm max}$ is contained in $W_\pi$. 
We thus obtain that 
\[ x = x_s + x_n \in W_\pi + C_\pi \subeq  W_\pi,\] 
and hence that $W_{\rm max} \subeq W_\pi$, which implies equality by~\eqref{eq:wpiclos} above. 
\end{prf}

\begin{ex} In the context of the preceding theorem, we note that, 
in general  $W_\pi \not= W_{\rm max}$ because $W_\pi$ need not be closed. 
Let 
\[ \sigma(x,y) := \sum_{j = 1}^n x_j y_{n+j} - y_j x_{n+j} \] 
be the canonical symplectic form on $\R^{2n}$ and 
$\bd \in \sp_{2n}(\R)$ an element for which the corresponding 
Hamiltonian function $H_\bd(v) = \shalf\sigma(\bd v,v)$ is positive definite. 
For the oscillator representation $(\pi, L^2(\R^n))$ of 
$\g = \heis(\R^{2n},\sigma) \rtimes \R \bd$,  we have 
\[ \oline{W_\pi} = \heis(\R^{2n},\sigma) \oplus \R_+ \bd, \] 
a closed half space with boundary $\heis(\R^{2n},\sigma)$. 
An element of $\heis(\R^{2n},\sigma)$ corresponds to a semibounded operator if and only 
if it is central. Therefore $W_\pi$ is not closed, 
hence different from~$W_{\rm max}$. 
\end{ex} 

\begin{ex} (The case of simple Lie algebras) 
(a) Assume that $\g$ is simple and that 
$\alpha$ is non-trivial. As all derivations of $\g$ are inner, 
we have $\alpha_t(g) = \exp(t\bd) g \exp(-t\bd)$ for some 
non-zero $\bd \in \g$. If $G$ has a non-trivial positive energy representation, 
then $\g$ must be compact or hermitian (\cite[\S\S VII.2/3]{Ne00}). 
\begin{itemize}
\item In the compact case all irreducible 
representations $(\pi, \cH)$ are finite dimensional, 
so that all operators $\partial \pi(x)$ for $x \in \g$ are bounded. 
Hence the positive energy condition is satisfied for every $\bd \in \g$ 
and ground states exist (cf.\ Corollary~\ref{cor:6.2}).  
\item In the hermitian case, there is a closed convex cone 
$W_{\rm max} \subeq \g$, such that there exists a positive energy representation 
if and only if $\bd \in W_{\rm max} \cup - W_{\rm max}$ 
(Theorem~\ref{thm:4.3}(iii)).
If $\bd \in W_{\rm max}$, then 
every irreducible representation $(\pi, \cH)$ for which 
$-i\partial \pi(\bd)$ is bounded from below is semibounded 
because the cone $W_\pi= W_{\rm max}$ has interior points. 
By Theorem~\ref{thm:4.3}(ii), the existence of ground states is equivalent to 
$\bd$ being elliptic. 
\end{itemize}

\nin (b) In a hermitian Lie algebra, there exist two closed convex invariant cones 
$W_{\rm min} \subeq W_{\rm max}$ such that, for every non-trivial closed 
convex invariant cone $W \subeq \g$ we have 
\[ W_{\rm min} \subeq W \subeq W_{\rm max} \quad \mbox{ or } \quad 
 W_{\rm min} \subeq -W \subeq W_{\rm max}.\] 
If $W_{\rm min} = W_{\rm max}$, which is the case for $\g = \sp_{2n}(\R)$, 
this means that $W$ is unique up to sign. We therefore have 
\[ W_{\rm min} \subeq C_\pi \subeq W_\pi \subeq W_{\rm max} \] 
for every positive energy representation, and thus $C_\pi = W_\pi$. 

In general the two cones $C_\pi$ and $W_\pi$ are different. 
Concrete examples are easily found for $\g = \su_{1,2}(\C)$. 
\end{ex}

\begin{rem} For finite dimensional Lie groups $G$ 
the classification of irreducible 
semibounded unitary representations easily boils down to a situation 
where one can apply Theorem~\ref{thm:6.2b}. 
If the semibounded unitary representations of $G$ 
separate the points, then the set $\comp(\g)$ of elliptic elements in 
$\g$ has interior points (\cite[Prop.~VII.3.4(c)]{Ne00}) 
and we may choose $\bd$ as an interior point which, in addition, 
is a regular element of $\g$. 
Then $\ft := \g^0 := \ker(\ad \bd)$ is a 
compactly embedded Cartan subalgebra and 
the corresponding subgroup $T := \exp(\ft) = G^0$ is abelian. 
For a semibounded representation $(\pi, \cH)$ of $G$ with discrete 
kernel we now choose a $\cW$-invariant positive system 
$\Delta_p^+$ of non-compact roots such that $W_\pi \subeq W_{\rm max}$ 
and obtain $\bd \in C_{\rm max}^0$. As $\bd$ is elliptic, 
Theorem~\ref{thm:4.3} now implies that $\pi$ is holomorphically 
induced from $(\pi^0, \cH^0)$. Further, 
this representation is irreducible, and since $G^0 = T$ is abelian, 
Schur's Lemma leads to $\dim \cH^0 = 1$. 
This means that $\pi^0(\exp x) = e^{i\lambda(x)}$ for some 
$\lambda \in \ft^*$, the {\it lowest weight} of the representation~$\pi$ 
with respect to the positive system 
$\Delta^+ := \{ \alpha \in \Delta \: \alpha(-i\bd) > 0\}$. 
In this case $C_\alpha \subeq C_{\pi^0}$ is equivalent to 
\begin{equation}
  \label{eq:poscondcoroot}
 \lambda(\alpha^\vee) \leq 0 \quad \mbox{ for } \quad 
\alpha \in \Delta_k^+ \quad \mbox{ and } \quad 
 \lambda([x_\alpha^*, x_\alpha]) \geq 0 \quad \mbox{ for } \quad 
\alpha \in \Delta_p^+, x_\alpha \in \g_\C^\alpha. 
\tag{PC}\end{equation}
As Remark~\ref{rem:countex} below shows, these conditions 
are in general not sufficient. 
\end{rem}

From \cite[Cor.~14.3.10]{HN12} we recall the following 
fact on the connectedness of the group $G^0$.
\begin{lem}  \label{lem:4.1}
If $G$ is connected and $\bd$ is elliptic, 
then the subgroup $G^0 = Z_G(\bd)$ 
is connected.\end{lem}

\subsection{Application to compact Lie groups} 
\label{subsec:6.2}

Theorem~\ref{thm:compcase} below is the main result of the present subsection. 
It shows that all unitary representations of 
compact connected Lie groups are ground state representations for 
any continuous homomorphism $\alpha \: \R \to \Aut(G)$. The following lemma 
prepares the crucial information for its proof. 

We first recall the root decomposition 
of the compact Lie algebra $\g$ with with respect to a Cartan subalgebra 
$\ft\subeq \g$ (Definition~\ref{def:4.3}) 
containing a fixed element~$\bd$ that we may pick 
in $[\g,\g]$ as $\g = \z(\g) \oplus [\g,\g]$. We have 
\[ \g_\C = \ft_\C \oplus \bigoplus_{\alpha \in \Delta} \g_\C^\alpha, \]
and obtain with 
\[ \g_\C^\pm := \sum_{\pm\alpha(-i \bd) > 0} \g_\C^\alpha 
\quad \mbox{ and } \quad 
\g_\C^0 = \fz_\g(\bd)_\C 
= \ft_\C + \sum_{\alpha(-i \bd) = 0} \g_\C^\alpha = \fz_{\g_\C}(\bd)\] 
the triangular decomposition 
\[ \g_\C = \g_\C^+ \oplus \g_\C^0 \oplus \g_\C^- \] 
of $\g_\C$ with respect to $-i\ad \bd$
 (cf.~\cite[Ch.~IX.5]{Ne00}). 
We choose a positive system $\Delta^+ \subeq \Delta$ in such a way that  
\[ \Delta^+  \supeq \Delta^{++} := \{ \alpha \in \Delta \: 
\alpha(-i\bd) > 0\}.\] 
Then $\Delta^{0+} := \Delta^0 \cap \Delta^+$ is a positive system 
in the root system 
\begin{equation}
  \label{eq:rootplusplus}
 \Delta^0 := \{ \alpha \in \Delta \: \alpha(-i\bd) = 0\}
\quad \mbox{ of }\quad (\fg^0,\ft), \quad \mbox{ and } \quad 
\Delta^+ = \Delta^{0+} \dot\cup \Delta^{++}.
\end{equation}

\begin{prop} \label{prop:likeso3} 
Let $G$ be a connected  Lie group with  compact Lie algebra~$\g$ 
and derivation $D = \ad \bd$ for some $\bd \in [\g,\g]$. 
For an irreducible unitary representation 
$(\pi^0, \cH^0)$ of $G^0 := Z_G(\bd)$, the following conditions are 
equivalent: 
\begin{itemize}
\item[\rm(i)] $(\pi^0, \cH^0)$ is holomorphically inducible. 
\item[\rm(ii)] $\dd\pi^0([z^*,z]_0) \geq 0$ for every $z \in \g_\C^+$, 
i.e., $C_\alpha \subeq C_{\pi^0}$. 
\item[\rm(iii)] $\dd\pi^0(\alpha^\vee) \leq 0$ for every $\alpha \in \Delta^{++}$. 
\item[\rm(iv)] The lowest weight $\lambda$ of $(\pi^0,\cH^0)$ satisfies 
$\lambda(\alpha^\vee) \leq 0$ for every $\alpha \in \Delta^{++}$. 
\item[\rm(v)] The lowest weight $\lambda$ of $(\pi^0,\cH^0)$ is 
$\Delta^+$-antidominant. 
\end{itemize}
\end{prop}

\begin{prf} (i) $\Rarrow$ (ii): This follows from \cite[Lemma~6.11]{Ne12}, 
but we repeat the direct argument. Condition (i) yields a 
realization of $(\pi^0, \cH^0)$ as a subrepresentation of a 
unitary representation $(\pi, \cH)$ of $G$ 
such that $\cH^0$ consists of smooth vectors 
and $\dd\pi(\g_\C^-) \cH^0 = \{0\}$ 
(Theorem~\ref{thm:6.2}). For $z \in \g_\C^+$ we now have
$z^* \in \g_\C^-$, so that we obtain for 
$\xi \in \cH^0$ with 
\[ \la \xi, \dd\pi(z) \xi \ra 
= \la \xi, \dd\pi(z_0) \xi \ra  \] 
the inequality 
\begin{equation}
  \label{eq:normpos}
\la \xi, \dd\pi^0([z^*,z]_0) \xi \ra 
= \la \xi, \dd\pi([z^*,z]_0) \xi \ra 
= \la \xi, \dd\pi([z^*,z]) \xi \ra 
= \la \xi, \dd\pi(z^*) \dd\pi(z) \xi \ra 
=  \|\dd\pi(z) \xi\|^2 \geq 0.
\end{equation}
By \eqref{eq:sumpos} in Remark~\ref{rem:4.7}, 
the cone $C_\alpha \subeq \g^0$ is generated by the elements 
$i[z^*,z]_0$, $z \in \g_\C^+$. Therefore \eqref{eq:normpos} 
implies $C_\alpha \subeq C_{\pi^0}$. 
  
\nin (ii) $\Rarrow$ (iii): For $\alpha \in \Delta^{++}$, we pick 
$x_\alpha \in \g_\C^\alpha$ with $[x_\alpha, x_\alpha^*] =\alpha^\vee$ 
(cf.~Definition~\ref{def:4.3}). 
Then (iii) follows from (ii) and $-\alpha^\vee \in C_\alpha$. 

\nin (iii) $\Rarrow$ (iv): This follows from the fact that 
$\lambda(\alpha^\vee)$ 
is an eigenvalue of $\dd\pi^0(\alpha^\vee)$.  

\nin (iv) $\Rarrow$ (v): By (iv), we only have to observe that 
$\lambda(\alpha^\vee) \leq 0$ for $\alpha \in \Delta^{0,+}$ 
because $\lambda$ is the lowest weight of a unitary representation 
(cf.~\cite[Prop.~IX.1.21]{Ne00} for more details). 

\nin (v) $\Rarrow$ (i): If $\lambda$ is antidominant 
with respect to $\Delta^+$, then 
the corresponding unitary lowest weight representation 
$(\pi_\lambda, \cH_\lambda)$ of $G$ contains $(\pi^0,\cH^0)$ 
as the subrepresentation on the minimal eigenspace for 
$-i\partial\pi(\bd)$, hence is holomorphically induced from  $(\pi^0,\cH^0)$ 
(Theorem~\ref{thm:c.3}). That $\pi_\lambda$ exists follows from 
the fact that $\lambda$ actually integrates to a character of~$T \subeq G^0$.
\end{prf}

\begin{rem} The set $\hat T_+ \subeq \hat T$ of chose characters satisfying 
Condition (iv) is invariant under the Weyl group $\cW^0 := \cW(\g^0, \ft)$. Identifying 
the unitary dual $\hat G$ of $G$ with the orbit space $\hat T/\cW$ 
(Cartan--Weyl Theorem, \cite[p.~209]{Wa72}), it follows that 
the ground state representations for $(G,\alpha)$ correspond to the subset 
$\hat T_+/\cW^0$. 
\end{rem}

The above lemma characterizes the holomorphically inducible 
representations of $G^0$. The following proposition switches the perspective 
from $G^0$ to $G$: 

\begin{prop}\label{2ndstep} Let $G$ be a connected 
Lie group with compact Lie algebra $\g$ and 
derivation $D = \ad \bd$ for some $\bd \in \g$.   
Then every irreducible unitary representations  $(\pi,\cH)$ of $G$ 
is holomorphically induced from the irreducible 
representations $(\pi^0, \cH^0)$ of $G^0$ on the minimal 
eigenspace of $-i \partial\pi(\bd)$ and it is a strict ground state 
representation. 
\end{prop}

\begin{proof} Since $\cH$ is finite dimensional, 
$-i \partial\pi(\bd)$ has an eigenspace $\cH^0$ for a minimal eigenvalue. 
This space carries  an irreducible representation 
of the connected group $G^0$ (Lemma~\ref{lem:4.1}). 
This follows easily from $U(\g_\C) = U(\g_\C^+) U(\g_\C^0) U(\g_\C^-)$ 
which  shows that every $G^0$-invariant subspace of $\cH^0$ 
generates an invariant subspace of $\cH$ under $U(\g_\C^+)$ 
(alternatively, we can use Theorem~\ref{thm:c.1}(ii)).  

As $\pi$ is irreducible, $\cH^0$ generates $\cH$ under the action of $G$. 
That $(\pi, \cH)$ is holomorphically induced from  $(\pi^0,\cH^0)$ 
follows from Theorem~\ref{thm:c.3} and that it is a strict 
ground state representation from Corollary~\ref{cor:6.2}. 
\end{proof}

\begin{rem} The representation $\pi^0$ need not be one-dimensional 
if $\ft \not=\g^0$, i.e., if $\g^0$ is not abelian. 
A simple example is obtained for 
$G = \U_3(\C)$, the identical representation $(\pi,\cH)$ on $\cH = \C^3$, and 
\[ \bd = i \diag(1,0,0).\] 
Then $\cH^0 = \C e_2 + \C e_3$ is $2$-dimensional and 
\[ G^0 \cong \U(\C e_1) \times \U(\C e_2 + \C e_3) 
\cong \T \times \U_2(\C).\] 
\end{rem}

\begin{thm} \label{thm:compcase}
If $G$ is a compact connected  Lie group 
and $\alpha \: \R \to \Aut(G)$ is a continuous homomorphism, 
then the following assertions hold: 
\begin{itemize}
\item[\rm(i)] Every unitary representation of $G$ is a ground state 
representation. 
\item[\rm(ii)] A unitary representation $(\pi^0, \cH^0)$ of $G^0 = \Fix(\alpha)$ 
extends to a ground state representation of~$G$ if and only if it satisfies 
the positivity condition $C_\alpha \subeq C_{\pi^0}$.
\item[\rm(iii)] $(G,\alpha)$ has the unique extension property, 
i.e., every ground state representation of $(G,\alpha)$ is strict. 
\end{itemize}
\end{thm}

\begin{prf} (i) As every unitary representation of $G$ is a direct 
sum of irreducible ones, Lemma~\ref{lem:dirsum} shows that it suffices 
to assume that $(\pi, \cH)$ is irreducible. 
Then the assertion follows from Proposition~\ref{2ndstep}. 

\nin (ii) The necessity of  $C_\alpha \subeq C_{\pi^0}$ follows from 
Theorem~\ref{thm:2.18}. 
To see that it is also sufficient, write 
$(\pi^0,\cH^0)$ as a direct sum of irreducible representations 
$(\pi_j^0,\cH_j^0)_{j \in J}$. 
By Proposition~\ref{prop:likeso3}, the representations 
$(\pi_j^0,\cH_j^0)$ are holomorphically inducible to unitary 
representations $(\pi_j, \cH_j)$ and we can form their 
direct sum $(\pi, \cH)$, which naturally contains $(\pi^0, \cH^0)$ as a 
subrepresentation. Now the subspace $\cF \subeq \cH$ 
generated by $\pi(G) \cH^0$ carries a ground state representation 
extending~$\pi^0$.

\nin (iii) We have to show that all ground state 
representations $(\pi, \cH)$ of $G$ are strict 
(cf.\ Proposition~\ref{prop:strictvsuniqueext}). 
As $\cH^0$ is $\pi(G)'$-invariant, it decomposes according to the 
decomposition $\pi \cong \oplus_{[\rho]\in \hat  G} \pi_{[\rho]}$ into 
isotypic $G$-representations. We have already seen that the 
passage from $\pi$ to $\pi^0$ defines for irreducible 
representations an injection $\hat G \into \hat{G^0}$, 
whose image has been characterized in Proposition~\ref{prop:likeso3}. 
Hence there are no non-zero $G^0$-intertwining operators between 
different representations $\pi_{[\rho]}^0$. 
This reduces the problem to the case where 
$\pi$ is isotypic, where it follows from the fact that 
$\rho$ is holomorphically induced from $\rho^0$ 
(Proposition~\ref{2ndstep}),  and now strictness follows from 
Theorem~\ref{thm:c.1}(ii). 
\end{prf}

\begin{rem} (Classification schemes) 

\nin (a) We think of the preceding theorem as a classification scheme 
for the irreducible representations of $G$. 
To recover the classical approach, let 
$\bd \in \g$ be a regular element, i.e., 
$\ft := \g^0 := \ker(\ad \bd)$ is abelian. 
Then $T := G^0$ is a maximal torus of $G$, and the preceding 
theorem asserts that the irreducible unitary representations 
of $G$ can be parametrized in terms of those 
irreducible unitary representations of $T$ arising as ground 
state representations for $\alpha_t(g) = \exp(t\bd) g \exp(-t\bd)$. 
Since $T$ is abelian, its irreducible 
representations are characters. So Theorem~\ref{thm:compcase} 
yields an injection 
\[ \hat G \into \hat T \] 
whose range is the subset 
\[ \hat T_\bd := \{ \chi \in \hat T \: -i \cdot\dd \chi \in C_\alpha^\star\}.\]  

As $\bd$ is regular, 
\[ \Delta^{+} := \{ \alpha \in \Delta(\g_\C, \ft_\C) \:  
-i\alpha(\bd) > 0 \} \] 
is a positive system of roots. Proposition~\ref{prop:likeso3}  
then implies that 
\[ \hat T_\bd = \{ \chi\in \hat T \: 
(\forall \alpha \in \Delta^+) \lambda(\alpha^\vee) \leq 0 \} \] 
consists of all antidominant weights. We thus recover 
the Cartan--Weyl Classification of the irreducible $G$-representations  
in terms of lowest weights. 

\nin (b) The key point of the preceding theorem is that it 
does not require $\bd$ to be regular. In any case we obtain an 
injection
\[ \hat G \into \hat{G^0} \] 
and an irreducible representation $\pi^0$ of $G^0$ is contained 
in the range of this map if and only if $C_\alpha \subeq C_{\pi^0}$. 
Note that $C_\alpha \subeq \g^0$ is a closed convex invariant cone in $\g^0$, 
hence determined by the intersection with any Cartan subalgebra 
$\ft \subeq \g^0$ (\cite[Thm.~VII.3.29]{Ne00}). Let $\lambda^0 \in i\ft^*$ be an extremal 
weight of an irreducible representation $\pi^0$ of $G^0$. 
Then all weights of $\pi^0$ are contained in $\conv(\cW^0 \lambda)$. 
Therefore $C_\alpha \subeq C_{\pi^0}$ is equivalent to 
\[ -i\cdot \lambda  \in (C_\alpha \cap \ft)^\star.\] 
Hence the image of $\hat G$ in $\hat{G^0} \subeq \hat T$ consists 
of all characters $\chi$ which are lowest weights of $G^0$-representation 
and, in addition, satisfy 
\[ -i\cdot\dd\chi \in (C_\alpha \cap \ft)^\star.\] 
\end{rem}

\begin{rem}  \label{rem:countex}
The assumptions of compactness and finite dimension in 
Theorem~\ref{thm:compcase}(i),(ii) are fundamental and cannot be removed.  
This is demonstrated by the following examples.
We examine one case, disproving (i), and another one, disproving (ii). 
\begin{enumerate}
\item[(a)] The group $G = \SL_2(\R)$ is a finite dimensional Lie group, hence 
locally compact, but not compact. 
It has irreducible unitary representations $(\pi, \cH)$ 
(the principal series) for which all non-zero elements $\bd\in \g$ 
correspond to unbounded 
hermitian operators $i\partial \pi(\bd)$ which are neither bounded from 
below or above. Therefore (i) of Theorem \ref{thm:compcase} fails for 
$\alpha_t(g) := \exp(t\bd) g \exp(-t\bd)$ and every non-zero $\bd \in \g$. 
The group $G^0 = Z_G(\bd)$ is compact if and only if $\bd$ is elliptic. 
Therefore, even if we require only $G^0$ to be compact, 
instead of the whole group $G$, Theorem \ref{thm:compcase}(i) fails.
\item[(b)] We illustrate another example, where 
Theorem \ref{thm:compcase}(ii) fails. 
We consider the group $G:=\SU_{1,2}(\C)$ and $\alpha_t(g) = \exp(t\bd) g \exp(-t\bd)$ for 
$\bd := i \diag(1,-1,-1)$. Then $\alpha$ has a compact group of 
fixed points 
\[ G^0 \cong  {\rm S}(\U_1(\C) \times \U_2(\C))
= \{ (\det(g)^{-1},g) \in \U_1(\C) \times \U_2(\C) \: g \in \U_2(\C)\} 
\cong  \U_2(\C).\] 
The subspace $\ft$ of diagonal matrices in $\g$ is a compactly embedded 
Cartan subalgebra. For a linear functional on $\ft$, represented by 
\[ \lambda = (\lambda_1,\lambda_2, \lambda_3)  \in \Z^3 \] 
and satisfying $\lambda_2 >\lambda_3$, 
the condition for the existence of a corresponding 
unitary highest weight representation is 
\[ \lambda_3 - \lambda_1 \geq 1\] 
(cf.~\cite[Lemma~I.6]{NO98}). 

However, the condition $C_\alpha \subeq C_{\pi^0}$, only implies that 
$\lambda_3 \geq \lambda_1$. Therefore $\lambda = (0,1,0)$ 
defines an irreducible representation of $G^0$ on $\cH^0 = \C^2$ by 
$\pi^0(g_1, g_2) = g_2$, for which $C_\alpha \subeq C_{\pi^0}$,  
but $\pi^0$ does not extend to a ground state representation. 
Therefore the conclusion of Theorem~\ref{thm:compcase}(ii) fails. 
Similar arguments apply to all hermitian Lie algebras of real rank 
$r \geq 2$. 
\end{enumerate}
\end{rem}

\section{Compact non-Lie groups} 
\label{sec:8} 

In this short subsection we show that the results on compact 
Lie groups in Section~\ref{subsec:6.2} can be extended 
to general compact groups. 


Let $G$ be a compact group. 
Then the group $\Aut(G)$ is a topological group with respect to the 
initial topology defined by the map $\Aut(G) \to C(G,G)_{\rm co}^2, 
\phi \mapsto (\phi, \phi^{-1})$, where $C(G,G)_{co}$ denotes the 
set $C(G,G)$, endowed with the 
compact open topology. The continuity of the evaluation map 
$\ev \:  C(G,G) \times G \to G$ implies that the action of $\Aut(G)$ on $G$ 
is continuous. 

A homomorphism 
$\alpha \: \R \to \Aut(G)$ is continuous if and only if 
it is continuous as a map into $C(G,G)_{co}$, which in turn is equivalent to the 
continuity if the corresponding action map 
\[ \alpha^\wedge \: \R \times G \to G, \quad (t,g) \mapsto \alpha_t(g).\]
Here we use the Exponential Law for locally compact spaces, asserting 
that the map 
\begin{equation}
  \label{eq:expolc}
 C(X, C(Y,Z)_{co})_{co} \to C(X \times Y, Z)_{co}, \quad 
f \mapsto f^\wedge, \qquad f^\wedge(x,y) := f(x)(y) 
\end{equation}
is a homeomorphism (\cite[Prop.~A.5.17]{GN}). 

The Lie algebra $\g$ of $G$ can be identified as a topological space with 
\[ \fL(G) := \Hom(\R,G) \subeq C(\R,G)_{co}.\]

\begin{lem} \label{lem:comp-redux} 
For a compact connected group $G$, the following assertions hold:
  \begin{itemize}
  \item[\rm(i)] The adjoint action 
$\Ad \: \Aut(G) \times \fL(G) \to \fL(G), 
\Ad(\phi,\gamma) := \Ad_\phi\gamma := \phi  \circ \gamma$ is continuous. 
  \item[\rm(ii)] Every continuous $\R$-action $\alpha$ on $G$ 
defines a continuous action on $\fL(G)$ by the automorphisms 
$\Ad^{\alpha_t}\gamma := \alpha_t \circ \gamma$. 
Moreover, there exists a filter basis $(Q_j)_{j \in J}$ of $\alpha$-invariant 
closed subgroups such that $G/Q_j$ is a compact Lie group and 
$G \cong \prolim G/Q_j$.  
  \end{itemize}
\end{lem}

\begin{prf} (i) By the Exponential Law, the continuity of this map 
is equivalent to the continuity of the map 
\[ \Ad^\wedge \: \Aut(G) \times \fL(G) \times \R \to G, \quad 
\Ad^\wedge(\phi,\gamma,t)  := \phi(\gamma(t)).\]
The continuity of this map follows from the continuity of the action of 
$\Aut(G)$ on $G$ and the continuity of the evaluation map 
$\fL(G) \times \R \to G$. 

\nin (ii) follows immediately from (i). 
As a topological Lie algebra, $\g$ is a projective limit of 
finite dimensional Lie algebras $(\g_j)_{j \in J}$ (\cite[Thm.~2.20, 
Lemma~3.20]{HM07}) 
and, accordingly, 
its topological dual $\g' \cong \indlim \g_j'$ is a directed union 
of the finite dimensional subspaces~$\g_j'$. 

We have seen above that any continuous one-parameter group 
$\alpha \: \R \to \Aut(G)$ defines a continuous action 
on $\g \cong \fL(G)$. On its topological dual, we thus obtain a 
one-parameter group $\beta \:  \R \to \GL(\g')$ with continuous 
orbit maps with respect to the weak-$*$ topology. 
As the weak-$*$ topology on $\g'$ is the finest locally 
convex topology (\cite[Thm.~A2.8]{HM07}), 
Proposition~\ref{prop:onepar-dirlim} 
shows that every $\g_j' \subeq \g'$ is contained in a 
finite dimensional $\beta$-invariant subspace
$W_j = \Spann \{ \beta_t(\g_j') \: t \in \R\}$. 
Let $q_j \: \g \to \g_j$ denote the quotient maps. Then 
$\g_j' \cong \ker(q_j)^\bot \subeq \g'$ and thus 
\[ W_j^\bot 
= \bigcap_{t \in \R} \beta_t (\g_j')^\bot 
= \bigcap_{t \in \R} \Ad_{\alpha_t}(\ker q_j) \subeq \ker q_j \] 
is an ideal of finite codimension. 
We further have $\bigcap_j W_j^\bot = \{0\}$. 
For the closed normal subgroups 
$Q_j := \oline{\exp (W_j^\bot)} \trile G$, the quotient 
$G/Q_j$ is a finite dimensional compact Lie group and 
$G \cong \prolim G/Q_j$. 
\end{prf}

By the preceding lemma, we may write a compact connected
Lie group $G$ as $\prolim G/Q_j$, 
where the normal subgroups are $\alpha$-invariant 
und each $G_j = G/Q_j$ is a compact connected Lie group that inherits 
an $\R$-action $\alpha_j$ from $\alpha$. 

Alternatively, this can be derived from the structure theory 
for the topological group $\Aut(G)$ developed in \cite[p.~264]{HM06}.
For a compact group $G$, 
the group $\Aut(G)_0 \cong G'/Z(G')$ is a compact group 
with Lie algebra $\g'$. Hence every continuous 
one-parameter group $\alpha \: \R\to \Aut(G)$ 
is obtained by the conjugation action of a one-parameter group 
$\gamma \: \R \to G'$. In particular, all normal subgroups are 
$\alpha$-invariant. All this follows from the following structure 
theorem: 

\begin{thm}{\rm(\cite[Cor.~9.87]{HM06})}\label{structureofaut}	Let $G$ be a compact connected group  with maximal pro-torus~$T$ and write 
$\Inn(G) \subeq \Aut(G)$ for the subgroup of inner automorphisms. 
Then there is a totally disconnected closed subgroup $D$ of $\Aut(G)$ contained in the normalizer 
\[ N_{\Aut(G)}(T)=\{\alpha \in  \Aut(G) \:  \ \alpha(T)=T\} \] 
of $T$ in  $\Aut(G)$ such that $$\Aut(G)= \Inn(G)  \cdot  D,  \ \ \Inn(G)  
\cap D=\{\id_G\}, \ \ \Aut(G)= \Inn(G)  \rtimes  \mathrm{Out}(G),  \ \  \mathrm{Out}(G)  \simeq  D.  $$
In particular, $\mathrm{Out}(G)$ is totally disconnected, $\Inn(G)\simeq G/Z(G)$ is compact connected semisimple center-free and isomorphic to $G'/Z(G')$.
\end{thm}

\begin{thm} \label{thm:8.13} Let $G$ be a connected compact group and 
$\alpha \: \R \to \Aut(G)$ be a  continuous one-parameter group. 
Then every continuous unitary representation 
of $G$ is a strict ground state representation for $(G,\alpha)$. 
\end{thm}

\begin{prf} 
We have $G \cong \prolim G_j$ for compact connected Lie groups~$G_j$ 
and quotient maps $q_j \: G \to G_j$. 
By Lemma~\ref{lem:comp-redux}, we may 
assume that the kernels of the quotient maps 
$q_j$ are $\alpha$-invariant, so that 
we obtain $\R$-actions $\alpha^j$ 
on the connected compact Lie groups~$G_j$. 

By \cite[Thm.~12.2]{Ne10}, every continuous unitary representation 
$(\pi, \cH)$ of $G$ is a direct sum of subrepresentations 
$(\pi_k, \cH_k)$ with $\ker q_{j_k} \subeq \ker \pi_k$ for some~$j_k$. 
Lemma~\ref{lem:dirsum} now shows that 
a unitary representation $(\pi,\cH)$ of $G$ is a ground state 
representation if and only if this holds for all 
representations of Lie quotient groups $G_j$, 
and for these the assertion follows from Theorem~\ref{thm:compcase}. 

Since any finite dimensional representation of $G$ factors 
through a representation of a Lie quotient group, 
the same argument as in the proof of  Theorem~\ref{thm:compcase} 
shows that the injectivity of the 
map $\hat G \to \hat{G^0}, [\pi] \mapsto [\pi^0]$ 
for compact connected Lie groups implies the same for 
$G$. Now we can argue as in the proof of 
Theorem~\ref{thm:compcase}(iii) to see that 
$(\pi,\cH)$ is a strict ground state representation.
\end{prf}

\section{Ground state representations of direct limits} 
\label{sec:7} 

After some general observation about direct limits, we discuss in 
this section some examples that show how 
the results on strict ground state representations 
on compact connected Lie groups (Theorem~\ref{thm:compcase}) 
can be extended to direct limits  of such groups. 

\subsection{Ground state representations} 

Here we use the unique extension property of 
compact connected Lie groups $G$ 
to extend Theorem~\ref{thm:compcase} to direct limit groups 
$G = \indlim G_n$. By Gl\"ockner's Theorem 
\cite{Gl05}, countable direct limits of directed systems 
of connected finite dimensional Lie groups always exist in the category 
of locally convex Lie groups, and the so obtained Lie group  
is also the direct limit in the category of topological groups. 
Under additional assumptions on the $G_n$, 
this provides in particular a classification of ground state representations 
in term of the corresponding representations $(\pi^0, \cH^0)$ of the 
subgroup~$G^0$. For abelian $G^0$ this leads in particular to direct integrals 
of lowest weight representations, but many other classes of representations 
appear naturally. \\

\nin {\bf Assumption:} In the following we assume that each subgroup 
$G_n\subeq G$ is $\alpha$-invariant and put 
\[ \alpha_n(t) := \alpha(t)\res_{G_n}\quad \mbox{ for } \quad t \in \R.\]
By Proposition~\ref{prop:onepar-dirlim}, the groups $G_n$ can always 
be chosen such that this is the case.
We also assume that the infinitesimal generators 
$D_n \in \der(\g_n)$ of the corresponding one-parameter subgroups 
of $\Aut(\g_n)$ are semisimple, so that 
$\g_n = D_n(\g_n) \oplus \ker(D_n)$ and (L1)-(L4) in Section~\ref{sec:3} 
are satisfied. In particular, the cone $C_{\alpha_n}$ in \eqref{eq:cd} is defined.

\begin{lem} \label{lem:dirlim-uniqueext}
If all pairs $(G_n, \alpha_n)_{n \in \N}$ have the unique extension 
property for ground state representations, then so does 
$(G,\alpha)$. 
\end{lem}

\begin{prf} Let $(\pi_j, \cH_j)_{j = 1,2}$ be two ground state 
representations of $G$ and 
$\phi \: \cH_1^0 \to \cH_2^0$ be a unitary $G^0$-equivalence. 
We write $\cH^n_j := \lbr \pi_j(G_n) \cH_j^0\rbr]$ for $j =1,2$. 
As each $(G_n, \alpha_n)$ has the unique extension property 
and the representation of $G_n$ on $\cH_j^n$ is ground state, 
there exist uniquely determined unitary $G_n$-equivalences 
$\Phi_n \: \cH_1^{n} \to \cH_2^{n}$ extending~$\phi$. 
For $n < m$ uniqueness implies that $\Phi_m\res_{\cH^n_1} = \Phi_n$. 
Therefore the $\Phi_n$ combine to a unitary $G$-equivalence 
$\Phi \: \cH_1 \to \cH_2$. 
\end{prf}

\begin{thm} \label{thm:5.3}
Assume that each pair $(G_n, \alpha_n)$  has the unique 
extension property and that the condition 
$C_{\alpha_n} \subeq C_{\pi^0}$ (see \eqref{eq:cd}) 
is sufficient for the extendability 
of a unitary representation $(\pi^0, \cH^0)$ of $G_n^0$ to~$G_n$.
Then every unitary representation 
$(\pi^0, \cH^0)$ of 
$G^0 = \indlim G_n^0$ satisfying $C_\alpha \subeq C_{\pi^0}$ 
extends uniquely to a strict  ground state representation of~$G$.
\end{thm}

\begin{prf} Each restriction $(\pi_n^0, \cH^0)$ is a unitary 
representation of $G_n^0$ satisfying $C_{\alpha_n} \subeq C_{\pi_n^0}$, 
hence extends uniquely to a unitary representation 
$(\pi_n, \cH_n)$ of $G_n$. 

If $n < m$, then $\pi_n^0 = \pi_m^0\res_{G_n^0}$, so that 
the unique extension property implies that the 
$G_n$-representation on the subspace 
$\lbr \pi_m(G_n) \cH^0\rbr \subeq \cH_m$ is $G_n$-equivalent to 
$(\pi_n, \cH_n)$. We thus obtain a natural isometric 
$G_n$-equivariant inclusion 
\[ \phi_{mn} \: \cH_n \to \cH_m \quad \mbox{ with } \quad 
\phi_{mn}\res_{\cH^0} = \id_{\cH^0}.\] 
Here we identify $\cH^0$ with a subspace of every space~$\cH_n$. 
These embeddings also intertwine the minimal unitary 
one-parameter groups $U_{n,t}^0 := e^{itH_n}$ implementing $\alpha$ 
on $\cH_n$ and satisfying $\cH_n^0 = \cH^0 = \ker H_n$. 

For $n < m < k$, we then have 
$\phi_{km} \circ \phi_{mn} = \phi_{kn}$, so that we obtain a unitary direct limit 
representation $(\pi, \cH) = \indlim (\pi_n, \cH_n)$ 
of the direct limit group $G = \indlim G_n$. Since the restriction 
of this presentation is continuous on every $G_n$ and 
$G$ carries the direct limit topology, $\pi$ is continuous. 
By construction, it is a ground state representation of 
$(G,\alpha)$. Lemma~\ref{lem:dirlim-uniqueext} implies that 
$(G,\alpha)$ has the unique extension property, 
so that the strictness of $(\pi, \cH)$ follows  from 
Proposition~\ref{prop:strictvsuniqueext}. 
\end{prf}

\subsection{Some infinite dimensional unitary groups}

We consider the Lie group 
$G := \U_\infty(\C) = \indlim \U_n(\C)$ and 
$\alpha_t \in \Aut(G)$ determined by 
\[ \alpha_t(g) = e^{t \bd} g e^{-t\bd}, \qquad 
\bd = \diag( i \cdot d_n)_{n \in \N}.\]
Then $G^0 \subeq G$ is the subgroup preserving all eigenspaces of 
the diagonal operator $\bd$ on $\C^{(\N)}$. 

We assume that the $(d_n)_{n \in \N}$ are mutually different, so that 
\begin{equation}
  \label{eq:inftorus}
 T := G^0 \cong \T^{(\N)} 
\end{equation}
is the subgroup of diagonal matrices in $G$. 
So $G$ is a direct limit of the compact subgroups $G_n \cong \U_n(\C)$ 
and the abelian group $G^0$ is the direct limit of the tori 
$T_n := T \cap \U_n(\C) \cong \T^n$. Accordingly, 
the character group 
\[ \hat{G^0} 
\cong \prolim \hat{\T^n} 
\cong \prolim \Z^n \cong \Z^\N \] 
carries a natural totally disconnected, metrizable group topology. 
Taking the differential in $e$, we identify the character group $\hat{G^0}$ by 
with a subgroup of the dual space $(\g^0)' \cong \R^\N$. 

The cone $C_\alpha \subeq \g^0$ is easy to determine. 
Let $D := \ad \bd$. Then $C_\alpha$ is generated by the elements 
$-i [x_\lambda^*, x_\lambda]$ for $\lambda > 0$ and 
$-i Dx_\lambda = -i[\bd, x_\lambda] = \lambda x_\lambda$ 
(cf.~Remark~\ref{rem:4.7}). 
All matrices $E_{nm} \in \gl_\infty(\C)$ are eigenvectors of $D$ with 
\[ -i D E_{nm} = (d_n - d_m) E_{nm}.\] 
For $d_n > d_m$ we thus obtain the generator 
\[ -i [E_{nm}^*, E_{nm}]  =  -i [E_{mn}, E_{nm}]  = -i(E_{mm} - E_{nn}),\] so that 
\[ C_\alpha = i \cone\{ E_{nn} - E_{mm} \: d_n > d_m\}.\]
Hence $C_\alpha \subeq C_{\pi^0}$ is equivalent to 
\begin{equation}
  \label{eq:seccond}
\partial \pi^0(E_{nn} - E_{mm}) \geq 0 \quad \mbox{ for } \quad d_n > d_m.
\end{equation}

Since $G^0$ is abelian, all irreducible representations are one-dimensional. 
A character $\chi_\lambda \in \hat{G^0}$ 
with $\chi_\lambda(\exp x) = e^{2\pi \lambda(x)}$, $\lambda \in \R^\N$, 
satisfies the positivity condition \eqref{eq:seccond} if and only if 
\begin{equation}
  \label{eq:thirdcond}
\lambda_n - \lambda_m  \geq 0 \quad \mbox{ for } \quad d_n > d_m.
\end{equation}
We write $\hat T(\bd) \subeq \hat T\cong \Z^\N$ 
for the closed subgroup of all characters 
satisfying this condition. 

By Theorems~\ref{thm:5.3} and~\ref{thm:2.18},  
a representations $(\pi^0, \cH^0)$ of $G^0$ is extendable to a 
ground state representations of $G$ if and only if 
$C_\alpha  \subeq C_{\pi^0}$.   
In view of \cite[Prop.~7.9]{Ba91}, the abelian group $T = G^0$ is nuclear 
because it is Hausdorff and a countable direct limit of 
compact abelian groups (which are nuclear). 
Therefore the Bochner Theorem for nuclear groups 
(\cite[Ch.~4]{Ba91} implies that 
its unitary representations are given by Borel spectral measures on 
the  character group $\hat T$, endowed with the topology of 
pointwise convergence, which is the product Borel structure on $\hat T 
\cong \Z^\N$. Therefore $\pi^0$ extends to a ground state representation 
if and only if its spectral measure is supported by the closed subset 
$\hat T(\bd)$. In particular, 
general ground state representations for $(G,\alpha)$ 
are direct integrals of unitary highest weight representations 
$(\pi_\lambda, \cH_\lambda)$ with $\lambda$ satisfying  \eqref{eq:thirdcond}. 

\begin{rem}
The classification results for unitary highest weight 
representations in \cite{MN16} also fit into this context. 
In that paper one finds a description of all pairs $(\lambda,\bd)$, 
where $\lambda \in i\ft^*\cong \R^\N$ and $D = \ad \bd$, for which 
the unitary highest weight representation $L(\lambda)$ of 
$\gl_\infty(\C)$ with highest weight $\lambda$ carries a positive energy 
representations for $(G,\alpha)$. This amounts to the condition that 
\[ \inf \la \cW\lambda - \lambda, -i\bd \ra  > -\infty \] 
holds for the corresponding Weyl group~$\cW$ 
(the group of all finite permutations in the present example). 
This condition is equivalent to the unitary highest weight representation 
$(\pi_\lambda, \cH_\lambda)$ to be a ground state representation for 
$(G,\alpha)$, where the minimal eigenvalue of 
$-i\bd$ is $\lambda(i\bd)$. 
\end{rem}

\begin{rem} The are also weight representations of $G = \U_\infty(\C)$ 
which have no extremal weight, but which are ground state representations. 
We refer to \cite{DP99} and \cite{DMP00} for classification results for 
weight modules $V$ of $\su_\infty(\C)_\C \cong \fsl_\infty(\C)$. 
As we have seen in \cite[Ex.~V.9]{Ne04}, some weight representations 
define bounded representations with no extremal weights. In particular, there 
are weight modules $V$ whose weight set $\cP_V$ consists of all 
functionals of the form 
\[ -i \alpha = \chi_N - \chi_M, \quad \mbox{ with } \quad 
N, M \subeq \N, \ N \cap M = \eset, \ |N| = |M| < \infty.\] 
As $\cP_V = - \cP_V$, the operator defined by $-i\bd \in \R^\N$ on $V$ 
is semibounded if and only if it is bounded, and this happens if and only if 
(up to an additive constant) $\bd$ has 
finite support. Write 
$\bd = \bd_+ - \bd_-$, where $\bd_\pm$ are non-negative with finite 
support. Then 
\[-i\alpha(\bd) 
= \chi_N(\bd) - \chi_M(\bd) 
= \chi_N(\bd_+) - \chi_N(\bd_-) 
- \chi_M(\bd_+) + \chi_M(\bd_-)  \] 
is minimal if 
$\supp(\bd_+) \subeq M$ and 
$\supp(\bd_-) \subeq N$, which leads to the minimal value 
\[-i\alpha(\bd) 
=  - \chi_N(\bd_-) - \chi_M(\bd_+)  
= - \sum_{n \in \N}  |d_n|.\] 
In particular, the representation on $V$ is a 
ground state representations for $G^0 = T$. 
\end{rem}

\begin{rem} The situation does not change significantly 
if we replace the group $\U_\infty(\C)$ 
by some Banach completion, such as 
$\U_1(\cH)$ (completion in the trace norm), or the group 
$\U_2(\cH)$ (completion in the Hilbert--Schmidt norm). 
The continuous unitary representations of these groups 
are simply those continuous unitary representations of $G$ 
which extend to these completions, so that we are 
dealing with less representations for the larger groups. 
\end{rem}

\begin{ex} Similar techniques apply to the direct limit 
$G = \indlim G_n$ of the 
groups $G_n := \U_{2^n}(\C)$ with their natural embedding, given by 
the connecting maps $g \mapsto \pmat{g & 0 \\ 0 & g}$. 
Then the abelian group $T := \indlim T_n$ can be identified with a 
subgroup of the mapping group 
$C(\{0,1\}^\N,\T)$ of $\T$-valued functions on the Cantor set 
$\{ 0,1\}^\N$. Here 
$T_n$ corresponds to the subgroup $C(\{0,1\}^n,\T) \cong \T^{2^n}$. 
Let $f \: \{0,1\}^\N \to \R$ be an injective function, such as 
\[ f((a_n)) := \sum_{n = 1}^\infty a_n 3^{-n}.\] 
This function can be used to define an automorphism group 
$(\alpha_t)_{t \in\R}$ of $G$ with $G^0 = T$. 

The embedding of $G$ into the $C^*$-algebra 
$\otimes_{n \in \N} M_2(\C)$ leads to bounded representations of $G$, 
and all these are positive energy representations. Many of them 
have no ground states. 
\end{ex}

\begin{ex}
The restricted direct product groups 
\[ G = \U_2(\C)^{(\N)} = \{ (g_n)_{n \in \N} \in \U_2(\C)^\N\: 
|\{ n \in \N \: g_n \not=\1\}| < \infty \} \] 
are direct limits of the 
compact groups $G_n = \U_2(\C)^n$ are also natural examples. 
Up to equivalence,  any one-parameter automorphism group $\alpha$ of $G$ 
is acting by $\alpha_t((g_n)) = (\exp(t\bd_n) g_n \exp(-t\bd_n))$, 
where $\bd_n = i \diag(x_n, -x_n)\in \su_2(\C)$ are diagonal matrices 
($x_n \in \R$). If all $\bd_n$ are non-zero, then 
\[ T := G^0 \cong (\T^2)^{(\N)}\] 
is a nuclear abelian group (the subgroup of diagonal matrices). 
Its irreducible unitary representations are given by its character 
group 
\[ \hat T \cong (\Z^2)^\N \] 
of sequences of pairs of integers $((k_n,m_n))_{n \in \N}$. 
If all $x_n$ are positive, we find that 
\[ \hat T(\alpha) = \{ ((k_n,m_n))_{n \in \N} \: 
(\forall n \in \N)\ k_n \geq m_n\}.\] 
Let $(\pi_n, \cH_n)$ denote the irreducible representation 
of $\U_2(\C)$ with highest weight $(k_n,m_n)$ 
and ${v_n\in\cH_n}$ be a unit vector of lowest weight, which is a 
ground state vector for $\bd_n$. Then the infinite tensor product 
\[ \bigotimes_{n \in \N} (\cH_n, v_n) \] 
carries an irreducible unitary representaton of $G$ which 
a corresponding ground state representation. 
In the same way as for the group $U_\infty(\C)$, 
we see that the ground state representations of~$G$ correspond 
to those representations of $T$ whose spectral measure is supported 
on the closed subset~$\hat T(\alpha)$. 
\end{ex}

\appendix 

\section{Arveson spectral theory} 
\label{app:arv} 

In this appendix we collect the concepts relating to 
spectral subspaces for the action of a one-parameter group 
on a complete locally convex space. We follow \cite{Ar74} 
with some  generalizations. We use these concepts for 
one-paramter groups of automorphisms of infinite dimensional 
Lie algebras. 

\subsection{Arveson spectral subspaces} 

\begin{defn}\label{equicontdef}
Let $E$ and $F$ be a locally convex spaces. We denote by $\Hom(E,F)$ the space of continuous linear maps from $E$ to $F$ and write $\End(E) := \Hom(E,E)$. 
A subset $Y \subset \Hom(E,F)$ is called \textit{equicontinuous} if for every open $0$-neighborhood $U$ in $F$ there exists a $0$-neighborhood $W$ in $E$ such 
that $T(W) \subset U$ holds for every $T\in Y$.
\end{defn}

\begin{defn} \label{def:spectralsubspaces} 
(cf.\ \cite{NSZ15}, \cite{Ne13}) 
Let $V$ be a complete complex locally convex space and let $\alpha:\R\rightarrow \GL(V)$, $t\mapsto \alpha_t$ be a strongly continuous representation.

\nin (a) Assume that $\alpha$ is polynomially bounded 
(Definition \ref{def:equicont}), i.e., 
for every continuous seminorm $p$, there exists a $0$-neighborhood $U\subeq V$ 
and $N \in \N$ such that 
\[ \sup_{x \in U} \sup_{t \in \R} \frac{p(\alpha_t(x))}{1 + |t|^N} < \infty.\]
We define
\begin{equation}\label{intAfS}
\alpha_f(v) := \int_\R f(t) \alpha_t(v)\, dt 
\quad \mbox{ for }\quad v \in V, f \in \cS(\R).
\end{equation}
Then $\alpha_f\in \End(V)$ and this yields a representation of the 
convolution algebra $(\cS(\R),*)$ on $V$. We define the \emph{spectrum} of an element $v\in V$
by 
\begin{equation*}
\Spec_\alpha(v;\cS) := \{ y \in \R \: 
(\forall f \in \cS(\R))\, \alpha_f v = 0 \Rarrow \hat f(y) = 0\}
\end{equation*}
which is the hull of the annihilator ideal of $v$.
Here we use the following version of the Fourier transform: 
\begin{equation}
  \label{eq:ft}
\hat f(y) := \int_\R e^{ixy} f(x)\, dx. 
\end{equation}

For a {\bf closed} subset $E \subeq \R$, we now define 
the corresponding {\it Arveson spectral subspace} 
\[  V(E;\cS) := \{ v \in V\: \Spec_\alpha(v;\cS) \subeq E \}.\] 
We define the {\it spectrum of $(\alpha, V)$} by 
\[ \Spec_\alpha(V) := \{  y \in \R \: 
(\forall f \in \cS(\R))\, \alpha_f = 0 \Rarrow \hat f(y) = 0\}.\]
We also put 
\begin{equation}
  \label{eq:vplus}
 V^+ = \oline{\bigcup_{\delta > 0} V([\delta,\infty);\cS)} 
\quad \mbox{ and }\quad 
 V^- = \oline{\bigcup_{\delta > 0} V((-\infty, -\delta];\cS)}. 
\end{equation}
We say that the {\it splitting condition} is satisfied if 
these subspace and the subspace $V^0 := V(\{0\})$ of fixed points 
(cf.\ Lemma~\ref{lem:a.10} below) satisfy 
\begin{equation}
  \label{eq:trideco}
V = V^+ \oplus V^0 \oplus V^-. \tag{SC}
\end{equation}

\nin (b) If $\alpha$ is equicontinuous, then \eqref{intAfS} exists for all $f \in L^1(\R)$ and we can define $\Spec_{\alpha}(v)$ and $V(E)$ as above with by $\cS(\R)$ replaced by $L^1(\R)$, see \cite[Def.~A.5(b)]{Ne13}. 
This was Arveson's original context. 
\end{defn}

\begin{ex} \label{ex:a.3} 
Suppose that $(\pi,\cH)$ is a continuous unitary 
representation of a finite dimensional Lie group 
$G$ and $\bd \in \g$ is such that $\Spec(\ad \bd) \subeq i \R$. 
We claim that, on the Fr\'echet space $\cH^\infty$ of smooth vectors, 
the representation of $\R$, defined by the unitary  
one-parameter group  $\pi_\bd(t) := \pi(\exp t \bd)$ is 
polynomially bounded (cf.\ \cite{NSZ15}). The topology on $\cH^\infty$ 
is defined by the seminorms 
\[ p_D(v) := \|\dd\pi(D)v\|, \qquad D \in \cU(\g).\] 
Therefore 
\[ p_D(\pi_\bd(t)v) 
= \| \dd\pi(D) \pi(\exp t\bd) v\|= \| \dd\pi(e^{-t \ad \bd} D) v\|,\] 
and this expression has polynomial  estimates because 
$D \in \cU(\g)$ is contained in a finite dimensional 
$\ad \bd$--invariant subspace $F$ on which the one-parameter 
group $e^{i\ad \bd}$ is of polynomial growth. 

As a consequence, the operators 
$\pi_\bd(f) = \int_\R f(t)\, \pi_\bd(t)\, dt, f \in \cS(\R),$ 
leave the subspace $\cH^\infty$ invariant und 
the spectral subspaces $\cH^\infty(E;\cS)$ are defined 
for every closed subset $E \subeq \R$ in the sense 
of Definition~\ref{def:spectralsubspaces}. 
\end{ex}

\begin{lem}
  \label{lem:a.17} {\rm(\cite[Lemma~A.16]{Ne13})} 
If $V$ is a Banach space and $D := \alpha'(0)$ is a bounded operator, i.e., 
$\alpha \: \R \to \Aut(V)$ is norm continuous, then there exists a 
$\delta > 0$ such that the splitting 
condition \eqref{eq:trideco}  is satisfied with 
\[ V^+ = V([\delta, \infty)) \quad \mbox{ and } \quad 
V^- = V((-\infty, -\delta]) \] 
if and only if $0$ is isolated in the spectrum of~$D$. 
\end{lem} 

\begin{ex} It is easy to see examples where the splitting condition 
(SC) fails. A very typical one is the Banach space 
\[ V := C([-1,1],\C) \quad \mbox{ with } \quad 
(\alpha_t h)(x) = e^{itx}h(x).\] 
Then $\alpha_f(h)(x) = \hat f(x)h(x)$ shows that 
\[ \Spec_\alpha(h) = \supp(h) \quad \mbox{ for } \quad h \in V.\] 
This leads to 
\[ V^+ = \{ h \in V \: h\res_{[-1,0]}= 0\}, \quad 
  V^- = \{ h \in V \: h\res_{[0,1]} = 0\} \quad 
\mbox{ and } \quad 
V^0 = \{0\}.\] 
In particular all functions in $V^+ + V^0 + V^-$ vanish in $0$, 
so that this is a proper subspace of $V$. 
\end{ex}

\begin{lem} \label{lem:a.10} {\rm(\cite[p.~226]{Ar74})} For 
$\lambda \in \R$, we have 
\[ V(\{\lambda\}) = V_\lambda(\alpha) := 
\{ v \in V \: (\forall t \in \R)\, \alpha_t(v) = e^{it\lambda}v \}.\] 
\end{lem}

\begin{rem} \label{rem:a.6}
(a) In \cite[p.~225]{Ar74} it is shown that (in the equicontinuous case) 
\[ V(E) = \{ v \in V \: (\forall f \in L^1(G))\ 
\supp(\hat f) \cap E = \eset \Rarrow  \alpha_f(v) = 0\}, \] 
which implies in particular that $V(E)$ is a closed subspace, which is 
clearly $\alpha$-invariant. 
Note that the condition $\supp(\hat f) \cap \oline E = \eset$ 
means that $\hat f$ vanishes on a neighborhood of~$\oline E$. 

\nin (b) If $(E_j)_{j \in J}$ is a family of closed subsets of $\R$, 
then $V\big(\bigcap_{j\in J} E_j\big) = \bigcap_{j \in J} V(E_j)$ 
follows immediately from the definition. 

\nin (c) Lemma~\ref{lem:a.10} implies in particular that 
\[ V(\eset) = \{0\} \] 
because this subspace is contained in every eigenspace $V(\{\lambda\})$,
$\lambda \in \R$. 
\end{rem}

The following proposition is an important 
technical tool. 

\begin{prop}\label{prop:spec-add} {\rm(\cite[Prop.~A.14]{Ne13})} 
Assume that $(\alpha_j, V_j)$, $j =1,2,3$ are 
continuous equicontinuous representations of $\R$
 on the complete locally convex complex spaces $V_j$ and that 
$\beta \: V_1 \times V_2 \to V_3$ 
is a continuous equivariant bilinear map. 
Then we have for closed subsets $E_1, E_2 \subeq \R$ the relation 
\[ \beta(V_1(E_1) \times V_2(E_2)) \subeq V_3(\oline{E_1 + E_2}). \] 
\end{prop}

\begin{ex} \label{ex:liealg} Let $\g$ be a complete locally convex Lie algebra 
and let $x \in \g$ be such that $\ad x$ generates a continuous 
equicontinuous one-parameter 
group $\alpha \: \R \to \Aut(\g)$ of automorphisms, i.e., 
$\alpha$ is strongly differentiable with $\alpha'(0) = \ad x$. 
Applying Definition~\ref{def:spectralsubspaces}  to the $\R$-representation 
on~$\g$ defined by $\alpha$, 
we obtain for each closed subset $E \subeq \R$ a spectral 
subspace $\g_\C(E)$. 
\end{ex}

\begin{lem} \label{lem:c.1} {\rm(\cite[Lemma~4.3]{Ne13})} 
Let $U_t := e^{itA}$ be a strongly continuous 
unitary one-parameter group with infinitesimal generator 
$A = A^*$. Then the following assertions hold: 
\begin{itemize}
\item[\rm(i)]  For each $f \in L^1(\R)$, we have 
$U(f) = \hat f(A)$. 
\item[\rm(ii)]  Let $P \: \fB(\R) \to B(\cH)$ be the unique 
spectral measure with $A = P(\id_\R)$. 
Then, for every closed subset $E \subeq \R$, the 
range space $P(E)\cH$ coincides with the Arveson spectral subspace 
$\cH(E)$. 
\end{itemize}
\end{lem}

\subsection{Applications to unitary representations}


\begin{prop} \label{prop:c.3} {\rm(\cite[Prop.~4.4]{Ne13})} 
Let $(\pi, \cH)$ be a smooth unitary representation 
of the Banach--Lie group $G$, $\bd \in \g$ be elliptic, 
and $P \: \fB(\R) \to \cL(\cH)$ be the spectral measure 
of the unitary one-parameter group 
$\pi_\bd(t) := \pi(\exp_G t\bd)$. 
Then the following assertions hold: 
\begin{itemize}
\item[\rm(i)] $\cH^\infty$ carries a Fr\'echet structure for which 
$\pi_\bd$ defines a continuous equicontinuous action of 
$\R$ on $\cH^\infty$. In particular, $\cH^\infty$ is invariant under 
$\pi_\bd(f)$ for every $f \in L^1(\R)$. 
\item[\rm(ii)] For every closed subset $E \subeq \R$, we have 
$\cH^\infty(E) = (P(E)\cH) \cap \cH^\infty$ for the corresponding 
spectral subspace. 
\item[\rm(iii)]  For every open subset $E \subeq \R$, 
$(P(E) \cH) \cap \cH^\infty$ is dense in $P(E)\cH^\infty$. 
More precisely, there exists a sequence $(f_n)_{n \in \N}$ 
in $L^1(\R)$ for which $\pi _\bd(f_n) \to P(E)$ in the 
strong operator topology and $\supp(\hat f_n) \subeq E$, so that 
$\pi_\bd(f_n)v \in \cH^\infty \cap P(E)\cH^\infty$ for every $v \in \cH^\infty$. 
\end{itemize} 
\end{prop}  


\begin{thm}\label{thm:specrel} {\rm(Spectral translation formula; 
\cite[Thm.~3.1]{NSZ15})}  \label{thm:spectrans}
Assume that $\g$ is a complete locally convex Lie algebra, $\alpha:\R\to\Aut(G)$
defines a continuous action of $\R$ on $G$, and that the induced action 
on $\g_\C$ is also continuous. 
Let $\pi^\flat(g,t) = \pi(g) U_t$ be a continuous unitary representation 
of $G\rtimes_\alpha\R$ on $\cH$ and let 
$\cH^\infty$ be the space of smooth vectors 
with respect to $\pi$. 
\begin{itemize}
\item[{\upshape (i)}] If $\alpha$ is equicontinuous, then 
\[ \dd\pi(\g_\C(E)) \cH^\infty(F) \subeq \cH^\infty(\oline{E + F})
\quad 
\mbox{ for } \quad E, F \subeq \R\ \mbox{ closed}.\] 
\item[{\upshape (ii)}] If $\alpha$ is polynomially bounded, 
then 
\[ \dd\pi(\g_\C(E;\cS)) \cH^\infty(F) \subeq \cH^\infty(\oline{E + F})
\quad 
\mbox{ for } \quad E, F \subeq \R \ \mbox{ closed}.\] 
\end{itemize}
\end{thm}

\section{Positive definite kernels} 
\label{app:B}

Let $X$ be a set and $\cE$ a Hilbert space. 
A Hilbert subspace $\cH \subeq \cE^X$ of the linear space of 
$\cE$-valued functions on $X$ is said to have continuous 
point evaluations if the linear maps 
\[ K_x \: \cH \to \cE, \quad f \mapsto f(x) \] 
are continuous. Then the function 
\[ K \: X \times X \to B(\cE), \quad K(x,y) := K_x K_y^* \] 
is called its {\it reproducing kernel}. 
As the kernel $K$ determines the subspace $\cH\subeq \cE^X$ 
and its scalar product uniquely, we  write 
$\cH_K\subeq \cE^X$ for the Hilbert space determined by $K$ 
and $\cH_K^0 \subeq \cH_K$ for the subspace 
spanned by the functions $K_y^* v$, $y \in X$, $v \in \cE$. 
A kernel function $K \: X \times X \to B(\cE)$ is the reproducing kernel 
of some Hilbert space if and only if it is {\it positive definite} in the sense 
that, for any finite collection $(x_1, v_1), \ldots, (x_n,v_n) \in 
X \times \cE$, the matrix 
$(\la v_j, K(x_j,x_k) v_k)_{1 \leq j,k\leq n}$ is positive semidefinite 
(cf.~\cite[Ch.~1]{Ne00}). 

If $X = G$ is a group and the kernel $K \: G \times G \to \cE$ is invariant 
under right translations, then it is of the form 
$K(g,h) = \phi(gh^{-1})$ for a function 
$\phi \: G \to B(\cE)$. Such functions are called 
{\it positive definite} if the kernel $K$ is positive definite.

The following proposition generalizes the well-known GNS construction 
to operator-valued functions on groups (cf.~\cite{NO18}). 

\begin{prop}{\rm(GNS-construction for groups)} \label{prop:gns} 
Let $\cE$ be a Hilbert space. 
  \begin{enumerate}
  \item[\rm(a)] Let $\phi \: G \to B(\cE)$ be a 
positive definite function with $\phi(e) = \1_\cE$. Then 
$(U^\phi_g f)(h) := f(hg)$ defines a unitary representation of $G$ on the 
reproducing kernel Hilbert space $\cH_\phi := \cH_K \subeq \cE^G$ with 
kernel $K(g,h) = \phi(gh^{-1})$ and the range of the isometric inclusion 
$K_e^* \: \cE \to \cH$ is a $G$-cyclic subspace, i.e., $U_G^\phi K_e^*\cE$ 
spans a dense subspace of $\cH$. We then have 
\begin{equation}
  \label{eq:gns-a}
\phi(g) = K_e U^\phi_g K_e^* 
\quad \mbox{ for } \quad g \in G. 
\end{equation}
  \item[\rm(b)]
If, conversely, $(U, \cH)$ is a unitary representation of $G$ and 
$j \:  \cE \to \cH$ is an isometric inclusion,  then 
\[ \phi \: G \to B(\cE), \quad \phi(g) := j^*  U_g j \] 
is a $B(\cE)$-valued positive definite function. 
If, in addition, $j(\cE)$ is cyclic, then 
$(U, \cH)$ is unitarily equivalent to $(U^\phi, \cH_\phi)$. 
  \end{enumerate}
\end{prop}

\section{Bosonic Fock space} 
\label{sec:7.1}

We start with the construction of the von Neumann algebras on the bosonic Fock space. 
For $v_1, \ldots, v_n \in \cH$, we define 
\[  v_1 \cdots v_n := v_1 \vee \cdots \vee v_n := 
\frac{1}{\sqrt{n!}} \sum_{\sigma \in S_n} v_{\sigma(1)} \otimes \cdots \otimes v_{\sigma(n)} \] 
and $v^n := v^{\vee n}$, so that 
\begin{eqnarray} 
  \label{eq:symprod}
\la v_1 \vee \cdots \vee v_n, w_1 \vee \cdots \vee w_n \ra 
&=&   \sum_{\sigma \in S_n} \la v_{\sigma(1)}, w_1 \ra 
\cdots \la v_{\sigma(n)}, w_m \ra. 
\end{eqnarray} 
For every $v \in \cH$, the series 
$\Exp(v) := \sum_{n = 0}^\infty \frac{1}{n!} v^n$ 
defines an element in $\cF_+(\cH)$ and the scalar product of two 
such elements is given by 
\[ \la \Exp(v), \Exp(w) \ra 
= \sum_{n = 0}^\infty \frac{n!}{(n!)^2} \la v, w\ra^n = e^{\la v, w \ra}.\] 
These elements span a dense subspace of $\cF_+(\cH)$, and therefore 
we have for each $x \in \cH$ a unitary operator on $\cF_+(\cH)$ determined by the 
relation 
\begin{equation}
  \label{eq:Ux-ops}
 U_x \Exp(v) = e^{ -\la x, v\ra - \frac{\|x\|^2}{2}} \Exp(v+x) \quad 
\mbox{ for } \quad x,v \in \cH.   
\end{equation}
A direct calculation then shows that 
\begin{equation}
  \label{eq:comm-rel-U}
U_x U_y = e^{-i\Im \la x, y \ra} U_{x+y} \quad \mbox{ for  } \quad 
x, y \in \cH.
\end{equation}
To obtain a unitary representation, we have to replace the 
additive group of $\cH$ by the {\it Heisenberg group}
\[ \Heis(\cH) := \T\times \cH \quad \mbox{ with } \quad 
(z,v)(z',v') := (zz' e^{-i\Im \la v,v' \ra}, v + v').  \] 
For this group, we obtain with \eqref{eq:comm-rel-U} a unitary representation 
\[ U \: \Heis(\cH) \to \U(\cF_+(\cH)) \quad \mbox{ by } \quad U_{(z,v)} := z U_v.\] 
In this physics literature, all this is expressed in terms of the 
so-called {\it Weyl operators} 
\[ W(v) := U_{iv/\sqrt{2}}, \qquad v \in \cH \] 
satisfying the {\it Weyl relations} 
\begin{equation}
  \label{eq:weyl}
  W(v) W(w) = e^{-i \Im \la v,w \ra/2} W(v+w), \qquad v,w \in \cH.
\end{equation}

We also note that the {\it vacuum vector} $\Omega := \Exp(0) \in \cF_+(\cH)$ 
satisfies 
\[ \la \Omega, U_x \Omega \ra 
= \la \Omega, e^{- \frac{\|x\|^2}{2}} \Exp(x)\ra 
= e^{- \frac{\|x\|^2}{2}}.\] 

\begin{rem} \label{rem:4.5x}
 If $(V,\sigma)$ is a symplectic vector space, 
then the corresponding 
{\it Weyl algebra} $C^*(V,\sigma)$ is the universal 
$C^*$-algebra with unitary generators $(W(v))_{v \in V}$, and the relations 
\[ W(v_1)W(v_2) = e^{i\frac{\sigma(v_1, v_2)}{2}} W(v_1 + v_2)\] 
(\cite[Thm.~5.2.8]{BR96}). 

The representations of this $C^*$-algebra are in one-to-one correspondence 
with the unitary representations $(\pi, \cH)$ of $\Heis(V,\sigma)$ satisfying 
$\pi(z,0) = z\1$ for $z \in \T$. An injective representation 
of $\Heis(V,\sigma)$ is obtained on $\ell^2(V)$ by 
\[ W(x) \delta_y = e^{-\frac{i}{2} \sigma(x,y)}\delta_y, \quad x,y  \in V, \quad 
\mbox{ where } \quad \delta_y(z) = \delta_{yz}\] 
is the canonical orthonormal basis in $\ell^2(V)$. 
As the $C^*$-algebra $C^*(V,\sigma)$ is simple 
by \cite[Thm.~5.2.8]{BR96}, all its representations are injective, 
and therefore the corresponding representations $\pi$ of $\Heis(V,\sigma)$ 
are injective as well.
\end{rem}

\section{Spaces with the 
finest locally convex topology} 
\label{app:D} 

Let $V$ be a countably dimensional real vector space, carrying the 
finest locally convex topology. This is the locally convex topology 
for which all seminorms $p \: V \to \R_+$ are continuous. 
So a net $(x_j)_{j \in J}$ converges in $V$ to $x$ if and only if, 
for every seminorm $p$ on $V$, we have $p(x_j -x) \to 0$. 
From any basis of $V$, we obtain an increasing sequence $(V_n)_{n \in \N}$ 
of finite dimensional linear subspaces for which $V = \bigcup_n V_n$, 
and the topology on $V$ is the direct limit topology with respect to the 
subspaces~$V_n$ (\cite[Ex.~B.13.3]{GN}), 
i.e., a subset $O \subeq V$ is open if and only if 
$O \cap V_n$ is an open subset of $V_n$ for every $n \in \N$. 
We refer to the survey paper \cite{GGH10} for a discussion of more general 
final topologies on topological groups. 

\begin{prop} \label{prop:onepar-dirlim} 
Let $V$ be a real vector space, endowed with the finest 
locally convex topology, i.e., all seminorms on $V$ are continuous. 
If $\alpha \: \R \to \GL(V)$ is a homomorphism defining an  
action of $\R$ on $V$ with continuous orbit maps, then 
\begin{itemize}
\item[\rm(i)] all $\alpha$-orbits lie in finite dimensional subspaces, and  
\item[\rm(ii)] there exists a locally finite endomorphism $D$ such that 
$\alpha_t = e^{tD}$ for all $t\in \R$. 
  \begin{footnote}{An endomorphism $D \in \End(V)$ is called {\it locally finite} 
if each $v \in V$ is contained in a finite dimensional $D$-invariant subspace. 
Then $e^{D}v = \sum_{k = 0}^\infty \frac{1}{k!} D^k v$ is defined 
and $(e^{tD})_{t \in \R}$ defines a one-parameter group of $\GL(V)$.}    
  \end{footnote}
\end{itemize}
\end{prop}

\begin{prf} Let $ v\in V$. Then 
$\alpha_{[-1,1]} v \subeq V$ is a compact subset, 
hence contained in a finite dimensional subspace $W$ 
(\cite[Prop.~7.25(iv)]{HM06}). 
For $f \in C^\infty_c(\R)$ with $\supp(f) \subeq [-1,1]$ this implies that 
\[ \alpha(f)v := 
\int_\R f(t) \alpha_t(v)\, dt  = \int_{-1}^1 f(t) \alpha_t(v)\, dt \in W.\] 
For $0 < \eps \leq 1$, let $W_\eps \subeq W$ denote the subspace generated by  
$\alpha(f) v$ for $\supp(f) \subeq [-\eps,\eps]$. 
Then $W_\eps \subeq W_{\eps'}$ for $\eps < \eps'$, and by the finiteness 
of $\dim W$, there exists an $\eps_0 \in (0,1]$ for which 
$W_{\eps_0}$ is minimal. Then $W_\eps = W_{\eps_0}$ for $0 < \eps \leq \eps_0$. 
As $\alpha(\delta_n) v \to v$  
for any sequence $(\delta_n)_{n \in \N}$ in $C^\infty_c(\R,[0,\infty))$ 
with $\supp(\delta_n) \subeq [-\frac{1}{n}, \frac{1}{n}]$ and 
$\int_\R \delta_n = 1$, it follows that 
$v \in W_{\eps_0}$. 
Since $W_{\eps_0}$ consists of smooth vectors for $\alpha$, this implies 
in particular that $v\in V^\infty$. 

We conclude that all orbit maps in $V$ are smooth. 
Therefore $Dw := \frac{d}{dt}\big|_{t = 0} \alpha_t w$ defines an  
element of $\End(V)$. From $\alpha_t(v) \subeq W$ for $|t| \leq 1$, 
it follows that 
$D^k v \in W$ for $k \in \N.$ Therefore 
$U := \Spann\{ D^k v \: k \in \N_0\}$ 
is a finite dimensional $D$-invariant subspace of~$W$. 
Let $D_U := D\res_U$. Then 
\[ \frac{d}{dt} \big(\alpha_{-t} e^{t D_U} v \big)
=  \alpha_{-t} (-D + D_U) e^{t D_U} v =0,\] 
so that $\alpha_t v = e^{t D_U} v$. This shows that $U$ 
is $\alpha$-invariant and that $\alpha_t = e^{tD}$ holds in the 
sense of exponentials of locally finite operators. 
\end{prf}

\end{document}